\documentclass[final,onefignum,onetabnum]{siamart171218}

\usepackage{braket,amsfonts}
\usepackage{amsmath}
\usepackage{amssymb}
\usepackage{graphicx}
\usepackage{placeins}
\usepackage{longtable}
\usepackage[utf8]{inputenc}
\usepackage{babel}
\usepackage[font=small,labelfont=bf]{caption}
\usepackage{array}
\usepackage{longtable}
\usepackage{tabularx}
\usepackage{xcolor}

\usepackage[margin=1.0in,footskip=0.25in]{geometry}

\usepackage[caption=false]{subfig}
\usepackage{pgfplots}

\allowdisplaybreaks

\newsiamthm{claim}{Claim}
\newsiamremark{remark}{Remark}
\newsiamremark{hypothesis}{Hypothesis}
\crefname{hypothesis}{Hypothesis}{Hypotheses}

\usepackage{algorithmic}

\usepackage{graphicx,epstopdf}
\usepackage{soul}
\newcommand{\rparam}{4/3}
\newcommand{\rdual}{4}
\newcommand{\sparam}{2}

\crefname{ALC@unique}{Line}{Lines}

\usepackage{amsopn}

\DeclareMathOperator*{\esssup}{ess\,sup}

\usepackage{xspace}
\usepackage{bold-extra}
\usepackage[most]{tcolorbox}

\colorlet{texcscolor}{blue!50!black}
\colorlet{texemcolor}{red!70!black}
\colorlet{texpreamble}{red!70!black}
\colorlet{codebackground}{black!25!white!25}

\allowdisplaybreaks

\lstdefinestyle{siamlatex}{%
	style=tcblatex,
	texcsstyle=*\color{texcscolor},
	texcsstyle=[2]\color{texemcolor},
	keywordstyle=[2]\color{texemcolor},
	moretexcs={cref,cref,maketitle,mathcal,text,headers,email,url},
}

\tcbset{%
	colframe=black!75!white!75,
	coltitle=white,
	colback=codebackground, 
	colbacklower=white, 
	fonttitle=\bfseries,
	arc=0pt,outer arc=0pt,
	top=1pt,bottom=1pt,left=1mm,right=1mm,middle=1mm,boxsep=1mm,
	leftrule=0.3mm,rightrule=0.3mm,toprule=0.3mm,bottomrule=0.3mm,
	listing options={style=siamlatex}
}

\newtcblisting[use counter=example]{example}[2][]{%
	title={Example~\thetcbcounter: #2},#1}

\newtcbinputlisting[use counter=example]{\examplefile}[3][]{%
	title={Example~\thetcbcounter: #2},listing file={#3},#1}

\DeclareTotalTCBox{\code}{ v O{} }
{ 
	fontupper=\ttfamily\color{black},
	nobeforeafter,
	tcbox raise base,
	colback=codebackground,colframe=white,
	top=0pt,bottom=0pt,left=0mm,right=0mm,
	leftrule=0pt,rightrule=0pt,toprule=0mm,bottomrule=0mm,
	boxsep=0.5mm,
	#2}{#1}

		\makeatletter
		\newcommand*{\bdiv}{%
			\nonscript\mskip-\medmuskip\mkern5mu%
			\mathbin{\operator@font div}\penalty900\mkern5mu%
			\nonscript\mskip-\medmuskip
		}
		\makeatother

\patchcmd\newpage{\vfil}{}{}{}
\begin{tcbverbatimwrite}{tmp_\jobname_header.tex}

\title{Optimal Control of Stationary Doubly Diffusive Flows in Bounded Lipschitz Domains: Analysis and Numerical Approximation\thanks{\funding{This work was supported by the SERB-CRG India (Grant Number : CRG/2021/002569).}}
}
\author{Jai Tushar\thanks{Center for Computation and Technology and Department of Mathematics, Louisiana State University, Baton Rouge, USA 
		(\email{jai.tushar@lsu.edu}).}
	\and Arbaz Khan\thanks{Department of Mathematics, Indian Institute of Technology, Roorkee, India 
		(\email{arbaz@ma.iitr.ac.in}, \email{maniltmohan@ma.iitr.ac.in}).}
	\and Manil T. Mohan\footnotemark[3]}

\headers{ Optimal control of doubly diffusive flows }{J. Tushar, A. Khan, and M. T. Mohan}
\end{tcbverbatimwrite}

\title{Optimal Control of Stationary Doubly Diffusive Flows in Bounded Lipschitz Domains: Analysis and Numerical Approximation\thanks{\funding{This work was supported by the SERB-CRG India (Grant Number : CRG/2021/002569).}}
}
\author{Jai Tushar\thanks{Center for Computation and Technology and Department of Mathematics, Louisiana State University, Baton Rouge, USA
		(\email{jai.tushar@lsu.edu}).}
	\and Arbaz Khan\thanks{Department of Mathematics, Indian Institute of Technology, Roorkee, India
		(\email{arbaz@ma.iitr.ac.in}, \email{maniltmohan@ma.iitr.ac.in}).}
	\and Manil T. Mohan\footnotemark[3]}

\headers{ Optimal control of doubly diffusive flows }{J. Tushar, A. Khan, and M. T. Mohan}

\ifpdf
\hypersetup{pdftitle={Optimal control of doubly diffusive flows} }
\fi

\newcommand{\cl}[1]{\mathcal{#1}}
\newcommand{\bs}[1]{\boldsymbol{#1}}
\newcommand{\norm}[1]{\ensuremath{\left\|#1\right\|}}
\newcommand{\tnorm}[1]{{\left\vert\kern-0.25ex\left\vert\kern-0.25ex\left\vert #1 
		\right\vert\kern-0.25ex\right\vert\kern-0.25ex\right\vert}}
\newcommand{\avg}[1]{\{ \hspace{-0.1cm}\{#1\} \hspace{-0.1cm}\}}
\newcommand{\jump}[1]{\ensuremath{[\![#1]\!]} }

\begin{document}
\maketitle
\date{\today}
\begin{abstract}
		In this work, we study the control constrained distributed optimal control of a stationary doubly diffusive flow model. For the control problem, we use a well-posedness analysis based on minimal assumptions on data and domain. We show the existence of an optimal control with quadratic type cost functional, study the Fr\'echet differentiability properties of the control-to-state map and establish the first-order necessary optimality conditions corresponding to the optimal control problem. Expanding on this we prove the local optimality of a reference control using second-order sufficient optimality condition for the control problem. Then, following the \textit{optimize-then-discretize} approach, we propose a fully non-conforming, locally exactly divergence-free discretization of the control problem based on the lowest order $(\texttt{CR}_0, \mathbb{P}_0)$ finite element spaces for the state and adjoint variables, and explicit and implicit discretizations for the control variable. Optimal \textit{a priori} error estimates are established for the state, adjoint and control variables in suitable norms up to the regularity of the solution. Finally, the optimal control is computed using a primal-dual active set strategy as a semi-smooth Newton method and numerical experiments are performed to validate the theoretical findings, and it's applicability to optimal control of thermal and solutal convection in a porous medium.
\end{abstract}


\section{Introduction}\label{Intro}

In this article, we study the following bilaterally constrained doubly diffusive flow optimal control problem:
\begin{eqnarray}\label{P:CF}
	\min_{\boldsymbol{U} \in \boldsymbol{\mathcal{U}}_{ad}} J(\boldsymbol{u},\boldsymbol{y},\boldsymbol{U}) = \frac{1}{2} \norm{\boldsymbol{u}-\boldsymbol{u}_d}^2_{0,\Omega} + \frac{1}{2}\norm{\boldsymbol{y} - \boldsymbol{y}_{d}}^2_{0,\Omega} + \frac{\lambda}{2} \norm{\boldsymbol{U}}^2_{0,\Omega},
\end{eqnarray}
subject to
\begin{align}\label{P:GE}
	\left\{
	\begin{aligned}
		\boldsymbol{K}^{-1} \boldsymbol{u} + (\boldsymbol{u} \cdot \nabla) \boldsymbol{u} - \boldsymbol{\bdiv}(\nu(T) \nabla \boldsymbol{u}) + \nabla p &= \boldsymbol{F}(\boldsymbol{y}) + \boldsymbol{U} \;\; \mbox{in} \;\; \Omega, \\
		\bdiv\hspace{0.04cm}\boldsymbol{u} &= 0 \;\; \mbox{in} \;\; \Omega, \\
		-\boldsymbol{\bdiv}\hspace{-0.1cm}(\boldsymbol{D} \nabla \boldsymbol{y}) + (\boldsymbol{u} \cdot \nabla) \boldsymbol{y} &= 0 \;\; \mbox{in} \;\; \Omega, \\
		\boldsymbol{y} = \boldsymbol{y}^D, \;\; \boldsymbol{u} &= \boldsymbol{0} \;\; \mbox{on} \;\; \Gamma,
	\end{aligned}
	\right.
\end{align}
where $\boldsymbol{u}$ denotes the fluid velocity, $p$ stands for the pressure field, $\boldsymbol{y}:= (T, S)^\top$, $S$ represents the concentration of a certain species within this fluid, and $T$ denotes the temperature, and $\boldsymbol{U}$ is a control. The domain $\Omega \subset \mathbb{R}^d \; (d = 2, 3)$ is  bounded with Lipschitz boundary $\Gamma$. The set of admissible controls $\boldsymbol{\mathcal{U}}_{ad}$  is a non-empty, closed and convex set defined by 
\begin{align}\label{Uad}
	\boldsymbol{\mathcal{U}}_{ad} := \{\boldsymbol{U}=(U_1,\ldots,U_d) \in \boldsymbol{L}^2(\Omega): U_{a_{j}}(x) \leq U_{j}(x) \leq U_{b_{j}}(x) \;\; \mbox{a.e. in} \;\; \Omega, \ j=1,\ldots,d\},
\end{align}
where $U_{a_{j}}(x), U_{b_{j}}(x) \in L^2(\Omega)$ are given, satisfying $U_{a_{j}}(x) \leq U_{b_{j}}(x)$ for all $j = 1, \ldots, d$ and a.e. $x \in \Omega$.  The desired states  $(\boldsymbol{u}_d,\boldsymbol{y}_d)\in\boldsymbol{L}^2(\Omega)\times [L^2(\Omega)]^2$, $\lambda > 0$ is a regularization parameter, $\nu>0$ denotes the temperature-dependent viscosity function, $\boldsymbol{K}(x) > 0$ is the permeability matrix, $\boldsymbol{F}(\boldsymbol{y})$ is a given function modelling buoyancy, and $\boldsymbol{D}$ is a $2 \times 2$  constant matrix of the thermal conductivity and solutal diffusivity coefficients, possibly with cross-diffusion terms. 	The choice of the admissible set of controls $\boldsymbol{\mathcal{U}}_{ad}$ is motivated from \cite{WachsmuthSSC} and allows $\boldsymbol{U}$ to be constrained by functions in $\boldsymbol{L}^2(\Omega)$. Note that in this setup the spaces $L^2(\Omega)$ and $L^{\infty}(\Omega)$ are not the same.

\subsection{Assumptions}\label{Assum}

Throughout this article, we make the following assumptions on the governing equation \eqref{P:GE}. The boundary data is $\boldsymbol{y}^D = (T^D, S^D)^\top \in [H^{1/2}(\Gamma)]^2$ regular. The kinematic viscosity is Lipschitz continuoys and uniformly bounded $|\nu(T_1) - \nu(T_2)| \leq \gamma_{\nu} |T_1 - T_2| \;\; \mbox{and} \;\; \nu_1 \leq \nu(T) \leq \nu_2,$ for all $T_1,T_2,T\in\mathbb{R}$, 	where $\gamma_{\nu}\geq 0$, $ \nu_1, \nu_2$ are positive constants  and $|\cdot|$ denotes the Euclidean norm in $\mathbb{R}^d$. There exist positive constants $\gamma_F, C_F$ such that the buoyancy term
$|\boldsymbol{F}(\boldsymbol{y}_1) - \boldsymbol{F}(\boldsymbol{y}_2)| \leq \gamma_F |\boldsymbol{y}_1 - \boldsymbol{y}_2|,\;\; \mbox{and} \;\; |\boldsymbol{F}(\boldsymbol{y})| \leq C_F |\boldsymbol{y}|,$ for all $\boldsymbol{y}_1,\boldsymbol{y}_2,\boldsymbol{y}\in\mathbb{R}^2$. $\boldsymbol{K}$ is a $d \times d$ permeability matrix with measurable coefficients which is assumed to be symmetric, uniformly positive definite,  and bounded, hence, its inverse satisfies $ \alpha_0 |v|^2 \geq \boldsymbol{v}^\top \boldsymbol{K}^{-1}(x) \boldsymbol{v} \geq \alpha_1 |\boldsymbol{v}|^2,$ for all  $\boldsymbol{v} \in \mathbb{R}^d$ and $x \in \Omega,$ for  constants $\alpha_0, \alpha_1 > 0$. The constant matrix $\boldsymbol{D}$ (not necessarily symmetric) is assumed to be positive definiteness condition, that is, $\boldsymbol{s}^\top \boldsymbol{D} \boldsymbol{s} \geq \alpha_2 |\boldsymbol{s}|^2,$ for all $\boldsymbol{s} \in \mathbb{R}^2$, for a constant $\alpha_2 > 0$.

\subsection{Related works} 
To analyze the proposed control problem, the first step is to investigate the well-posedness of the governing (or state) equation \eqref{P:GE} under the given assumptions. The solvability of the stationary Navier–Stokes equations on bounded domains with smooth boundaries has been established in the classical works of \cite{temam2001navier, Temam_1995}. The authors in \cite{burger2019h} studied the existence of weak solutions and a uniqueness result for the proposed governing equation, however this analysis does not translate directly to the cases, when either the domain is Lipschitz and/or the data is less regular (eg, $\bs{L}^2(\Omega)$). A new well-posedness analysis based on minimal assumptions on data and geometry of the domain was presented by us in \cite{TKM_ncDDF} to overcome this issue, moreover this analysis is readily applicable to our proposed problem. 

In the context of optimal control problems for the Navier–Stokes equations, we refer to the foundational works \cite{abergel1990, SS_OptimalControlViscousFlow, MG_Book}. Moving closer to the specific problem addressed in this article, optimal control problems governed by the stationary Boussinesq equations have been studied in \cite{SICON_Lee, JMAA_Lee}, where both distributed and Neumann boundary control cases are considered, and optimality systems are derived using Lagrangian multipliers under small data assumptions. In many of the above works, smallness assumptions are imposed to ensure uniqueness and differentiability of the control-to-state map. We refer to \cite{JCD_smalldata, WachsmuthSSC, EO_smalldata} for such results. We must also mention that, in some practical scenarios, it is known that solutions to the Navier–Stokes equations are locally unique (non-singular) and continuously depend on the viscosity without any smallness assumption (see \cite[Section 3, p. 297]{Book_FEM_NSE}). This observation is utilized in \cite{Casas_OC-NSE} and related works, where the local optimal solution is assumed to be non-singular, enabling the derivation of optimality conditions of {\it Fritz-John} type without imposing smallness assumptions. We aim explore this approach for the proposed problem in the future.

\subsection{Main contributions and outline}\label{Main_Contributions}
We use the well-posedness analysis and regularity results under minimal assumptions on domain geometry and data regularity, that we proved in \cite{TKM_ncDDF} for the governing equation, to establish the results for the uncontrolled state equation in Section \ref{CtsStateEq}. From a control perspective, the new result here is the derivation of continuous dependence on data result in a more regular space, which is proven in Theorem \ref{State.Wp.Reg}. 
Then we show the existence of an optimal control (see Theorem \ref{ExistenceOptimalControl}), Fr\'echet differentiability of the control-to-state map (see Lemma \ref{FrechetDifferentiability}), the formulation of first-order necessary optimality conditions and the well-posedness of the adjoint equation (see Theorem \ref{Fopt_AdjointEqn} and Lemma \ref{EnergyEstAdjoint}), which are discussed in Section \ref{Sec:FirstOrderCondition}. Expanding upon these findings, in Section \ref{Sec:SecondOrderCondition}, we extend the methodology established in \cite{WachsmuthSSC} to establish the local optimality of the control using a second-order sufficient optimality condition to scenarios when Taylor expansion of temperature-dependent viscosity parameter and the nonlinear buoyancy term do not terminate after the second-order term with zero remainder (see Theorem \ref{SSC_thm1}). 

\subsection{Notations}
Let $\Omega \subset \mathbb{R}^d \; (d = 2, 3)$ be a bounded domain with Lipschitz boundary $\Gamma$. We use the following notations throughout this article: The usual $p^\mathrm{th}$ integrable Lebesgue spaces are denoted by $L^{p}(\Omega)$ with norm $\norm{f}_{L^p(\Omega)}=\left(\int_{\Omega}|f(x)|^pdx\right)^{\frac{1}{p}}$, for $f\in L^p(\Omega)$, $p \in [1,\infty)$.  For $p= \infty$, $L^{\infty}(\Omega)$ is the space of all essentially bounded Lebesgue measurable functions on $\Omega$ with the norm $\norm{f}_{\infty,\Omega} = \esssup_{x \in \Omega} (|f(x)|), $ for $f \in L^{\infty}(\Omega)$. For $p= 2,$ we denote the norm by $\norm{\cdot}_{0,\Omega}$ and $(\cdot, \cdot)$ represents the usual inner product in $L^2(\Omega)$. The space of square integrable functions with zero mean is defined as $L_0^2(\Omega) := \left\{w \in L^2(\Omega): \int_{\Omega} w \; dx = 0\right\}.$
Sobolev spaces are denoted by the standard notation $W^{k,p}(\Omega)$  with the norm $\norm{f}_{W^{k,p}(\Omega)}=\big(\sum_{|\alpha| \leq k}\int_{\Omega} |D^{\alpha} f(x)|^p\big)^{\frac{1}{p}}, \mbox{ if } k \in [1,\infty), 
$ for $f\in W^{k,p}(\Omega)$. For $p = \infty$ we denote the norm by $\norm{f}_{W^{k,\infty}(\Omega)} = \max_{|\alpha| \leq k} \big\{\esssup_{x \in \Omega} |D^{\alpha} f(x)|\big\},$ for $f \in W^{k,\infty}(\Omega).$ Moreover,  for $p = 2,$ we use the notation $W^{k,2}(\Omega) = H^k(\Omega)$ with the corresponding norm denoted by $\norm{\cdot}_{k,\Omega}$ and 
$\boldsymbol{H}_0^1(\Omega) := \left\{\boldsymbol{v} \in \boldsymbol{H}^1(\Omega) : \boldsymbol{v}|_{\Gamma} = \boldsymbol{0} \right\},$ 
where the vector valued functions in dimension $d$ are denoted by bold face and $\boldsymbol{v}|_{\Gamma} = \boldsymbol{0}$ is understood in the sense of trace. The dual space of $\boldsymbol{H}_0^1(\Omega)$ is denoted by $\boldsymbol{H}^{-1}(\Omega)$ with the following norm:
$\norm{\boldsymbol{u}}_{-1,\Omega} := \sup_{0 \neq \boldsymbol{v} \in \boldsymbol{H}_0^1(\Omega)} \frac{\langle \boldsymbol{u}, \boldsymbol{v} \rangle}{~~\norm{\boldsymbol{v}}_{1,\Omega}},$
where $\langle \cdot, \cdot \rangle$  denotes the duality pairing between $\boldsymbol{H}_0^1(\Omega)$ and $\boldsymbol{H}^{-1}(\Omega).$
The duality product between a space $\boldsymbol{V}$ and its dual $\boldsymbol{V}'$ is denoted by $\langle \boldsymbol{f}, \boldsymbol{v}\rangle,$ where $\boldsymbol{f} \in \boldsymbol{V}'$ and $\boldsymbol{v} \in \boldsymbol{V}.$
In the duality product $\langle \boldsymbol{f}, \boldsymbol{v} \rangle$, if $\boldsymbol{f}$ is replaced by a function $\boldsymbol{U} \in \boldsymbol{L}^{r}(\Omega),$ then $\langle f,\boldsymbol{v} \rangle = \left(\boldsymbol{U}, \boldsymbol{v}\right)_{r,r'}$ such that $1/r + 1/r' = 1,$ for all $\boldsymbol{v} \in \boldsymbol{V} \cap \boldsymbol{L}^{r'}(\Omega)$. Throughout the article $C$ will denote a generic positive constant. We will also use the following  fractional form of the Gagliardo-Nirenberg  inequality (see (2.6) in \cite{kumar2021large}; \cite{FractionalGagliardoNirenberg}) is used in the sequel:
Fix $1 \leq r_2, r_1 \leq \infty $ and $d \in \mathbb{N}. $ Let $\theta \in \mathbb{R} $ and $l \in \mathbb{R}^+ $ such that
\begin{align*}
	\begin{aligned}
		\frac{1}{p_1} = \frac{l}{d} + \theta \left(\frac{1}{r_1} - \frac{k}{d}\right) + \frac{1-\theta}{r_2}, \  \frac{l}{k} \leq \theta \leq 1, 
	\end{aligned}
\end{align*}
then we have
\begin{align}\label{FractionalGagliardoNirenberg}
	\begin{aligned}
		\|D^{l} \boldsymbol{u}\|_{L^{p_1}(\Omega)} \leq C_{gn} \|D^{k} \boldsymbol{u}\|^{\theta}_{L^{r_1}(\Omega)} \norm{\boldsymbol{u}}_{L^{r_2}(\Omega)}^{1-\theta} \; \forall \; \boldsymbol{u} \in \boldsymbol{W}^{k,r_1}(\Omega).
	\end{aligned}
\end{align}
We will also invoke the following fractional Sobolev embedding (see for reference Theorem 4.57 in \cite{FractionalSobolevEmbeddings}) in some proofs,
\begin{align}\label{FractionalSobolevEmbedding}
	\begin{aligned}
		\mbox{ if } \; l r_1 < d, \; \mbox{ then } \;  \boldsymbol{W}^{l, r_1}(\Omega) \hookrightarrow \boldsymbol{L}^{p_2}(\Omega) \;  \mbox{ for every} \; p_2 \leq d r_1 / (d - l r_1).
	\end{aligned}
\end{align}

\section{The governing equation}\label{CtsStateEq}
The variational formulation of the governing or state equation \eqref{P:GE} is obtained by testing against suitable functions and integrating by parts, and can be formulated as follows:\\

For some $\bs{U} \in \bs{L}^{\rparam}(\Omega)$, find $(\bs{u},p,\bs{y}) \in \bs{H}_0^1(\Omega) \times L_0^2(\Omega) \times [H^1(\Omega)]^2$ satisfying $\bs{y} = \bs{y}^D$ on $\Gamma$ and 
\begin{align}\label{P:S}
	\left\{
	\begin{aligned}
		a(\bs{y}; \bs{u}, \bs{v}) + c(\bs{u}, \bs{u}, \bs{v}) + b(\bs{v},p) &= d(\bs{y},\bs{v}) + (\bs{U},\bs{v})_{\rparam,\rdual}  \;\; \forall \;\; \bs{v} \in \bs{H}_0^1(\Omega), \\
		b(\bs{u},q) &= 0 \;\; \forall \;\; q \in L_0^2(\Omega), \\
		a_{\bs{y}}(\bs{y},\bs{s}) + c_{\bs{y}}(\bs{u},\bs{y},\bs{s}) &= 0 \;\; \forall \;\; \bs{s} \in [H_0^1(\Omega)]^2. 
	\end{aligned}
	\right.
\end{align}
Standard Sobolev embeddings indicate:
\begin{align}\label{H1embedding}
	\begin{aligned}
		&\mbox{For}\; r' \in [1,\infty) \; \mbox{if} \; d = 2 \;  \mbox{and} \; r' \in [1, 6] \;  \mbox{if} \; d = 3, \; \mbox{there exists} \; C_{{r'}_d} > 0 \\ &\mbox{such that} \; \norm{\bs{w}}_{L^{r'}(\Omega)} \leq C_{{r'}_d} \norm{\bs{w}}_{1,\Omega}, \; \mbox{for all} \; \bs{w} \in \bs{H}^1(\Omega).
	\end{aligned}
\end{align}
Thus $\bs{v} \in \bs{H}_0^1(\Omega)$ implies $\bs{v} \in \bs{L}^{\rdual}(\Omega)$ provided the embedding \eqref{H1embedding} is satisfied. 

The involved forms in \eqref{P:S} are defined as follows:
\begin{align*}
	&a(\bs{y}; \bs{u}, \bs{v}) := (\bs{K}^{-1} \bs{u}, \bs{v}) + (\nu(\bs{y}) \nabla \bs{u}, \nabla \bs{v}), \;\; c(\bs{w},\bs{u},\bs{v}) := ((\bs{w} \cdot \nabla) \bs{u}, \bs{v}), \;\; b(\bs{v},q) := - (q, \bdiv \bs{v}), \\
	& d(\bs{s},\bs{v}) := (\bs{F}(\bs{s}),\bs{v}), \;\; a_{\bs{y}}(\bs{y},\bs{s}) := (\bs{D} \nabla\bs{y}, \nabla \bs{s}), \;\; c_{\bs{y}}(\bs{v},\bs{y},\bs{s}) := ((\bs{v}\cdot\nabla)\bs{y},\bs{s}),
\end{align*}
where, $\nu(\bs{y})$ is understood as the kinematic viscosity depending only on the first component of $\bs{y}$. 
Now we state some results for the governing equation, which are directly applicable via following \cite{TKM_ncDDF} with only minor modifications.


\begin{lemma}[Boundedness (cf. \cite{TKM_ncDDF}{, Section 2.1})]\label{ContinuityProp}
	For all $\bs{u}, \bs{v}, \bs{w} \in \bs{H}^1(\Omega)$, $q \in L^2(\Omega)$, \\and $\bs{y}, \bs{s} \in [H^1(\Omega)]^2,$ there holds
	\begin{align*}
		&|a(\bs{y};\bs{u},\bs{v})|  \leq C_a\|\bs{u}\|_{1,\Omega}\|\bs{v}\|_{1,\Omega}, \;\; |a_{\bs{y}}(\bs{y},\bs{s})| \leq \hat{C}_a \norm{\nabla \bs{y}}_{0,\Omega} \norm{\nabla \bs{s}}_{0,\Omega},  \\
		&|b(\bs{v},q)| \leq \sqrt{d} \norm{q}_{0,\Omega} \norm{\nabla\bs{v}}_{0,\Omega}, \;\; |d(\bs{y},\bs{v})| \leq C_{F} \norm{\bs{y}}_{0,\Omega} \norm{\bs{v}}_{0,\Omega}, \;\; 	|(\bs{U},\bs{v})_{\rparam,\rdual}| \leq  \norm{\bs{U}}_{L^{\rparam}(\Omega)} \norm{\bs{v}}_{L^{\rdual}(\Omega)}, \\
		&|c(\bs{w},\bs{u},\bs{v})| \leq C_{6_d}  C_{3_d} \norm{\bs{w}}_{1,\Omega} \norm{\nabla \bs{u}}_{0,\Omega} \norm{\bs{v}}_{1,\Omega}, \;\; |c_{\bs{y}}(\bs{w},\bs{y},\bs{s})| \leq C_{6_d}C_{3_d} \norm{\bs{w}}_{1,\Omega}  \norm{\nabla \bs{y}}_{0,\Omega} \norm{\bs{s}}_{1, \Omega}. \\
	\end{align*} 
\end{lemma}
Next, Poincar\'e's inequality implies that the bilinear forms $a(\cdot;\cdot,\cdot)$ (for a fixed temperature) and $a_{\bs{y}}(\cdot,\cdot)$ are coercive, that is,
\begin{align*}\label{Coer:a_ay}
	a(\cdot;\bs{v},\bs{v}) 
	\geq \alpha_a \norm{\bs{v}}^2_{1,\Omega} \;\; \forall \;\; \bs{v} \in \bs{H}_0^1(\Omega), \;\mbox{ and }\;
	a_{\bs{y}}(\bs{y},\bs{y}) \geq \alpha_2 \norm{\nabla \bs{y}}^2_{0,\Omega} \;\; \forall \;\; \bs{y} \in [H^1(\Omega)]^2. 
\end{align*}
Let $\bs{X}:= \left\{\bs{v} \in \bs{H}_0^1(\Omega): b(\bs{v},q) =  0\; \forall \; q \in L_0^2(\Omega)\right\} = \left\{\bs{v} \in \bs{H}_0^1(\Omega) : \bdiv\hspace{0.025cm}\bs{v} = 0 \;\; \mbox{in} \;\; \Omega\right\}.$
Then for all $\bs{w} \in \bs{X}$ and $\bs{v} \in \bs{H}_0^1(\Omega)$, the following properties of the trilinear form hold:
\begin{align}\label{Prop:TrilinearForm}
	c(\bs{w},\bs{v},\bs{v}) = 0, \;\; c(\bs{w}, \bs{u}, \bs{v}) = -c(\bs{w}, \bs{v}, \bs{u}), \;\mbox{and}\; c(\bs{u}, \bs{v}, \bs{w}) = ((\nabla \bs{v})^\top \bs{w}, \bs{u}).
\end{align}


\begin{lemma}[Equivalence (cf. \cite{TKM_ncDDF}{, Lemma 3.2})]\label{State:equivalence}
	If $(\bs{u},p,\bs{y}) \in \bs{H}_0^1(\Omega) \times L_0^2(\Omega) \times [H^1(\Omega)]^2$ solves \eqref{P:S} for all $ \bs{U} \in \bs{L}^{\rparam}(\Omega)$, then $(\bs{u},\bs{y}) \in \bs{X} \times [H^1(\Omega)]^2$ satisfies $\bs{y}|_{\Gamma} = \bs{y}^D$ and
	\begin{align}\label{P:Sred}
		\left\{
		\begin{aligned}
			a(\bs{y};\bs{u},\bs{v}) + c(\bs{u}, \bs{u}, \bs{v}) - d(\bs{y}, \bs{v}) - (\bs{U},\bs{v})_{\rparam,\rdual} &= 0 \;\; \forall \;\; \bs{v} \in \bs{X},\\
			a_{\bs{y}}(\bs{y},\bs{s}) + c_{\bs{y}}(\bs{u}, \bs{y}, \bs{s}) &= 0 \;\; \forall \;\; \bs{s} \in [H_0^1(\Omega)]^2.
		\end{aligned}
		\right.
	\end{align}
	Conversely, if $(\bs{u},\bs{y}) \in \bs{X} \times [H^1(\Omega)]^2$ is a solution of the reduced problem \eqref{P:Sred}, then for some $\bs{U} \in \bs{L}^{\rparam}(\Omega),$ there exists a $p \in L_0^2(\Omega)$ such that $(\bs{u},p,\bs{y})$ is a solution of \eqref{P:S}.
\end{lemma}

We write $\bs{y}$ using a lifting argument as $\bs{y} = \bs{y}_0 + \bs{y}_1$, where $\bs{y}_0 \in [H^1_0(\Omega)]^2$ and $\bs{y}_1$ is such that
\begin{eqnarray}\label{lifting}
	\bs{y}_1 \in [H^1(\Omega)]^2 \;\; \mbox{with} \;\; \bs{y}_1|_{\Gamma} = \bs{y}^D.
\end{eqnarray}

\begin{lemma}[Well-posedness and regularity]\label{State.Wp.Reg}
		For every $\bs{U} \in \bs{L}^{\rparam}(\Omega)$, and $\bs{y}^D \in [H^{1 + \delta}(\Gamma)]^2$ there exists a lifting $\bs{y}_1 \in [H^1(\Omega)]^2$ satisfying \eqref{lifting} such that the problem \eqref{P:Sred} has a unique \emph{weak solution} $(\bs{u},\bs{y}) \in [\bs{X} \cap \bs{H}^{3/2 + \delta}(\Omega)] \times [H^1(\Omega) \cap H^{3/2 + \delta}(\Omega)]^2, \delta \in (0,1/2]$ under the following smallness assumption: 
	\begin{align}\label{smalldata}
		\begin{aligned}
			&\alpha_a > C_{6_d} C_{3_d} ((\gamma_{\nu} C_{p_{2_d}} C_{gn} M M_{\bs{y}})/{\hat{\alpha}_a} + M_{\bs{u}} + (\gamma_F M_{\bs{y}})/\hat{\alpha}_a),
		\end{aligned}
	\end{align}	 
	where, $M_{\bs{u}} := C_{\bs{u}} (\|\bs{y}^D\|_{1/2,\Gamma} + \|\bs{U}\|_{L^{r}(\Omega)})$ and $M_{\bs{y}} := C_{\bs{y}} (\|\bs{y}^D\|_{1/2,\Gamma} +  \|\bs{U}\|_{L^{r}(\Omega)})$.  We also have the following bounds:
	\begin{align}\label{MuMy}
		\begin{aligned}
			&\|\bs{u}\|_{1,\Omega} \leq M_{\bs{u}}, \|\bs{y}\|_{1,\Omega} \leq  M_{\bs{y}},  \mbox{ and }\norm{\bs{u}}_{3/2+\delta,\Omega} + \norm{\bs{y}}_{3/2+\delta,\Omega} \leq M,
		\end{aligned}
	\end{align}
	where, $M$ depends on the data  $\|\bs{U}\|_{L^{\rparam}(\Omega)}$ and $\|\bs{y}^D\|_{1+\delta,\Gamma}$. Furthermore, if $(\bs{u}_i, \bs{y}_i) \in \bs{X} \cap \bs{H}^{3/2 + \delta}(\Omega) \times \left[H^1(\Omega) \cap H^{3/2 + \delta}(\Omega)\right]^2$ are weak solutions of \eqref{P:Sred} corresponding to $\bs{U}_i \in \bs{L}^{\rparam}(\Omega),$ for $i = 1, 2$ and $\bs{y}^D  \in \left[H^{1+\delta}(\Gamma)\right]^2$, then under \eqref{smalldata} we have the following continuous dependence of solutions on the data:
	\begin{align}
			\norm{\bs{u}_1 - \bs{u}_2}_{1,\Omega} + \norm{\bs{y}_1 - \bs{y}_2}_{1,\Omega}  &\leq C \norm{\bs{U}_1 - \bs{U}_2}_{L^{\rparam}(\Omega)},\label{stabilityH1}\\			
			\norm{\bs{u}_1 - \bs{u}_2}_{3/2 + \delta,\Omega} + \norm{\bs{y}_1 - \bs{y}_2}_{3/2 + \delta, \Omega} &\leq C \norm{\bs{U}_1 - \bs{U}_2}_{L^{\rparam}(\Omega)}.\label{stabilityH3/2}
	\end{align}
	where $C$ (possibly not the same) is a generic positive constant such that \eqref{smalldata} holds.
\end{lemma}

\begin{proof}
	The constants $C_{\bs{u}}, C_{\bs{y}}$ in the energy estimates \eqref{MuMy} are trivial to quantify following proof of \cite[Lemma 3.4]{TKM_ncDDF}. The existence and minimal regularity of the weak solution follows from \cite[Theorem 3.6, Theorem 3.9]{TKM_ncDDF}. The minimal regularity result allows to prove the uniqueness under the smallness assumption \eqref{smalldata} following \cite[Theorem 3.12]{TKM_ncDDF}. The stability result \eqref{stabilityH1} can be derived using \eqref{H1embedding}, \eqref{smalldata} with $\bs{\alpha} := \bs{u}_1 - \bs{u_2}$, and  the fact that
	\begin{align*}
		\left[\alpha_a -C_{6_d} C_{3_d} \left(\frac{\gamma_{\nu} C_{p_{2_d}} C_{gn} M M_{\bs{y}}}{\hat{\alpha}_a} + M_{\bs{u}} + \frac{\gamma_F M_{\bs{y}}}{\hat{\alpha}_a}\right)\right]\norm{\bs{\alpha}}^2_{1,\Omega} \leq \norm{\bs{U}_1 - \bs{U}_2}_{L^{\rparam}(\Omega)}\norm{\bs{\alpha}}_{L^{\rdual}(\Omega)}.
	\end{align*} 
	Now we focus our attention on proving the continuous dependence on the data result in a more regular space. 	Consider the following difference equation:
	\begin{align*}
		- \nu(T_1) \Delta(\bs{u}_1 - \bs{u}_2) = \bs{\mathcal{F}}_1, \quad
		- \bs{\bdiv}(\bs{D} \nabla(\bs{y}_1 - \bs{y}_2)) = \bs{f}_1,
	\end{align*}
	where,
	\begin{align*}
		\bs{\mathcal{F}}_1 &:= \bs{F}(\bs{y}_1) - \bs{F}(\bs{y}_2) + \bs{U}_1 - \bs{U}_2 - \bs{K}^{-1}(\bs{u}_1 - \bs{u}_2) - ((\bs{u}_1 - \bs{u}_2) \cdot \nabla) \bs{u}_1 - (\bs{u}_2 \cdot \nabla)(\bs{u}_1 - \bs{u}_2) \\
		&~~~~~- \bdiv((\nu(T_1) - \nu(T_2)) (\nabla \bs{u}_2) - (\nabla \nu(T_1) \cdot \nabla)(\bs{u}_2-\bs{u}_1), \\
		\bs{f}_1 &:= ((\bs{u}_2 - \bs{u}_1) \cdot \nabla)  \bs{y}_2 + (\bs{u}_1 \cdot \nabla) (\bs{y}_2 - \bs{y}_1).
	\end{align*}
	Our strategy is to make use of  the regularity in \eqref{MuMy} and \eqref{stabilityH1}. To this end, the following terms can be bounded using  fractional-Leibniz rule with $\frac{1}{p_1} + \frac{1}{p_2} = \frac{1}{q_1} + \frac{1}{q_2} = \frac{1}{2}$ as follows:
	\begin{align}
		\norm{((\bs{u}_2 - \bs{u}_1) \cdot \nabla) \bs{y}_2}_{-1/2+\delta,\Omega} &= \norm{\nabla((\bs{u}_1 - \bs{u}_2) \otimes \bs{y}_2)}_{-1/2+\delta,\Omega} \nonumber\\
		&\leq C \norm{(\bs{u}_1 - \bs{u}_2) \otimes \bs{y}_2}_{1/2+\delta,\Omega}  = C  \; \|D^{1/2+\delta}((\bs{u}_2 - \bs{u}_1) \otimes \bs{y}_2)\|_{0,\Omega} \nonumber\\
		&\leq \norm{\bs{u}_1 - \bs{u}_2}_{L^{p_1}(\Omega)}  \|D^{1/2+\delta} \bs{y}_2\|_{L^{p_2}(\Omega)} + \|D^{1/2+\delta}(\bs{u}_1 - \bs{u}_2)\|_{L^{q_1}(\Omega)} \norm{\bs{y}_2}_{L^{q_2}(\Omega)}. \label{y1my2Reg_Eq1}
	\end{align}
	Choosing $p_1 = \frac{4}{1+2 \delta} = q_2, p_2 = \frac{4}{1-2 \delta} = p_1$  in two dimensions and $p_2= \frac{\delta}{3} = q_1, p_1 = \frac{6}{3-2\delta} = q_2$ in three dimensions and applying \eqref{FractionalGagliardoNirenberg} with $\theta = 1, l = \frac{1}{2} - \delta$ in \eqref{y1my2Reg_Eq1} yields
	\begin{align}\label{y1my2Regularity1}
		\norm{((\bs{u}_2 - \bs{u}_1) \cdot \nabla) \bs{y}_2}_{-1/2+\delta,\Omega} \leq C \norm{\bs{u}_1 - \bs{u}_2}_{1,\Omega} \norm{\bs{y}_2}_{1,\Omega}.
	\end{align}
	Following the same steps, we obtain
	\begin{align}\label{y1my2Regularity2}
		\norm{(\bs{u}_1 \cdot \nabla) (\bs{y}_2 - \bs{y}_1)}_{-1/2+\delta,\Omega} &\leq \norm{\bs{u}_1}_{L^{p_1}(\Omega)}  \|D^{1/2+\delta} (\bs{y}_2 - \bs{y}_1)\|_{L^{p_2}(\Omega)} + \|D^{1/2-\delta} \bs{u}_1 \|_{L^{q_1}(\Omega)} \norm{\bs{y}_2 - \bs{y}_1}_{L^{q_2}(\Omega)} \nonumber\\
		&\leq C \norm{\bs{u}_1}_{1,\Omega} \norm{\bs{y}_2 - \bs{y}_1}_{1,\Omega}.
	\end{align}
	Combining \eqref{y1my2Regularity1} and \eqref{y1my2Regularity2}, we arrive at 
	\begin{align}\label{y1my2Regularity3}
		\norm{\bs{y}_1 - \bs{y}_2}_{3/2+\delta,\Omega} \leq \norm{\bs{f}_1}_{-1/2+\delta,\Omega} \leq C \left(\norm{\bs{u}_1 - \bs{u}_2}_{1,\Omega} + \norm{\bs{y}_1 - \bs{y}_2}_{1,\Omega}\right).
	\end{align}
	Due to the Lipschitz continuity of $\bs{F}(\bs{y})$ and  assumptions on $\bs{K}^{-1}$, we get
	\begin{align}
		&\norm{ \bs{F}(\bs{y}_1) - \bs{F}(\bs{y}_2)}_{{-1/2 - \delta},\Omega} \leq C_e \norm{ \bs{F}(\bs{y}_1) - \bs{F}(\bs{y}_2)}_{0,\Omega} \leq C_e \gamma_{F} \norm{\bs{y}_1 - \bs{y}_2}_{1,\Omega}, \label{main-1}\\
		&\norm{\bs{K}^{-1}(\bs{u}_1 - \bs{u}_2)}_{{-1/2 - \delta},\Omega} \leq C_a \norm{\bs{u}_1 - \bs{u}_2}_{1,\Omega}. \label{main-2}
	\end{align}
	
	The terms $((\bs{u}_1 - \bs{u}_2) \cdot \nabla) \bs{u}_1$ and $ (\bs{u}_2 \cdot \nabla)(\bs{u}_1 - \bs{u}_2)$ in $\bs{\mathcal{F}_1}$ can be bounded similarly to \eqref{y1my2Regularity1} and \eqref{y1my2Regularity2}, respectively using the regularity of $\bs{u}_1$ and $\bs{u}_2$.
	The following bounds can be obtained analogously to \cite[Equation (3.25)]{TKM_ncDDF} in two dimensions:
	\begin{align}
		\norm{(\nabla \nu(T_1) \cdot \nabla) (\bs{u}_2 - \bs{u_1})}_{-1/2 + \delta, \Omega} \leq C
		\norm{\bs{u}_1 - \bs{u}_2}_{1,\Omega} \norm{T_1}_{3/2+\delta,\Omega} \;\; \mbox{in} \;\; \mbox{2D}, \label{main0}
	\end{align} 
	In three dimensions, we employ the iterative technique used in Step 2 of Part II of \cite[Theorem 3.9]{TKM_ncDDF}. Considering $\frac{1}{2(k+1)} \leq \delta < \frac{1}{2}$ for $k = 1, 2, \ldots$, we first show that similar to \cite[Equation (3.32)]{TKM_ncDDF} we have 
	$$\norm{(\nabla \nu(T_1) \cdot \nabla) (\bs{u}_2 - \bs{u_1})}_{-1/2-k\delta,\Omega} \leq C \norm{\bs{u}_1 - \bs{u}_2}_{1,\Omega} \norm{T_1}_{3/2+\delta,\Omega} \; \mbox{ provided } \; \frac{1}{2} \leq (k+1) \delta.$$ Then using the above relation, one can see that
	$$\norm{(\nabla \nu(T_1) \cdot \nabla) (\bs{u}_2 - \bs{u_1})}_{-1/2-(k-1)\delta,\Omega} \leq C \norm{\bs{u}_1 - \bs{u}_2}_{3/2-(k-1)\delta,\Omega} \norm{T_1}_{3/2+\delta,\Omega}.$$ Repeating this $(k-1)$ times  lead to 
	\begin{align}\label{eq!!_2}
		\norm{(\nabla \nu(T_1) \cdot \nabla) (\bs{u}_2 - \bs{u_1})}_{-1/2-(k-(k-1))\delta,\Omega} \leq C \norm{\bs{u}_1 - \bs{u}_2}_{3/2-(k-(k-2))\delta,\Omega} \norm{T_1}_{3/2+\delta,\Omega} \;\; \mbox{in} \;\; \mbox{3D}.
	\end{align}
	Now an application of the fractional Leibniz-rule with $\frac{1}{p_1} + \frac{1}{p_2} = \frac{1}{q_1} + \frac{1}{q_2} = \frac{1}{2}$ gives
	\begin{align}
		\norm{\bdiv((\nu(T_1) - \nu(T_2)) \nabla \bs{u}_2)}_{-1/2-\delta,\Omega} &\leq \sqrt{d}\norm{(\nu(T_1) - \nu(T_2)) \nabla \bs{u}_2}_{1/2-\delta,\Omega} 	\nonumber\\
		&\leq C \norm{D^{1/2-\delta} ((\nu(T_1) - \nu(T_2)) \nabla \bs{u}_2)}_{0,\Omega} \nonumber\\
		&\leq C \left(\norm{\nu(T_1) - \nu(T_2)}_{L^{p_1}(\Omega)} \norm{\nabla \bs{u}_2}_{W^{1/2-\delta,p_2}(\Omega)} \right. \nonumber\\ &~~~~~\left.+  \norm{\nu(T_1) - \nu(T_2)}_{W^{1/2-\delta,q_1}(\Omega)} \norm{\nabla \bs{u}_2}_{L^{q_2}(\Omega)}\right).  \label{main1}
	\end{align}
	Applying \eqref{FractionalGagliardoNirenberg} with $\theta = 1, l = 3/2 - \delta$,  \eqref{H1embedding} and \eqref{FractionalSobolevEmbedding} leads to
	\begin{align}
		\norm{\nu(T_1) - \nu(T_2)}_{L^{p_1}(\Omega)} \norm{\nabla \bs{u}_2}_{W^{1/2-\delta,p_2}(\Omega)}& \leq C \norm{T_1 - T_2}_{L^{p_1}(\Omega)} \norm{\bs{u}_2}_{W^{3/2-\delta,p_2}(\Omega)}\nonumber\\
		&\leq C\begin{cases}
			\norm{T_1 - T_2}_{L^{1/\delta}(\Omega)} \norm{\bs{u}_2}_{W^{3/2-\delta,2/(1-2\delta)}(\Omega)} \; \mbox{in} \; \mbox{2D}, \nonumber\\
			\norm{T_1 - T_2}_{L^{3/2\delta}(\Omega)} \norm{\bs{u}_2}_{W^{3/2-\delta,6/(3-4\delta)}(\Omega)}	\; \mbox{in} \; \mbox{3D}.
		\end{cases}\\
		&\leq C\begin{cases}
			\norm{T_1 - T_2}_{1,\Omega} \norm{\bs{u}_2}_{3/2+\delta,\Omega} \; \mbox{in} \; \mbox{2D}, \label{}\\
			\norm{T_1 - T_2}_{3/2 - 2\delta,\Omega} \nonumber \norm{\bs{u}_2}_{3/2+\delta,\Omega}	\; \mbox{in} \; \mbox{3D}.
		\end{cases}
	\end{align} 
	The treatment of the second term on the right hand side of \eqref{main1} is more delicate and requires another application of the fractional Leibniz rule with $\frac{1}{r_1} + \frac{1}{r_2} = \frac{1}{m_1} + \frac{1}{m_2} = \frac{1}{q_1}$, $\theta \in (0,1)$ and the assumption that $\nu(\cdot)$ is a twice continuously Fr\'echet differentiable function with bounded derivatives upto second order. 
	\begin{align}
		&\norm{\nu(T_1) - \nu(T_2)}_{W^{1/2-\delta,q_1}(\Omega)} \norm{\nabla \bs{u}_2}_{L^{q_2}(\Omega)}  = \|D^{1/2-\delta}(\nu_T(\theta T_1 + (1-\theta)T_2)(T_1-T_2))\|_{L^{q_1}(\Omega)} \norm{\nabla \bs{u}_2}_{L^{q_2}(\Omega)} \nonumber\\
		&\leq \big(\|\nu_T(\theta T_1 + (1-\theta)T_2)\|_{L^{r_1}(\Omega)} \|D^{1/2-\delta}(T_1 - T_2)\|_{L^{r_2}(\Omega)}  \nonumber\\
		&\qquad+ \|D^{1/2-\delta}(\nu_T(\theta T_1  + (1-\theta)T_2))\|_{L^{m_1}(\Omega)} \|T_1 - T_2\|_{L^{m_2}(\Omega)}\big) \norm{\nabla \bs{u}_2}_{L^{q_2}(\Omega)} \nonumber\\
		&\leq \big(\|\nu_T(\theta T_1 + (1-\theta)T_2)\|_{L^{r_1}(\Omega)} \|D^{1/2-\delta}(T_1 - T_2)\|_{L^{r_2}(\Omega)}  \label{New!!}\\
		&\qquad+ \|\nu_{TT}(T)\|_{L^{\infty}(\Omega)} \big(\|D^{1/2-\delta}T_1\|_{L^{m_1}(\Omega)} + \|D^{1/2-\delta}T_2\|_{L^{m_1}(\Omega)}\big) \|T_1 - T_2\|_{L^{m_2}(\Omega)}\big) \norm{\nabla \bs{u}_2}_{L^{q_2}(\Omega)}. \nonumber
	\end{align}
	In two dimensions, we choose $q_1 = 4/(1-2\delta), q_2 = 4/(1+2\delta), r_1 = \infty, r_2 = 4/(1-2\delta), m_2 = \infty, m_1 = 4/(1-2\delta)$,  and apply \eqref{FractionalGagliardoNirenberg} and the following embedding in fractional Sobolev spaces (see Theorem 4.57 in \cite{FractionalSobolevEmbeddings}),
	\begin{align}\label{fractionalLinftyEmbedding}
		\mbox{ if } \; l r_1>d, \; \mbox{ then } \; \bs{W}^{l,r_1}(\Omega) \hookrightarrow \bs{L}^{\infty}(\Omega),
	\end{align}
	leads to
	\begin{align}
		&\norm{\nu(T_1) - \nu(T_2)}_{W^{1/2-\delta,q_1}(\Omega)} \norm{\nabla \bs{u}_2}_{L^{q_2}(\Omega)} \leq C\big( \|\nu_T(\theta T_1 + (1-\theta)T_2)\|_{L^{\infty}(\Omega)} \norm{T_1 - T_2}_{1,\Omega} \nonumber\\
		&\qquad+ \|\nu_{TT}(T)\|_{L^{\infty}(\Omega)} (\norm{T_1}_{1,\Omega} + \norm{T_2}_{1,\Omega}) \norm{T_1 - T_2}_{3/2+\delta,\Omega}\big) \norm{\bs{u}_2}_{3/2-\delta,\Omega}. \nonumber
	\end{align}
	Similarly, in three dimensions, we choose $q_1 = 6/(1+2\delta), q_2 = 3/(1-\delta), r_1 = \infty, r_2 = 6/(1+2\delta), m_2 = \infty, m_1 = 6/(1+2\delta)$,  we get
	\begin{align}
		&	\norm{\nu(T_1) - \nu(T_2)}_{W^{1/2-\delta,q_1}(\Omega)} \norm{\nabla \bs{u}_2}_{L^{q_2}(\Omega)} \leq C\big( \|\nu_T(\theta T_1 + (1-\theta)T_2)\|_{L^{\infty}(\Omega)} \norm{T_1 - T_2}_{3/2-2\delta,\Omega} \nonumber\\
		&\qquad+ \|\nu_{TT}(T)\|_{L^{\infty}(\Omega)} (\norm{T_1}_{3/2-2\delta,\Omega} + \norm{T_2}_{3/2-2\delta,\Omega}) \norm{T_1 - T_2}_{3/2+\delta,\Omega}\big) \norm{\bs{u}_2}_{3/2+\delta,\Omega}. \nonumber
	\end{align}
	Combining the above bounds and making use of \eqref{y1my2Regularity2}  results in
	\begin{align}
		\norm{\bs{u}_1 - \bs{u}_2}_{3/2-\delta,\Omega} \leq \norm{\bs{\mathcal{F}}}_{-1/2-\delta,\Omega} \leq C\left(\norm{\bs{y}_1-\bs{y}_2}_{1,\Omega} + \norm{\bs{u}_1-\bs{u}_2}_{1,\Omega}\right). \label{eq**}
	\end{align}
	Now it remains to show the bounds of $\norm{\bs{\mathcal{F}}}_{-1/2+\delta,\Omega}$. To this end, using the embedding $\bs{L}^2(\Omega) \hookrightarrow \bs{H}^{-1/2+\delta}(\Omega)$, we can obtain the same bounds as in \eqref{main-1} and \eqref{main-2}. Now following \eqref{main1}, we have 
	\begin{align}
		\norm{\bdiv((\nu(T_1) - \nu(T_2)) \nabla \bs{u}_2)}_{-1/2+\delta,\Omega} &\leq 
		C \left(\norm{\nu(T_1) - \nu(T_2)}_{L^{p_1}(\Omega)} \norm{\nabla \bs{u}_2}_{W^{1/2+\delta,p_2}(\Omega)} \right. \nonumber\\ &~~~~~\left.+  \norm{\nu(T_1) - \nu(T_2)}_{W^{1/2+\delta,q_1}(\Omega)} \norm{\nabla \bs{u}_2}_{L^{q_2}(\Omega)}\right), \label{main2}
	\end{align}
	for $\frac{1}{2} = \frac{1}{p_1} + \frac{1}{p_2} = \frac{1}{q_1} + \frac{1}{q_2}$. In two and three dimensions, choosing $p_2 = 2$ and applying \eqref{fractionalLinftyEmbedding} with $l = \frac{3}{2} + \delta$ and $r_1 = 2$ leads to
	{\small\begin{align*}
		\norm{\nu(T_1) - \nu(T_2)}_{L^{p_1}(\Omega)} \norm{\nabla \bs{u}_2}_{W^{1/2+\delta,p_2}(\Omega)} &\leq C \norm{T_1 - T_2}_{\infty,\Omega} \norm{\nabla \bs{u}_2}_{W^{1/2+\delta,2}(\Omega)} \leq C \norm{T_1 - T_2}_{3/2+\delta,\Omega} \norm{ \bs{u}_2}_{3/2+\delta,\Omega}.
	\end{align*}}
	For the second term on the right hand side of \eqref{main2}, analogous to \eqref{New!!}, for $\frac{1}{r_1} + \frac{1}{r_2} = \frac{1}{m_1} + \frac{1}{m_2} = \frac{1}{q_1}$, we have
	\begin{align}
		&\norm{\nu(T_1) - \nu(T_2)}_{W^{1/2+\delta,q_1}(\Omega)} \norm{\nabla \bs{u}_2}_{L^{q_2}(\Omega)}  \leq \big(\|\nu_T(\theta T_1 + (1-\theta)T_2)\|_{L^{r_1}(\Omega)} \|D^{1/2+\delta}(T_1 - T_2)\|_{L^{r_2}(\Omega)}  \nonumber\\
		&\quad+ \|\nu_{TT}(T)\|_{L^{\infty}(\Omega)} \big(\|D^{1/2+\delta}T_1\|_{L^{m_1}(\Omega)} + \|D^{1/2+\delta}T_2\|_{L^{m_1}(\Omega)}\big) \|T_1 - T_2\|_{L^{m_2}(\Omega)}\big) \norm{\nabla \bs{u}_2}_{L^{q_2}(\Omega)}. \nonumber				
	\end{align} 
	In two dimensions, choose $q_1 = 4/(1+2\delta), q_2 = 4/(1-2\delta), r_1 = \infty, r_2 = 4/(1+2\delta), m_2 = \infty, m_1 = 4/(1+2\delta)$,  and apply \eqref{FractionalGagliardoNirenberg} and	\eqref{fractionalLinftyEmbedding},
	\begin{align}
		&	\norm{\nu(T_1) - \nu(T_2)}_{W^{1/2-\delta,q_1}(\Omega)} \norm{\nabla \bs{u}_2}_{L^{q_2}(\Omega)} \leq C\big( \|\nu_T(\theta T_1 + (1-\theta)T_2)\|_{L^{\infty}(\Omega)} \norm{T_1 - T_2}_{1,\Omega} \nonumber\\
		&\qquad+ \|\nu_{TT}(T)\|_{L^{\infty}(\Omega)} (\norm{T_1}_{1,\Omega} + \norm{T_2}_{1,\Omega})  \norm{T_1 - T_2}_{3/2+\delta,\Omega}\big) \norm{\bs{u}_2}_{3/2+\delta,\Omega}. \nonumber
	\end{align}
	Analogously, in three dimensions, choosing $q_1 = 6, q_2 = 3, r_1 = \infty, r_2 = 6, m_2 = \infty, m_1 = 6$, results to
	\begin{align}
		&	\norm{\nu(T_1) - \nu(T_2)}_{W^{1/2-\delta,q_1}(\Omega)} \norm{\nabla \bs{u}_2}_{L^{q_2}(\Omega)} \leq C\big( \|\nu_T(\theta T_1 + (1-\theta)T_2)\|_{L^{\infty}(\Omega)} \norm{T_1 - T_2}_{3/2+\delta,\Omega} \nonumber\\
		&\qquad+ \|\nu_{TT}(T)\|_{L^{\infty}(\Omega)} (\norm{T_1}_{3/2+\delta,\Omega} + \norm{T_2}_{3/2+\delta,\Omega})\norm{T_1 - T_2}_{3/2+\delta,\Omega}\big) \norm{\bs{u}_2}_{3/2+\delta,\Omega}. \nonumber
	\end{align}	
	The following bounds can be obtained analogous to \cite[Equations (3.25), (3.27) and (3.42)]{TKM_ncDDF} in two and three dimensions, respectively,
	\begin{align}\label{mainEq1}
		\norm{(\nabla \nu(T_1) \cdot \nabla) (\bs{u}_2 - \bs{u_1})}_{-1/2 + \delta, \Omega} \leq C \norm{\bs{u}_1 - \bs{u}_2}_{1,\Omega} \norm{T_1}_{3/2+\delta,\Omega} \;\; \mbox{in} \;\; \mbox{2D},
	\end{align}
	and for $\delta_a = (0,\frac{1}{4})$ and $\delta_b = [\frac{1}{4}, \frac{1}{2})$, we have
	\begin{align}\label{mainEq2}
		\begin{aligned}
			&\norm{(\nabla \nu(T_1) \cdot \nabla) (\bs{u}_2 - \bs{u_1})}_{-1/2 + \delta_a, \Omega} \leq \norm{\bs{u}_1 - \bs{u}_2}_{3/2-\delta_a,\Omega} \norm{T_1}_{3/2+2\delta_a,\Omega} \;\; \mbox{in} \;\; \mbox{3D}, \\
			&\norm{(\nabla \nu(T_1) \cdot \nabla) (\bs{u}_2 - \bs{u_1})}_{-1/2 + \delta_b, \Omega} \leq \norm{\bs{u}_1 - \bs{u}_2}_{3/2+\delta_a,\Omega} \norm{T_1}_{3/2+(\delta_b - \delta_a),\Omega} \;\; \mbox{in} \;\; \mbox{3D},
		\end{aligned}
	\end{align}
	where, $2\delta_a$ belongs to $(0, \frac{1}{2})$ and $(\delta_b - \delta_a)$ belongs to $(0,\frac{1}{2})$.
	Combining the above bounds and making use of \eqref{y1my2Regularity3}, \eqref{y1my2Regularity2} and \eqref{eq**} result in
	\begin{align}\label{eq***}
		\norm{\bs{u}_1 - \bs{u}_2}_{3/2+\delta,\Omega} \leq \norm{\bs{\mathcal{F}}}_{-1/2+\delta,\Omega} \leq C\left(\norm{\bs{y}_1-\bs{y}_2}_{1,\Omega} + \norm{\bs{u}_1-\bs{u}_2}_{1,\Omega}\right). 
	\end{align}
	Now an application of \eqref{stabilityH1} in \eqref{y1my2Regularity2} and \eqref{eq***} leads to \eqref{stabilityH3/2}.
\end{proof}

\begin{remark}
	The assumption of twice continuous Fr\'echet differentiablilty and bounded derivatives of  $\nu(\cdot)$ (upto order two) above can be relaxed by assuming that $\nu_{TT}$ exists a.e. such that $\|\nu_{TT}\|_{L^{\infty}(\Omega)}<+\infty$. 
\end{remark}

\begin{remark}[Smallness assumption]\label{cts:smalldata}
	By smallness assumption we mean either the data $\bs{y}^D$ and $\bs{U}$ are ``sufficiently small" or $\alpha_a$ is ``sufficiently large''. For example in the context of state equation \eqref{P:Sred},  the condition \eqref{smalldata} in Lemma \ref{State.Wp.Reg} implies a smallness assumption, which when satisfied  guarantees the uniqueness of a regular weak solution to the state equation. 
\end{remark}

\section{First order necessary optimality conditions} \label{Sec:FirstOrderCondition}
In this section, we discuss the existence of an optimal control corresponding to the cost-functional \eqref{P:CF} and the well-posedness of linearized equations, followed by a study of the differentiability properties of the control-to-state map. Finally, we use these results to derive a reference control’s first-order necessary optimality conditions.  We begin by recalling the definition of $\bs{\mathcal{U}}_{ad}$ from \eqref{Uad} and defining the set of admissible solutions $\bs{S}_{ad}$ as follows:
\begin{align*}
	\bs{S}_{ad} &:= \{(\bs{u},\bs{y},\bs{U}) \in \bs{X} \cap \bs{H}^{3/2+\delta}(\Omega) \times [H^1(\Omega) \cap H^{3/2+\delta}(\Omega)]^2 \times \bs{L}^{\rparam}(\Omega):\\
	&\hspace{3cm} (\bs{u},\bs{y},\bs{U}) \; \mbox{ satisfies } \; \eqref{P:Sred} \ \mbox{ with the control }\  \bs{U} \in \bs{\mathcal{U}}_{ad},\; \bs{y}|_{\Gamma} = \bs{y}^D\}.
\end{align*}
Notice that for all $\bs{U} \in \bs{L}^{\sparam}(\Omega),$ it holds by the interpolation inequality  that
\begin{align}\label{ControlInterpolation}
	\norm{\bs{U}}^2_{{L}^{\rparam}(\Omega)} \leq \norm{\bs{U}}_{L^1(\Omega)} \norm{\bs{U}}_{L^{\sparam}(\Omega)}.
\end{align}

\begin{theorem}[Existence of an optimal control]\label{ExistenceOptimalControl}
	The optimal control problem \eqref{P:CF}-\eqref{P:GE} admits an optimal solution.
\end{theorem}
\begin{proof}
	Since $J(\bs{u},\bs{y},\bs{U}) \geq 0$, there exists an infimum  $\bar{J}$ of $J$ over $\bs{S}_{ad}$, that is,
	$$0 \leq  \bar{J} := \inf_{\bs{S}_{ad}} J(\bs{u},\bs{y},\bs{U}) < + \infty.$$
	By the definition of infimum, there exists a minimizing sequence $(\bs{u}^n, \bs{y}_0^n+ \bs{y}_{1}, \bs{U}^n)\in\bs{S}_{ad}$ such that $$J(\bs{u}^n, \bs{y}_0^n + \bs{y}_{1}, \bs{U}^n) \longrightarrow \bar{J} \;  \mbox{ as } \; n \rightarrow \infty,$$
	where $(\bs{u}^n,\bs{y}_0^{n}+\bs{y}_1)$ is the unique weak solution of the reduced problem \eqref{P:Sred} with control $\bs{U}^n$. Therefore, using the definition and greatest lower bound characterization of infimum, we have for some $N > n$, 
	$$\bar{J} \leq J(\bs{u}^n,\bs{y}_0^{n}+\bs{y}_{1},\bs{U}^n) \leq \bar{J} + \frac{1}{n},$$
	which implies that we can find a large enough constant $\tilde{C} > 0$, such that the set $\left\{\bs{U}^n\right\}_{n\in\mathbb{N}}$ is uniformly bounded in $\bs{L}^2(\Omega)$, that is,
	\begin{eqnarray}\label{ExistenceOptimalControlEq1}
		\norm{\bs{U}^n}^2_{0,\Omega} \leq C < + \infty.
	\end{eqnarray}
	From \eqref{MuMy}, we have the following energy estimates:
	\begin{align*}
		\norm{\bs{u}^n}_{1,\Omega} \leq C_{\bs{u}} \left(\norm{\bs{y}_1}_{1,\Omega} + \norm{\bs{U}^n}_{L^{r}(\Omega)}\right) \;\; \mbox{and} \;\;
		\norm{\nabla \bs{y}_0^n}_{1,\Omega} \leq C_{\bs{y}} \left(\norm{\bs{y}_1}_{1,\Omega} + \norm{\bs{U}^n}_{L^{r}(\Omega)}\right).	
	\end{align*}	
	Now an application of the continuous embedding $\bs{L}^2\hookrightarrow\bs{L}^r$ and \eqref{ExistenceOptimalControlEq1} imply that $\left\{\bs{u}^n\right\}$ and $\left\{\bs{y}_0^n\right\}$ are uniformly bounded in $\bs{H}^1(\Omega)$ and $[H^1(\Omega)]^2$, respectively. Thus, using the Banach-Alaglou Theorem, we can extract a subsequence $\left\{\bs{u}^{n_{k}}, \bs{y}_0^{n_{k}} + \bs{y}_{1}, \bs{U}^{n_{k}}\right\}$ such that
	$$\left\{\bs{u}^{n_{k}}, \bs{y}_0^{n_{k}} + \bs{y}_{1}, \bs{U}^{n_{k}}\right\} \rightharpoonup (\bar{\bs{u}}, \bar{\bs{y}}_0 + {\bs{y}}_{1}, \bar{\bs{U}}) \;\; \mbox{in} \;\; \bs{H}^1(\Omega) \times [H^1(\Omega)]^2 \times \bs{L}^2(\Omega).$$
	Since $\bs{\mathcal{U}}_{ad}$ is closed and convex in $\bs{L}^2(\Omega)$, then by Mazur's Theorem, $\bs{\mathcal{U}}_{ad}$ is weakly closed in $\bs{L}^2(\Omega)$. In other words, every  weakly convergent sequence in $\bs{\mathcal{U}}_{ad}$ has its limit in $\bs{\mathcal{U}}_{ad}$, which implies that $\bar{\bs{U}} \in \bs{\mathcal{U}}_{ad}$. Now it remains to show that $(\bar{\bs{u}},\bar{\bs{y}},\bar{\bs{U}})\in \bs{S}_{ad} $, where $\bar{\bs{y}}=\bar{\bs{y}}_0+\bs{y}_1$.  Since $\bs{H}_0^1(\Omega)$ is compactly embedded in $\bs{L}^2(\Omega)$, we can extract a subsequence $(\bs{u}^{n_{k_j}})_{j \in \mathbb{Z}^+}$ and $(\bs{y}_0^{n_{k_j}})_{j \in \mathbb{Z}^+}$ such that
	\begin{align*}
		\bs{u}^{n_{k_j}} \longrightarrow \bar{\bs{u}} \; \mbox{ in } \; \bs{L}^2(\Omega) \; \mbox{ and } \;\;
		\bs{y}_0^{n_{k_j}} \longrightarrow \bar{\bs{y}}_0 \; \mbox{ in } \; \left[{L}^2(\Omega)\right]^2 \; \mbox{ as } \; j \rightarrow \infty.
	\end{align*}
	Passing the limit in 
	\begin{align}\label{ExistenceOptimalControlEq2}
		\left\{
		\begin{aligned}
			a(\bs{y}_0^n + \bs{y}_{1}; \bs{u}^n, \bs{v}) + c(\bs{u}^n,\bs{u}^n,\bs{v}) - d(\bs{y}_0^n + \bs{y}_{1}, \bs{v}) - (\bs{U}^n, \bs{v}^j) &= 0, \\
			a_{\bs{y}}(\bs{y}_0^n + \bs{y}_{1}, \bs{s}) + c_{\bs{y}}(\bs{u}^n,\bs{y}_0^n + \bs{y}_{1}, \bs{s}) &= 0.
		\end{aligned}
		\right.
	\end{align}
	using the above convergences  as $n \rightarrow \infty$ along $(n_{k_{j}})$, we find that $(\bar{\bs{u}},\bar{\bs{y}},\bar{\bs{U}})$ satisfies \eqref{P:Sred}, provided
	\begin{itemize}
		\item $\bs{c}(\bs{u}^{n_{k_j}},\bs{u}^{n_{k_j}},\bs{v}) \longrightarrow \bs{c}(\bs{\bar{u}},\bs{\bar{u}},\bs{v}) \;\; \forall \;\; \bs{v} \in \bs{C}_0^{\infty}(\Omega)$, which follows by using the above conclusions and following steps analogous to the proof of Theorem \cite[Theorem 3.6]{TKM_ncDDF}.
		\item Similarly, one can obtain $$	\left(\nu(\bs{y}_0^{n_{k_j}} + \bs{y}_{1}) \nabla \bs{u}^{n_{k_j}}, \nabla \bs{v}\right) \longrightarrow \left(\nu(\bar{\bs{y}}_0 + \bs{y}_1) \nabla \bar{\bs{u}}, \nabla \bs{v}\right) \;\; \forall \;\; \bs{v} \in \bs{C}_0^{\infty}(\Omega).$$
	\end{itemize}	
	Limits can be passed to the other terms of \eqref{ExistenceOptimalControlEq2} through the subsequence $(n_{k_j})$ as $j \rightarrow \infty$. Therefore, $(\bar{\bs{u}},\bar{\bs{y}},\bar{\bs{U}})\in \bs{S}_{ad} $. 
	Now it remains to show that $\bar{J} = J(\bar{\bs{u}}, \bar{\bs{y}}, \bar{\bs{U}})$. The continuity and convexity of $J$ ensure its weak lower semicontinuity. Consequently, using (Theorem 2.12, \cite{troltzsch2010optimal}) for a sequence 
	$$\left\{\bs{u}^n, \bs{y}^n, \bs{U}^n\right\} \rightharpoonup \left(\bar{\bs{u}}, \bar{\bs{y}}, \bar{\bs{U}}\right) \;\; \mbox{in} \;\; \bs{H}^1(\Omega) \times \left[H^1(\Omega)\right]^2 \times \bs{L}^2(\Omega),$$
	we have,
	$$\liminf_{n \rightarrow \infty} J(\bs{u}^n, \bs{y}^n, \bs{U}^n) \geq J(\bar{\bs{u}}, \bar{\bs{y}}, \bar{\bs{U}}).$$
	Therefore, we obtain 
	$\bar{J} \leq J(\bar{\bs{u}}, \bar{\bs{y}}, \bar{\bs{U}}) \leq \liminf_{n \rightarrow \infty} J(\bs{u}^n, \bs{y}^n, \bs{U}^n) = \lim_{n \rightarrow \infty} J(\bs{u}^n, \bs{y}^n, \bs{U}^n) = \bar{J},$
	and hence $\left(\bar{\bs{u}}, \bar{\bs{y}}, \bar{\bs{U}}\right)$ is a minimizer of \eqref{P:CF}-\eqref{P:GE}.
\end{proof}

	\begin{remark}\label{Zad_bdd_in_Lr}
		Let the $\bs{L}^{\rparam}(\Omega)$-norm of the admissible controls be bounded by $$\mathcal{M}_{\rparam} = \sup_{\bs{U} \in \bs{\mathcal{U}}_{ad}} \norm{\bs{U}}_{L^{\rparam}(\Omega)}.$$ If the condition  \eqref{smalldata} is satisfied, then using the bounds in \eqref{MuMy}, one can see that Lemma \ref{State.Wp.Reg} ensures the existence and uniqueness of state variables corresponding to the control variable. Therefore, the admissible class $\bs{S}_{ad}$  of solutions is non-empty. 
	\end{remark}

\begin{lemma}[Linearized equations]\label{EnergyEstLinearized}
	Assume that $\bs{F}(\cdot)$ and $\nu(\cdot)$ are twice continuously Fr\'echet differentiable with bounded derivatives  of  order up to two. Given $(\bar{\bs{u}}, \bar{\bs{y}}) \in \bs{X} \cap \bs{H}^{3/2+\delta}(\Omega) \times \left[H^1(\Omega) \cap H^{3/2+\delta}(\Omega)\right]^2$, {\color{blue}let} $(\bs{\psi}, \bs{\chi}) \in \bs{X} \times \left[H_0^1(\Omega)\right]^2$ be solution of the linearized equations 
	\begin{align}\label{linearizedEqns}
		\left\{
		\begin{aligned}
			& a(\bar{\bs{y}}; \bs{\psi}, \bs{v}) + c(\bs{\psi}, \bar{\bs{u}}, \bs{v}) + c(\bar{\bs{u}}, \bs{\psi}, \bs{v}) \\
			&\qquad\qquad+ ((\nu_{T} (\bar{T})) \chi^T \; \nabla \bar{\bs{u}}, \nabla \bs{v}) - ((\bs{F}_{\bs{y}}(\bar{\bs{y}})) \bs{\chi}, \bs{v})  = \langle \bs{\hat{f}}, \bs{v} \rangle \;\; \forall \;\; \bs{v} \in \bs{X}, \\
			& a_{\bs{y}}(\bs{\chi}, \bs{s}) + c_{\bs{y}}(\bar{\bs{u}}, \bs{\chi}, \bs{s}) + c_{\bs{y}}(\bs{\psi}, \bar{\bs{y}}, \bs{s}) = \langle \bs{\tilde{f}}, \bs{s} \rangle  \;\; \forall \;\; \bs{s} \in \left[H_0^1(\Omega)\right]^2,
		\end{aligned}
		\right.
	\end{align}	
	satisfy the following a priori estimate:
	\begin{align}
		\norm{\bs{\psi}}_{1,\Omega} &\leq C_{\bs{\psi}} \left(\|\bs{\hat{f}}\|_{\bs{X}^*,\Omega} + \|\bs{\tilde{f}}\|_{-1,\Omega}\right) := M_{\bs{\psi}}, \label{M_psi}\\
		\norm{\bs{\chi}}_{1,\Omega} &\leq C_{\bs{\chi}} \left(\|\bs{\hat{f}}\|_{\bs{X}^*,\Omega} +  \|\bs{\tilde{f}}\|_{-1,\Omega}\right)  := M_{\bs{\chi}}. \label{M_chi}
	\end{align}
	where, $\bs{\chi} := (\chi^T, \chi^S)$ and $C_{\bs{\psi}}, C_{\bs{\chi}}$ are positive constants and 
	\begin{align}\label{LinearizedEquationConditions}
		\begin{aligned}
			&\alpha_a > C_{6_d} C_{3_d} \left(\frac{C_{\nu_T} C_{p_{2_d}} C_{gn} M M_{\bs{y}}}{\hat{\alpha}_a} + M_{\bs{u}} + \frac{C_{F_{\bs{y}}} M_{\bs{y}}}{\hat{\alpha}_a}\right),	
		\end{aligned}
	\end{align}
	and, $M_{\bs{u}}, M_{\bs{y}}$ and $M$ are given in Lemma \ref{State.Wp.Reg}. 
\end{lemma}
\begin{proof}
	We have the following continuity properties  for all $\bs{v} \in \bs{X}$
	\begin{align}
		|((\bs{F}_{\bs{y}}(\bar{\bs{y}})) \bs{\chi}, \bs{v})| &\leq \norm{\bs{F}_{\bs{y}}(\bar{\bs{y}})}_{\infty,\Omega} \norm{\bs{\chi}}_{0,\Omega} \norm{\bs{v}}_{0,\Omega} \leq C_{F_{\bs{y}}} \norm{\bs{\chi}}_{1,\Omega} \norm{\bs{v}}_{1,\Omega}, \label{CtyFy.1} \\
		|((\nu_{T}(\bar{T}) \chi^T \nabla \bs{\bar{u}}, \nabla \bs{v}))| 
		&\leq \norm{\nu_T(\bar{T})}_{\infty,\Omega} \norm{\nabla \bs{v}}_{0,\Omega} \norm{\nabla \bar{\bs{u}}}_{L^{p_1}(\Omega)} \norm{\chi^T}_{L^{p_2}(\Omega)} \nonumber\\
		&\leq C_{\nu_T} C_{gn} C_{p_{2_d}} \norm{\bs{v}}_{1,\Omega} \norm{\bar{\bs{u}}}_{3/2 + \delta,\Omega} \norm{\chi^T}_{1,\Omega} \nonumber\\
		&\leq C_{\nu_T} C_{gn} C_{p_{2_d}} M \norm{\bs{\chi}}_{1,\Omega} \norm{\bs{v}}_{1,\Omega}, \label{CtynuT.1}
	\end{align}
	where in \eqref{CtynuT.1} we have used \eqref{MuMy} and the fact that for all $0 < \delta < 1/2$ and $\theta = 1$, the  inequality \eqref{FractionalGagliardoNirenberg} gives 
	\begin{align}
		\begin{aligned}
			\mbox{for} \; d = 2,\;  \norm{\nabla \bs{u}}_{
				L^{4/(1-2\delta)}(\Omega)}  &\leq C_{gn} \norm{\bs{u}}_{3/2 + \delta,\Omega}, \label{GN}\\
			\mbox{for} \; d = 3,\;  \norm{\nabla \bs{u}}_{L^{3/(1-\delta)}(\Omega)}  &\leq C_{gn} \norm{\bs{u}}_{3/2 + \delta,\Omega}, 
		\end{aligned}
	\end{align}
    Now, putting $(\bs{v}, \bs{s}) = (\bs{\psi}, \bs{\chi})$ in the linearized equation \eqref{linearizedEqns} and applying  the boundedness properties discussed in Lemma \ref{State.Wp.Reg} along with \eqref{CtyFy.1} and \eqref{CtynuT.1} yield
	\begin{align}
		\alpha_a \norm{\bs{\psi}}^2_{1,\Omega} &\leq C_{6_d} C_{3_d} M_{\bs{u}} \norm{\bs{\psi}}^2_{1,\Omega} + C_{\nu_T} C_{p_{2_d}} C_{gn} M \norm{\bs{\psi}}_{1,\Omega} \norm{\bs{\chi}}_{1,\Omega} + C_{F_{\bs{y}}} \norm{\bs{\psi}}_{1,\Omega} \norm{\bs{\chi}}_{1,\Omega} \nonumber  \\
		& ~~~~~ + \|\bs{\hat{f}}\|_{\bs{X}^*,\Omega}  \norm{\bs{\psi}}_{1,\Omega}, \label{LinearizedEq1}\\
		\hat{\alpha}_a \norm{\bs{\chi}}_{1,\Omega} &\leq C_{6_d} C_{3_d} M_{\bs{y}} \norm{\bs{\psi}}_{1,\Omega} + \|\bs{\tilde{f}}\|_{-1,\Omega}. \label{LinearizedEq2}
	\end{align}
	Substituting \eqref{LinearizedEq2} in \eqref{LinearizedEq1}, we get the desired bounds, given the condition in \eqref{LinearizedEquationConditions} is satisfied. 
\end{proof}

\begin{theorem}[Existence and uniqueness of the linearized equations]\label{WellposednessLinearizedEquations}
	Let $(\bar{\bs{u}}, \bar{\bs{y}}) \in \bs{X} \cap \bs{H}^{3/2 + \delta}(\Omega) \times \left[H^1(\Omega) \cap H^{3/2 + \delta}(\Omega)\right]^2$, $\delta \in \big(0,\frac{1}{2}\big)$  be the states associated with the control $\bar{\bs{U}} \in \bs{\mathcal{U}}_{ad}$. Then under the condition \eqref{LinearizedEquationConditions}, for every $\bs{\hat{f}} \in \bs{X}^*$, $\bs{\tilde{f}} \in \left[H^{-1}(\Omega)\right]^2,$ there exists a unique solution $(\bs{\psi}, \bs{\chi}) \in \bs{X}  \times \left[H_0^1(\Omega)\right]^2,$ where $\bs{\chi} := \left(\chi^T, \chi^S\right),$ of the linearized equations \eqref{linearizedEqns}.
	
\end{theorem}




\begin{proof} Due to the linear nature of \eqref{linearizedEqns}, the existence of a unique solution follows similar to the proof of \cite[Theorem 3.6]{TKM_ncDDF} using Lemma 1.4 of Chapter II in \cite{temam2001navier} as long as the coercivity of the map $\bs{L}_n$ holds. Let $\bs{X}_n$ and $\bs{Y}_n$ be spaces as defined in the proof of Theorem \cite[Theorem 3.6]{TKM_ncDDF}. We define $\bs{L}_n := \bs{L}_n^X \times \bs{L}_n^Y$ such that
	\begin{align}
		\left[\bs{L}_n^X(\bs{\psi}), \bs{v}\right] = \left(\nabla \bs{L}_n^X(\bs{\psi}), \nabla \bs{v}\right)
		&= a(\bar{\bs{y}}; \bs{\psi}, \bs{v}) + c(\bs{\psi}, \bar{\bs{u}}, \bs{v}) + c(\bar{\bs{u}}, \bs{\psi}, \bs{v}) \nonumber\\
		&\qquad+ ((\nu_{T} (\bar{T})) \chi^T \; \nabla \bar{\bs{u}}, \nabla \bs{v}) - ((\bs{F}_{\bs{y}}(\bar{\bs{y}})) \bs{\chi}, \bs{v}) - \langle \bs{\hat{f}}, \bs{v} \rangle, \nonumber\\ 
		\left[\bs{L}_n^Y(\bs{\chi}), \bs{s}\right] = \left(\nabla \bs{L}_n^Y(\bs{\chi}), \nabla \bs{s}\right) 
		&= a_{\bs{y}}(\bs{\chi}, \bs{s}) + c_{\bs{y}}(\bar{\bs{u}}, \bs{\chi}, \bs{s}) + c_{\bs{y}}(\bs{\psi}, \bar{\bs{y}}, \bs{s}) - \langle \bs{\tilde{f}}, \bs{s} \rangle. \nonumber
	\end{align}
	The continuity of $\bs{L}_n$ can be deduced analogous to that of $\bs{P}_n$ in \cite[THeorem 3.6]{TKM_ncDDF} using the bounds discussed in Section \ref{CtsStateEq}, \eqref{CtyFy.1} and \eqref{CtynuT.1}. Similarly, it remains to show the coercivity of $\bs{L}_n$,
	\begin{align*}
		&\left[\bs{L}_n^X(\bs{\psi}), \bs{\psi}\right] \geq  \alpha_a \norm{\bs{\psi}}^2_{1,\Omega}  - C_{6_d} C_{3_d} M_{\bs{u}} \norm{\bs{\psi}}^2_{1,\Omega} - C_{\nu_T} C_{p_{2_d}} C_{gn} M \norm{\bs{\psi}}_{1,\Omega}  \norm{\bs{\chi}}_{1,\Omega} \nonumber\\
		&\hspace{2.5cm} - C_{F_{\bs{y}}} \norm{\bs{\psi}}_{1,\Omega} \norm{\bs{\chi}}_{1,\Omega} - \|\bs{\hat{f}}\|_{\bs{X}^*,\Omega} \norm{\bs{\psi}}_{1,\Omega},\\
		&\left[\bs{L}_n^Y(\bs{\chi}), \bs{\chi}\right] \geq \hat{\alpha}_a \norm{\bs{\chi}}^2_{1,\Omega} - C_{6_d} C_{3_d} M_{\bs{y}} \norm{\bs{\psi}}_{1,\Omega} \norm{\bs{\chi}}_{1,\Omega} - \|\bs{\tilde{f}}\|_{-1,\Omega} \norm{\bs{\chi}}_{1,\Omega}.
	\end{align*}
	It follows from \eqref{M_psi} and \eqref{M_chi} that $\left[\bs{L}_n^X(\bs{\psi}), \bs{\psi}\right] > 0$ for $\left[\bs{\psi}\right] = \norm{\bs{\psi}}^2_{1,\Omega} = \kappa_3$ sufficiently large  such that
	$$\kappa_3 > \left\{\frac{1}{\alpha_a} \left(C_{6_d} C_{3_d} M_{\bs{u}} M^2_{\bs{\psi}}+ C_{\nu_T} C_{p_{2_d}} C_{gn} M M_{\bs{\psi}} M_{\bs{\chi}} + C_{F_{\bs{y}}} M_{\bs{\psi}} M_{\bs{\chi}} + M_{\bs{\psi}} \|\bs{\hat{f}}\|_{X^{*},\Omega} \right)\right\},$$
	and $\left[\bs{L}_n^Y(\bs{\chi}), \bs{\chi}\right] > 0$ for $\left[\bs{\chi}\right] = \norm{\bs{\chi}}^2_{1,\Omega} = \kappa_4$ sufficiently large such that
	$$ \kappa_4 > \left\{\frac{1}{\hat{\alpha}_a} \left(C_{6_d} C_{3_d} M_{\bs{y}} M_{\bs{\psi}} M_{\bs{\chi}} + M_{\bs{\chi}} \|\bs{\tilde{f}}\|_{-1,\Omega}\right)\right\}.$$
	From the coercivity of $\bs{L}_n^X$ and $\bs{L}_n^Y$ and the definition of $\bs{L}_n$, we deduce  that $\left[\bs{L}_n(\hat{\bs{w}}), \hat{\bs{w}}\right] > 0$ for $\left[\hat{\bs{w}}\right] = \kappa_3 + \kappa_4 > 0$. 
\end{proof}

\vspace{0.2cm}
Moving next to the aspects of optimization, we denote the solution operator of \eqref{P:Sred} which maps $\bs{U} \mapsto (\bs{u}, \bs{y})$  by $G(\bs{U}) = (\bs{u}, \bs{y})$. In the following lemma, we discuss the Fr\'echet differentiability of the solution mapping.

\begin{lemma}[Fr\'echet differentiability of the solution mapping]\label{FrechetDifferentiability}
	The solution operator $G: \bs{L}^r(\Omega) \longrightarrow \bs{X} \cap \bs{H}^{3/2 + \delta}(\Omega) \times \left[H^{3/2 + \delta}(\Omega) \right]^2$ is Fr\'echet differentiable. In particular, $G$ is Fr\'echet differentiable  from $\bs{L}^2(\Omega)$ to $\bs{X} \cap \bs{H}^{3/2 + \delta} (\Omega) \times \left[H^{3/2+\delta}(\Omega)\right]^2 $. The derivative $G'(\bar{\bs{U}})\bs{h}= (\bs{\zeta}, \bs{\mu})$, where under the condition \eqref{LinearizedEquationConditions}, $(\bs{\zeta}, \bs{\mu})\in \bs{X} \cap \bs{H^{{3/2}+\delta}}(\Omega)  \times \left[H_0^1(\Omega)\right]^2 \cap \left[H^{{3/2}+\delta}(\Omega)\right]^2$ is the unique weak solution of 
	\begin{align}\label{PreAdjoint}
		\left\{
		\begin{aligned}
			& a(\bar{\bs{y}}; \bs{\zeta}, \bs{v}) + c(\bs{\zeta}, \bar{\bs{u}}, \bs{v}) + c(\bar{\bs{u}}, \bs{\zeta}, \bs{v}) \\
			&\qquad\qquad + ((\nu_{T} (\bar{T})) \mu^T \; \nabla \bar{\bs{u}}, \nabla \bs{v}) - ((\bs{F}_{\bs{y}}(\bar{\bs{y}})) \bs{\mu}, \bs{v})  = \left(\bs{h}, \bs{v}\right) \;\; \forall \;\; \bs{v} \in \bs{X}, \\
			& a_{\bs{y}}(\bs{\mu}, \bs{s}) + c_{\bs{y}}(\bar{\bs{u}}, \bs{\mu}, \bs{s}) + c_{\bs{y}}(\bs{\zeta}, \bar{\bs{y}}, \bs{s}) = 0  \;\; \forall \;\; \bs{s} \in \left[H_0^1(\Omega)\right]^2,
		\end{aligned}
		\right.
	\end{align}
	with $\bs{\mu} := (\mu^T, \mu^S)$, $\bs{U} \in \bs{\mathcal{U}}_{ad}, (\bar{\bs{u}}, \bar{\bs{y}}) := G(\bar{\bs{U}})$ and $\bs{h} \in \bs{L}^2(\Omega)$.
\end{lemma}
\begin{proof}
	Since $\bs{L}^r(\Omega)$ is a linear space, we have $\bar{\bs{U} }+ \bs{h} \in \bs{L}^r(\Omega),$ for all $\bar{\bs{U} },\bs{h} \in \bs{L}^r(\Omega)$. For any perturbation $\bs{h}$, set $(\bs{u}, \bs{y}) = G(\bar{\bs{U}} + \bs{h})$. We find that the difference $( \bs{\delta\zeta},  \bs{\delta\mu}) := (\bs{u} - \bar{\bs{u}}, \bs{y} - \bar{\bs{y}})$, where $ \bs{\delta\mu} := ( \delta\mu^T,  \delta\mu^S)$ is the weak solution of (under assumption analogous to \eqref{LinearizedEquationConditions})
	\begin{align*}
		\left\{
		\begin{aligned}
			& a(\bar{\bs{y}};  \bs{\delta\zeta}, \bs{v}) + c( \bs{\delta\zeta}, \bar{\bs{u}}, \bs{v}) + c(\bar{\bs{u}},  \bs{\delta\zeta}, \bs{v}) \\
			&\qquad\qquad+ ((\nu_{T} (\bar{T}))  \delta\mu^T \; \nabla \bar{\bs{u}}, \nabla \bs{v}) - ((\bs{F}_{\bs{y}}(\bar{\bs{y}}))  \bs{\delta\mu}, \bs{v})  = \left(\bs{h}, \bs{v}\right)_{\rparam,\rdual} + R_{\bs{u}} \;\; \forall \;\; \bs{v} \in \bs{X}, \\
			& a_{\bs{y}}(\bs{\delta\mu}, \bs{s}) + c_{\bs{y}}(\bar{\bs{u}}, \bs{\delta\mu}, \bs{s}) + c_{\bs{y}}(\bs{\delta\zeta}, \bar{\bs{y}}, \bs{s}) = R_{\bs{y}}  \;\; \forall \;\; \bs{s} \in \left[H_0^1(\Omega)\right]^2.
		\end{aligned}
		\right.
	\end{align*}
	Here, due to the inherent nonlinear nature of the coupled problem, $R_{\bs{u}}$ and $R_{\bs{y}}$ consist of quadratic and mixed higher-order remainder terms and are defined for $\Theta \in (0,1),$ respectively as follows:
	\begin{align*}
		&R_{\bs{u}} := c(\bs{u} - \bar{\bs{u}}, \bs{u} - \bar{\bs{u}}, \bs{v}) + \frac{1}{2!} (F_{\bs{y}\bs{y}} (\bs{y} + \Theta (\bs{y} - \bar{\bs{y}})) (\bs{y} - \bar{\bs{y}})^2, \bs{v}) \\
		&~~~~~~~~~- \frac{1}{2!}(\nu_{TT}(T + \Theta (T - \bar{T})) (T - \bar{T})^2 \nabla \bar{\bs{u}}, \nabla \bs{v}) - (\nu_T(T + \Theta (T - \bar{T})) (T - \bar{T}) \nabla (\bs{u} - \bar{\bs{u}}), \nabla \bs{v}), \\
		&R_{\bs{y}} := c_{\bs{y}}(\bs{u} - \bar{\bs{u}}, \bs{y} - \bar{\bs{y}}, \bs{s}).
	\end{align*} 
	We split $ \bs{\delta\zeta}$ and $ \bs{\delta\mu}$ into $ \bs{\delta\zeta} = \bs{\zeta} + r_{\bs{\zeta}}$ and $ \bs{\delta\mu} = \bs{\mu} + r_{\bs{\mu}}$, respectively, where $(\bs{\zeta}, \bs{\mu})$ and $(r_{\bs{\zeta}}, r_{\bs{\mu}})$ are weak solutions of (under assumptions analogous to \eqref{LinearizedEquationConditions})
	\begin{align*}
		\left\{
		\begin{aligned}
			& a(\bar{\bs{y}}; \bs{\zeta}, \bs{v}) + c(\bs{\zeta}, \bar{\bs{u}}, \bs{v}) + c(\bar{\bs{u}}, \bs{\zeta}, \bs{v}) + ((\nu_{T} (\bar{T})) \mu^T \; \nabla \bar{\bs{u}}, \nabla \bs{v}) - ((\bs{F}_{\bs{y}}(\bar{\bs{y}}))  \bs{\mu}, \bs{v})  = \left(\bs{h}, \bs{v}\right)_{\rparam,\rdual} \; \forall \; \bs{v} \in \bs{X}, \\
			& a_{\bs{y}}(\bs{\mu}, \bs{s}) + c_{\bs{y}}(\bar{\bs{u}}, \bs{\mu}, \bs{s}) + c_{\bs{y}}(\bs{\zeta}, \bar{\bs{y}}, \bs{s}) = 0  \; \forall \; \bs{s} \in \left[H_0^1(\Omega)\right]^2,
		\end{aligned}
		\right.
	\end{align*} 
	and
	\begin{align*}
		\left\{
		\begin{aligned}
			& a(\bar{\bs{y}}; r_{\bs{\zeta}}, \bs{v}) + c(r_{\bs{\zeta}}, \bar{\bs{u}}, \bs{v}) + c(\bar{\bs{u}}, r_{\bs{\zeta}}, \bs{v}) + ((\nu_{T} (\bar{T})) r_{\mu^T} \; \nabla \bar{\bs{u}}, \nabla \bs{v}) - ((\bs{F}_{\bs{y}}(\bar{\bs{y}})) r_{\bs{\mu}}, \bs{v})  =  R_{\bs{u}} \; \forall \; \bs{v} \in \bs{X}, \\
			& a_{\bs{y}}(r_{\bs{\mu}}, \bs{s}) + c_{\bs{y}}(\bar{\bs{u}}, r_{\bs{\mu}}, \bs{s}) + c_{\bs{y}}(r_{\bs{\zeta}}, \bar{\bs{y}}, \bs{s}) = R_{\bs{y}}  \; \forall \; \bs{s} \in \left[H_0^1(\Omega)\right]^2,
		\end{aligned}
		\right.
	\end{align*}
	respectively.
	An application of Lemma \ref{EnergyEstLinearized} with $\langle \bs{\hat{f}}, \bs{v}\rangle = (\bs{h}, \bs{v})_{\rparam,\rdual}$ and $\langle \bs{\tilde{f}}, \bs{v}\rangle = 0$ yields,
	$$\norm{\bs{\zeta}}_{1,\Omega} + \norm{\bs{\mu}}_{1,\Omega} \leq C \norm{\bs{h}}_{L^{\rparam}(\Omega)}.$$
	This shows that the mapping $\bs{h} \mapsto (\bs{\zeta}, \bs{\mu})$ is continuous from $\bs{L}^r(\Omega)$ to $\bs{X} \cap \bs{H}^{3/2 + \delta}(\Omega) \times \left[H_0^1(\Omega) \cap H^{3/2+\delta}(\Omega)\right]^2$. If we show that 
	\begin{align}\label{FrechetDerivative}
		\frac{\norm{(\bs{u} - \bar{\bs{u}}, \bs{y} - \bar{\bs{y}}) - (\bs{\zeta}, \bs{\mu})}_{\bs{X}  \times \left[H_0^1(\Omega)\right]^2}}{\norm{\bs{h}}_{L^{\rparam}(\Omega)}} \longrightarrow 0 \;\; \mbox{as} \;\; \norm{\bs{h}}_{L^{\rparam}(\Omega)} \longrightarrow 0,
	\end{align}
	then $(\bs{\zeta}, \bs{\mu})$ will be Fr\'echet derivative of $G$ at $\bar{\bs{U}}$ in the direction of $\bs{h}$. In order to show this, let us first study the boundedness of remainder terms. By using the boundedness properties discussed in Section \ref{} and utilizing the regularity of $\bs{u},$ we obtain
	\begin{align}
		&|c(\bs{u} - \bar{\bs{u}}, \bs{u} - \bar{\bs{u}}, \bs{v})| \leq C_{6_d} C_{3_d} \norm{\bs{u} - \bar{\bs{u}}}^2_{1,\Omega} \norm{\bs{v}}_{1,\Omega} =  C_{6_d} C_{3_d} \norm{G(\bar{\bs{U}} + \bs{h}) - G(\bar{\bs{U}})}^2_{1,\Omega} \norm{\bs{v}}_{1,\Omega}, \nonumber\\
		&|(F_{\bs{y} \bs{y}} \left(\bar{\bs{y}} + \Theta (\bs{y} - \bar{\bs{y}})\right) \left(\bs{y} - \bar{\bs{y}})^2, \bs{v}\right)| \leq \norm{F_{\bs{y} \bs{y}} \left(\bar{\bs{y}} + \Theta (\bs{y} - \bar{\bs{y}})\right)}_{\infty,\Omega} \norm{(\bs{y} - \bar{\bs{y}})^2}_{0,\Omega} \norm{\bs{v}}_{0,\Omega} \nonumber\\
		&~~\hspace{3.8cm}\leq C_{F_{\bs{y} \bs{y}}}  \norm{\bs{y} - \bar{\bs{y}}}^2_{L^4(\Omega)} \norm{\bs{v}}_{1,\Omega} \leq C_{F_{\bs{y} \bs{y}}} C_{4_d} \norm{G(\bar{\bs{U}} + \bs{h}) - G(\bar{\bs{U}})}^2_{1,\Omega} \norm{\bs{v}}_{1,\Omega},\nonumber\\
		&|(\nu_{TT} \left(\bar{T} + \Theta (T - \bar{T})\right) (T - \bar{T})^2 \nabla \bar{\bs{u}}, \nabla \bs{v})| \nonumber\\
		&~~~\leq \norm{\nu_{TT}\left(\bar{T} + \Theta (T - \bar{T})\right)}_{\infty,\Omega} \norm{\nabla \bar{\bs{u}}}_{L^{r_1}(\Omega)} \norm{T - \bar{T}}^2_{L^{2r_2}(\Omega)} \norm{\nabla \bs{v}}_{0,\Omega} \nonumber\\
		&~~~\leq\begin{cases}
			C_{\nu_{TT}} C_{gn} C_{r_{2_d}} C_{r_{3_d}} \norm{\bar{\bs{u}}}_{3/2 +\delta, \Omega} \norm{\bs{y} - \bar{\bs{y}}}^2_{1,\Omega} \norm{\bs{v}}_{1,\Omega} \;\;\; \mbox{in 2D}\nonumber\\
			C_{\nu_{TT}} C  \norm{\bar{\bs{u}}}_{3/2+\delta,\Omega} \norm{\bs{y} - \bar{\bs{y}}}^2_{3/2-\delta,\Omega} \norm{\bs{v}}_{1,\Omega} \;\;\; \mbox{in 3D} \nonumber
			&\end{cases}\\
		&~~~\leq C\begin{cases}
			\norm{G(\bar{\bs{U}} + \bs{h}) - G(\bar{\bs{U}})}^2_{1,\Omega} \norm{\bs{v}}_{1,\Omega} \;\;\; \mbox{in 2D}\label{FrechetDiff_Eq1}\\\
			\norm{G(\bar{\bs{U}} + \bs{h}) - G(\bar{\bs{U}})}^2_{3/2-\delta,\Omega} \norm{\bs{v}}_{1,\Omega}  \;\;\; \mbox{in 3D} ,
			&\end{cases}
	\end{align}
	where the last inequality holds with $r_1 = \frac{4}{1-2\delta}$ and $r_2 = \frac{4}{1+2\delta}$ in two dimensions and in three dimensions it can be settled by choosing $r_1 = \frac{3}{1-\delta}$ and $r_2 = \frac{6}{1+2\delta}$ and applying \eqref{FractionalSobolevEmbedding}. In order to bound the last term in $R_{\bs{u}},$ we use the Gagliardo-Nirenberg inequality followed by an application of the Sobolev embedding (see Theorem 5, Chapter 5 in \cite{mitrovic1997fundamentals})
	and the stability estimates in \eqref{stabilityH3/2} to get
	\begin{align*}
		&|(\nu_T(T + \Theta (T - \bar{T})) (T - \bar{T}) \nabla (\bs{u} - \bar{\bs{u}}), \nabla \bs{v})| \\
		&~~\leq \norm{\nu_T(T + \Theta (T - \bar{T}))}_{\infty,\Omega} \norm{T - \bar{T}}_{L^6(\Omega)} \norm{\nabla (\bs{u} - \bar{\bs{u}})}_{L^3(\Omega)} \norm{\nabla \bs{v}}_{0,\Omega}, \\
		&~~\leq\begin{cases}
			C_{\nu_T} C_{6_d} C_{gn} \norm{T - \bar{T}}_{1,\Omega} \norm{\bs{u} - \bar{\bs{u}}}_{{4/3},\Omega} \norm{\bs{v}}_{1,\Omega} \;\;\; \mbox{in}\;\; \mbox{2D}\\
			C_{\nu_T} C_{6_d} C_{gn} \norm{T - \bar{T}}_{1,\Omega} \norm{\bs{u} - \bar{\bs{u}}}_{{3/2},\Omega} \norm{\bs{v}}_{1,\Omega} \;\;\; \mbox{in}\;\; \mbox{3D}
			&\end{cases}\\
		&~~\leq C_{\nu_T} C_{6_d} C_{gn} C \norm{\bs{y} - \bar{\bs{y}}}_{1,\Omega} \norm{\bs{u} - \bar{\bs{u}}}_{3/2+\delta,\Omega} \norm{\bs{v}}_{1,\Omega} \\
		&~~\leq  C_{\nu_T} C_{6_d} C_{gn} C \norm{G(\bar{\bs{U}} + \bs{h}) - G(\bar{\bs{U}})}_{1,\Omega} \norm{G(\bar{\bs{U}} + \bs{h}) - G(\bar{\bs{U}})}_{3/2+\delta,\Omega} \norm{\bs{v}}_{1,\Omega}.
	\end{align*}
	Combining the above bounds,  one can obtain the following bound on $R_{\bs{u}}$: 
	\begin{align}
		|R_{\bs{u}}| &\leq C \left(\norm{G(\bar{\bs{U}} + \bs{h}) - G(\bar{\bs{U}})}^2_{1,\Omega} + \norm{G(\bar{\bs{U}} + \bs{h}) - G(\bar{\bs{U}})}^2_{3/2-\delta,\Omega} \right. \nonumber\\
		&\quad\quad \left.+ \norm{G(\bar{\bs{U}} + \bs{h}) - G(\bar{\bs{U}})}_{1,\Omega} \norm{G(\bar{\bs{U}} + \bs{h}) - G(\bar{\bs{U}})}_{3/2+\delta,\Omega}\right) \norm{\bs{v}}_{1,\Omega}. \label{Ru}
	\end{align}

	Similarly, we can get the following bound on $R_{\bs{y}}$:
	\begin{align}\label{Ry}
		|R_{\bs{y}}| \leq C_{6_d} C_{3_d} \norm{\bs{u} - \bar{\bs{u}}}_{1,\Omega} \norm{\bs{y} - \bar{\bs{y}}}_{1,\Omega} \norm{\bs{s}}_{1,\Omega} \leq C \norm{G(\bar{\bs{U}} + \bs{h}) - G(\bar{\bs{U}})}^2_{1,\Omega} \norm{\bs{s}}_{1,\Omega}.
	\end{align}
	
	Following steps analogous to the proof of Lemma \ref{EnergyEstLinearized} and using \eqref{Ru}, \eqref{Ry}, we can readily obtain the following estimates under the condition \eqref{LinearizedEquationConditions}:
	\begin{align}
		\norm{r_{\bs{\zeta}}}_{1,\Omega} + \norm{r_{\bs{\mu}}}_{1,\Omega} &\leq C \left(\norm{G(\bar{\bs{U}} + \bs{h}) - G(\bar{\bs{U}})}^2_{1,\Omega}  + \norm{G(\bar{\bs{U}} + \bs{h}) - G(\bar{\bs{U}})}^2_{3/2-\delta,\Omega} \right.\nonumber\\ 
		&\qquad\qquad \left.+ \norm{G(\bar{\bs{U}} + \bs{h}) - G(\bar{\bs{U}})}_{1,\Omega} \norm{G(\bar{\bs{U}} + \bs{h}) - G(\bar{\bs{U}})}_{3/2+\delta,\Omega}\right). \label{rury}
	\end{align}
	Now from the definition of Fr\'echet derivative (see Section 2.6 in \cite{troltzsch2010optimal}) and Lipschitz continuity of the solution mapping (see \eqref{stabilityH1} and \eqref{stabilityH3/2}) in \eqref{rury}, one can deduce the following:
	\begin{align*}
		\norm{(\bs{u} - \bar{\bs{u}}, \bs{y} - \bar{\bs{y}}) - (\bs{\zeta}, \bs{\mu})}_{\bs{X}  \times \left[H_0^1(\Omega)\right]^2} 
		&= \norm{(r_{\bs{\zeta}}, r_{\bs{\mu}})}_{\bs{X}  \times \left[H_0^1(\Omega)\right]^2}
		= \norm{r_{\bs{\zeta}}}_{1,\Omega} + \norm{r_{\bs{\mu}}}_{1,\Omega} \leq C \norm{\bs{h}}^2_{L^{\rparam}(\Omega)}.
	\end{align*}
	Thus \eqref{FrechetDerivative} is fulfilled and $G$ is Fr\'echet differentiable with derivative $G'(\bar{\bs{U}}) \bs{h}= (\bs{\zeta}, \bs{\mu})$. Consequently, we can observe that $G : \bs{L}^2(\Omega) \longrightarrow \bs{X} \cap \bs{H}^{3/2 + \delta} (\Omega) \times \left[H^{3/2+\delta}(\Omega)\right]^2$ is Fr\'echet differentiable. Since all the calculations will also hold for $(\bs{h}, \bs{v})$ for all $\bs{h} \in \bs{L}^2(\Omega)$ and $\bs{v} \in \bs{X}$.  Hence, the proof is completed. 
\end{proof}

Due to the presence of nonlinearity in the governing equation, the optimization problem \eqref{P:CF}-\eqref{P:GE} is non-convex, and hence the optimal solution $(\bar{\bs{u}}, \bar{\bs{y}}, \bar{\bs{U}}) $ obtained in Theorem \ref{ExistenceOptimalControl}  may not be unique without imposing additional assumptions. Theoretically, many global and local minima are possible. We use the notion of a locally optimal reference control to study the norms where local optimality can be assured  and show that it satisfies first order necessary optimality conditions.

\begin{definition}[Local optimal control]\label{local_optimal_control}
	A control $\bar{\bs{U}} \in \bs{\mathcal{U}}_{ad}$ is called \emph{locally optimal} in the sense of $\bs{L}^{\sparam}(\Omega)$, if there exists a positive constant $\rho$ such that $$J(\bar{\bs{u}}, \bar{\bs{y}}, \bar{\bs{U}}) \leq J(\bs{u}, \bs{y}, \bs{U})$$
	holds for all $\bs{U} \in \bs{\mathcal{U}}_{ad}$ with $\norm{\bar{\bs{U}} - \bs{U}}_{L^{\sparam}(\Omega)} \leq \rho$, where $\bar{\bs{u}}, \bar{\bs{y}}$ and $\bs{u}, \bs{y}$ denote the states corresponding to $\bar{\bs{U}}$ and $\bs{U}$, respectively. 
\end{definition}

		\begin{remark}
		The buoyancy term is understood as a mapping
		\[
		\bs{F}:\mathbb{R}^2 \to \mathbb{R}^d, 
		\qquad \bs{y}=(T,S)^\top \mapsto \bs{F}(\bs{y}).
		\]
		Hence, for $\bar{\bs{y}}=(\bar T,\bar S)^\top$, the derivative $\bs{F}_y(\bar{\bs{y}})$ denotes the Jacobian matrix of $\bs{F}$ w.r.t. $\bs{y}$ (for $d = 3$)
		\[
		\bs{F}_y(\bar{\bs{y}})
		=
		\begin{pmatrix}
		\partial_T \bs{F}_1(\bar{\bs{y}}) & \partial_S \bs{F}_1(\bar{\bs{y}}) \\
		\partial_T \bs{F}_2(\bar{\bs{y}}) & \partial_S \bs{F}_2(\bar{\bs{y}}) \\
		\partial_T \bs{F}_3(\bar{\bs{y}}) & \partial_S \bs{F}_3(\bar{\bs{y}})
		\end{pmatrix}
		\in \mathbb{R}^{3\times 2}.
		\]
		\end{remark}

	\begin{theorem}[First order necessary optimality condition]\label{Fopt_AdjointEqn}
		Let $\bar{\bs{U}}$ be a \emph{local optimal control} in the sense of $L^2(\Omega)$ for the optimization problem \eqref{P:CF}-\eqref{P:GE} corresponding to the states $\bar{\bs{u}} = \bs{u}(\bar{\bs{U}})$ and $\bar{\bs{y}} = \bs{y}(\bar{\bs{U}})$. Then there exists a unique solution $(\bar{\bs{\varphi}}, \bar{\bs{\eta}}) \in \bs{X}  \times \left[H_0^1(\Omega)\right]^2$ of the adjoint equation for all $(\bs{v}, \bs{s}) \in \bs{X} \times \left[H_0^1(\Omega)\right]^2$,
		{\small\begin{align}\label{adjointEq}
			\left\{
			\begin{aligned}
				&a(\bar{\bs{y}},\bar{\bs{\varphi}},\bs{v}) - c(\bar{\bs{u}}, \bar{\bs{\varphi}},\bs{v}) + c(\bs{v},\bar{\bs{u}},\bar{\bs{\varphi}}) + c_{\bs{y}}(\bs{v},\bar{\bs{y}},\bar{\bs{\eta}}) = (\bar{\bs{u}}-\bs{u}_d,\bs{v}), \\
				&a_{\bs{y}}(\bar{\bs{\eta}},\bs{s}) - c_{\bs{y}}(\bar{\bs{u}},\bar{\bs{\eta}},\bs{s}) + (\nu_{T}(\bar{T}) \nabla \bar{\bs{u}}:\nabla \bar{\bs{\varphi}}, \bs{s}) - ((\bs{F}_{\bs{y}}(\bar{\bs{y}}))^{\top} \bar{\bs{\varphi}}, \bs{s}) = (\bar{\bs{y}}-\bs{y}_d,\bs{s}),
			\end{aligned}
			\right.
		\end{align}	}
		where  $(\nu_{T}(\bar{T}) \nabla \bar{\bs{u}} : \nabla \bar{\bs{\varphi}}, \bs{s})=\left(\binom{\nu_{T}(\bar{T}) \nabla \bar{\bs{u}}:\nabla \bar{\bs{\varphi}}}{0}, \binom{s_1}{s_2}\right)$. 
		Furthermore, the following variational inequality is  satisfied:
		\begin{align}\label{vi}
			\left(\lambda \bar{\bs{U}} + \bar{\bs{\varphi}}, \bs{U} - \bar{\bs{U}}\right) \geq 0 \;\; \forall \;\; \bs{U} \in \bs{\mathcal{U}}_{ad}.
		\end{align}
	\end{theorem}
	
	\begin{proof}
		Note that we are dealing with the constrained control  case, that is, $\bs{\mathcal{U}}_{ad}\neq \bs{L}^2(\Omega)$. 	Let $\tilde{G}: \bs{L}^2(\Omega) \longrightarrow \bs{X} \times \left[H^1(\Omega)\right]^2$ denote the solution operator $G$ restricted to $\bs{L}^2(\Omega)$.  Then the reduced objective functional \eqref{P:CF} can be written as
		$$j(\bs{U}) := J(\tilde{G}(\bs{U}), \bs{U}) = \frac{1}{2} \|\tilde{G}(\bs{U}) - \bs{\mathcal{Y}}_d\|^2_{0,\Omega} + \frac{\lambda}{2} \norm{\bs{U}}^2_{0,\Omega},$$
		where $\bs{\mathcal{Y}_d} := (\bs{u}_d, \bs{y}_d)$.	Due to Lemma \ref{FrechetDifferentiability}, we know that $\tilde{G}$ is Fr\'echet differentiable, and $\bs{\mathcal{U}}_{ad}$ is a closed convex subset of $\bs{L}^2(\Omega)$,  we have the following necessary condition for $\bar{\bs{U}}$ to be a local optimum of $j(\bs{U})$,
		$$j'(\bar{\bs{U}})(\bs{U} - \bar{\bs{U}}) \geq 0 \;\; \forall \;\; \bs{U} \in {\bs{\mathcal{U}}}_{ad},$$
		which can be rewritten for all  $\bs{U} \in \bs{\mathcal{U}}_{ad}$ as,
		\begin{align}\label{Fonoc_Eq1}
			\left(\tilde{G}(\bar{\bs{U}}) - \bs{\mathcal{Y}}_d, \tilde{G}'(\bar{\bs{U}})(\bs{U} - \bar{\bs{U}})\right) + \lambda (\bar{\bs{U}},\bs{U} - \bar{\bs{U}}) \geq 0.
		\end{align}
		Set $(\bs{\zeta},\bs{\mu}) = \tilde{G}'(\bar{\bs{U}})(\bs{U} - \bar{\bs{U}})$, then $(\bs{\zeta}, \bs{\mu})$ satisfies \eqref{PreAdjoint} with $\bs{h}$  replaced by $\bs{U} - \bar{\bs{U}}$. Let $(\bar{\bs{\varphi}}, \bar{\bs{\eta}})$ be the solution of \eqref{adjointEq} and its existence can be argued as in the proof of Theorem \ref{WellposednessLinearizedEquations}. Testing \eqref{PreAdjoint} by $(\bar{\bs{\varphi}}, \bar{\bs{\eta}})$, we get
		{\small
			\begin{align}\label{Fonoc_Eq2}
			\left\{
			\begin{aligned}
				& a(\bar{\bs{y}}; \bs{\zeta}, \bar{\bs{\varphi}}) + c(\bs{\zeta}, \bar{\bs{u}}, \bar{\bs{\varphi}}) + c(\bar{\bs{u}}, \bs{\zeta}, \bar{\bs{\varphi}}) + ((\nu_{T} (\bar{T})) \mu^T \; \nabla \bar{\bs{u}}, \nabla \bar{\bs{\varphi}}) - ((\bs{F}_{\bs{y}}(\bar{\bs{y}})) \bs{\mu}, \bar{\bs{\varphi}})  = \left(\bs{U} - \bar{\bs{U}}, \bar{\bs{\varphi}}\right), \\
				& a_{\bs{y}}(\bs{\mu}, \bar{\bs{\eta}}) + c_{\bs{y}}(\bar{\bs{u}}, \bs{\mu}, \bar{\bs{\eta}}) + c_{\bs{y}}(\bs{\zeta}, \bar{\bs{y}}, \bar{\bs{\eta}}) = 0,
			\end{aligned}
			\right.
		\end{align}}
		where $\mu^T$ is first component of $\bs{\mu}$.	Testing \eqref{adjointEq} by $(\bs{\zeta}, \bs{\mu})$ yields
		\begin{align}\label{Fonoc_Eq3}
			\left\{
			\begin{aligned}
				&a(\bar{\bs{y}},\bar{\bs{\varphi}},\bs{\zeta}) - c(\bar{\bs{u}}, \bar{\bs{\varphi}},\bs{\zeta}) + c(\bs{\zeta},\bar{\bs{u}},\bar{\bs{\varphi}}) + c_{\bs{y}}(\bs{\zeta},\bar{\bs{y}},\bar{\bs{\eta}}) = (\bar{\bs{u}}-\bs{u}_d,\bs{\zeta}), \\
				&a_{\bs{y}}(\bar{\bs{\eta}},\bs{\mu}) - c_{\bs{y}}(\bar{\bs{u}},\bar{\bs{\eta}},\bs{\mu}) + (\nu_{T}(\bar{T}) \nabla \bar{\bs{u}}:\nabla \bar{\bs{\varphi}}, \mu^T) - ((\bs{F}_{\bs{y}}(\bar{\bs{y}}))^{\top} \bar{\bs{\varphi}}, \bs{\mu}) = (\bar{\bs{y}}-\bs{y}_d,\bs{\mu}).
			\end{aligned}
			\right.
		\end{align}	
		
		Subtracting \eqref{Fonoc_Eq2} from \eqref{Fonoc_Eq3}, using the calculations in Remark \ref{adjointCalculation} below and simplifying, one can deduce the following:

		\begin{align*}
			(\bs{U} - \bar{\bs{U}}, \bar{\bs{\varphi}}) = 
			(\bar{\bs{u}}-\bs{u}_d, \bs{\zeta}) + (\bar{\bs{y}}-\bs{y}_d, \bs{\mu}) = ((\bar{\bs{u}}, \bar{\bs{y}}) - (\bs{u}_d,\bs{y}_d),(\bs{\zeta}, \bs{\mu})) =
			(\tilde{G}(\bar{\bs{U}}) - \bs{\mathcal{Y}}_d,\tilde{G}'(\bar{\bs{U}})(\bs{U} - \bar{\bs{U}})),
		\end{align*} 

		Substituting in \eqref{Fonoc_Eq1}, we get
		$$\left(\lambda \bar{\bs{U}} + \bar{\bs{\varphi}}, \bs{U} - \bar{\bs{U}}\right) \geq 0 \;\; \forall \;\; \bs{U} \in \bs{\mathcal{U}}_{ad},$$
		which is the variational inequality associated with the control.
	\end{proof}
	\begin{remark}[Derivation of adjoint equation]\label{adjointCalculation}
		Let $((\bar{\bs{u}} \cdot \nabla)\bs{\zeta},\bar{\bs{\varphi}}) + ((\bs{\zeta} \cdot \nabla) \bar{\bs{u}}, \bar{\bs{\varphi}}) = (\bs{\hat{L}}(\bar{\bs{u}})\bs{\zeta},\bar{\bs{\varphi}}) + (\bs{\tilde{L}}(\bar{\bs{u}}) \bs{\zeta}, \bar{\bs{\varphi}}).$
		Then using integration by parts and the fact that $\bar{\bs{u}} \in \bs{X}$, $\bs{\zeta} \in \bs{H}_0^1(\Omega)$ and \eqref{Prop:TrilinearForm}, we have
		{\small\begin{align*}
			&(\bs{\hat{L}}(\bar{\bs{u}}) \bs{\zeta}, \bs{\varphi}) = \sum_{i,j = 1}^{d} \int_{\Omega} \bar{u}_i \frac{\partial \zeta_j}{\partial x_i} \varphi_j \; dx 
			= - \sum_{i,j=1}^{d} \int_{\Omega} \bar{u}_i \zeta_j \frac{\partial \varphi_j}{\partial x_i} \; dx\Rightarrow(\bs{\zeta}, \bs{\hat{L}}^*(\bar{\bs{u}})\bs{\varphi}) = - (\bs{\zeta}, (\bar{\bs{u}} \cdot \nabla) \bs{\varphi}) . \\
			&(\bs{\tilde{L}}(\bar{\bs{u}}) \bs{\zeta},\bar{\bs{\varphi}}) = ((\nabla \bar{\bs{u}}) \bs{\zeta}, \bar{\bs{\varphi}})\Rightarrow  (\bs{\zeta}, \bs{\tilde{L}}^*(\bar{\bs{u}}) \bar{\bs{\varphi}}) = (\bs{\zeta}, (\nabla \bar{\bs{u}})^{\top} \bar{\bs{\varphi}}).
		\end{align*}}
		Similar calculations hold for $c_{\bs{y}}(\cdot,\cdot,\cdot)$ also.
	\end{remark}
	\begin{remark}[Characterization of optimal control]
		Moreover, using a projection on the admissible control set $\bs{\mathcal{U}}_{ad}$,
		\begin{align}\label{ae_pointwise_projection}
			\mathcal{P}_{\bs{\mathcal{U}}_{ad}}: \bs{\mathcal{U}} \longrightarrow \bs{\mathcal{U}}_{ad}, 	\;\; \mathcal{P}_{\bs{\mathcal{U}}_{ad}} (U_j(x)) = \max(U_{a_j}(x), \min(U_{b_j}(x), U_j(x))),
		\end{align}
		the variational inequality \eqref{vi} can be expressed component-wise as follows for a.e. $x \in \Omega$:
		\begin{align}\label{projection_formula}
			\bar{U}_j(x) =  \mathcal{P}_{\bs{\mathcal{U}}_{ad}} \left(- \frac{1}{\lambda} \bar{\varphi}_j(x)\right)=\max\left(U_{a_j}(x), \min\left(U_{b_j}(x), - \frac{1}{\lambda} \bar{\varphi}_j(x)\right)\right),
		\end{align}
		where the subscript $j = 1, \dots, d$ (cf. \cite{troltzsch2010optimal}).
	\end{remark}
	\begin{remark}[Regularity of adjoint]\label{adjointRegularity}
		Similar to the governing equation, we can recover the adjoint pressure, denoted by $\xi$ in the reduced adjoint formulation \eqref{adjointEq} by using the ``de Rham's Theorem"
		, which is to find 
		$(\bar{\bs{\varphi}}, \xi, \bar{\bs{\eta}}) \in \bs{H}_0^1(\Omega) \times L_0^2(\Omega) \times [H_0^1(\Omega)]^2$ for all $(\bs{v}, q, \bs{s}) \in \bs{H}_0^1(\Omega) \times L_0^2(\Omega) \times [H_0^1(\Omega)]^2$ such that
		\begin{align}\label{P:A}
			\left\{
			\begin{aligned}
				a(\bar{\bs{y}},\bar{\bs{\varphi}},\bs{v}) - c(\bar{\bs{u}}, \bar{\bs{\varphi}},\bs{v}) + c(\bs{v},\bar{\bs{u}},\bar{\bs{\varphi}}) + b(\bs{v}, \xi) + c_{\bs{y}}(\bs{v},\bar{\bs{y}},\bar{\bs{\eta}}) &= (\bar{\bs{u}}-\bs{u}_d,\bs{v}),\\
				b(\bar{\bs{\varphi}},q) &= 0,\\
				a_{\bs{y}}(\bar{\bs{\eta}},\bs{s}) - c_{\bs{y}}(\bar{\bs{u}},\bar{\bs{\eta}},\bs{s}) + (\nu_{T}(\bar{T}) \nabla \bar{\bs{u}} : \nabla \bar{\bs{\varphi}}, \bs{s}) - ((\bs{F}_{\bs{y}}(\bar{\bs{y}}))^{\top} \bar{\bs{\varphi}}, \bs{s}) &= (\bar{\bs{y}}-\bs{y}_d,\bs{s}).
			\end{aligned}	
			\right.
		\end{align}
	Furthermore, similar to the state equation for $\bs{U} \in \bs{L}^r(\Omega),\; \bs{y}^D \in [H^{1+\delta}(\Gamma)]^2, (\bs{u}, p, \bs{y}) \in [\bs{H}_0^1(\Omega) \cap \bs{H}^{3/2+\delta}(\Omega)] \times [L^2_0(\Omega) \cap H^{1/2+\delta}(\Omega)] \times [H^{3/2+ \delta}(\Omega)]^2, \delta \in (0,1/2)$, the weak solution to \eqref{P:A} satisfies  
	$$(\bar{\bs{\varphi}}, \xi, \bar{\bs{\eta}}) \in [\bs{H}_0^1(\Omega) \cap \bs{H}^{3/2+\delta}(\Omega)] \times [L^2_0(\Omega) \cap H^{1/2+\delta}(\Omega)] \times [H_0^1(\Omega) \cap H^{3/2+ \delta}(\Omega)]^2,$$
	and
	$$\norm{\bar{\bs{\varphi}}}_{3/2+\delta,\Omega} + \norm{\bar{\bs{\eta}}}_{3/2+\delta,\Omega} \leq  C (\bar{\bs{u}}, \bar{\bs{y}}, {\bs{y}^D}, \bar{\bs{U}}, \bs{u}_d, \bs{y}_d) =: \bar{M}.$$
	\end{remark}
	
	The following estimate can be proven along the lines of Lemma \ref{EnergyEstLinearized}.
	\begin{lemma}[Energy estimates]\label{EnergyEstAdjoint}
		The adjoint state $(\bar{\bs{\varphi}}, \bar{\bs{\eta}})$, given by \eqref{adjointEq}, satisfies the following a priori estimate for some positive constants  $C_{\bs{\varphi}}$ and $C_{\bs{\eta}}$ provided \eqref{LinearizedEquationConditions} holds:
		{\small\begin{align*}
			\norm{\bar{\bs{\varphi}}}_{1,\Omega} \leq C_{\bs{\varphi}} \left(\norm{\bar{\bs{u}} - \bs{u}_d}_{0, \Omega} + \norm{\bar{\bs{y}} - \bs{y}_d}_{0,\Omega}\right) := M_{\bs{\varphi}}, 
			\norm{\bar{\bs{\eta}}}_{1,\Omega} \leq C_{\bs{\eta}} \left(\norm{\bar{\bs{u}} - \bs{u}_d}_{0, \Omega} + \norm{\bar{\bs{y}} - \bs{y}_d}_{0,\Omega}\right) := M_{\bs{\eta}}.
		\end{align*}}
	\end{lemma}

\begin{remark}[KKT optimality system]\label{CtsOptimalitySystem}
	In summary, we have derived the \emph{continuous optimality system} made up of the state or governing equation \eqref{P:S}, the adjoint equation \eqref{adjointEq}, and the variational inequality \eqref{vi} for the unknowns {\small$(\bar{\bs{u}}, p, \bar{\bs{y}}, \bar{\bs{\varphi}}, \xi, \bar{\bs{\eta}}, \bar{\bs{U}}) \in \bs{H}_0^1(\Omega) \times L_0^2(\Omega) \times [H^1(\Omega)]^2 \times \bs{H}_0^1(\Omega) \times L_0^2(\Omega) \times [H_0^1(\Omega)]^2 \times \bs{\mathcal{U}}_{ad}$ satisfying $\bar{\bs{y}}|_{\Gamma} = \bs{y}^D$}. 	Every solution $(\bar{\bs{u}}, p, \bar{\bs{y}}, \bar{\bs{U}})$ to the optimal control problem \eqref{P:CF}-\eqref{P:GE} must, together with $(\bar{\bs{\varphi}}, \xi, \bar{\bs{\eta}})$, satisfy this optimality system.

\end{remark}

\section{Second order sufficient optimality conditions}\label{Sec:SecondOrderCondition}
The sufficient conditions we present here are inspired by the framework developed in \cite{WachsmuthSSC}, in which the authors prove the coercivity of the second derivative of the Lagrangian for a subspace of all possible directions by using strongly active constraints. Throughout this section, we assume that $\bs{F}(\cdot)$ and $\nu(\cdot)$ are thrice continuously Fr\'echet differentiable with bounded derivatives of order up to three. Now, let us introduce the Lagrangian corresponding to the cost functional \eqref{P:CF} and the reduced state equation \eqref{P:Sred},
\begin{align}\label{Lagrangian}
	&~~~~~~~~~~~~~\cl{L} : \bs{X} \times [H^1(\Omega)]^2 \times \bs{L}^2(\Omega) \times \bs{X} \times [H_0^1(\Omega)]^2 \longrightarrow \mathbb{R} \;\; \mbox{such that} \nonumber\\
	&\cl{L}(\bs{u},\bs{y},\bs{U},\bs{\varphi},\bs{\eta}) = J(\bs{u},\bs{y},\bs{U}) - \left[a(\bs{y};\bs{u},\bs{\varphi}) + c(\bs{u}, \bs{u}, \bs{\varphi}) - d(\bs{y}, \bs{\varphi}) - (\bs{U},\bs{\varphi})\right] \nonumber \\
	&~~~~~~~~~~~~~~~~~~~~~~~~~~~~~~~~~~~~~~~~~~~~~~~~~~~~~~~~~~~~~~~~~~~~~~ - \left[a_{\bs{y}}(\bs{y},\bs{\eta}) + c_{\bs{y}}(\bs{u}, \bs{y}, \bs{\eta})\right]. 
\end{align}
In the next Lemma, we discuss the Fr\'echet-differentiability of $\cl{L}$.

\begin{lemma}\label{LagrangianFrechetDifferentiability}
	The Lagrangian $\cl{L}$ is thrice Fr\'echet-differentiable with respect to $\bs{w} = (\bs{u}, \bs{y}, \bs{U})$ from $\bs{X} \cap \bs{H}^{3/2 + \delta}(\Omega) \times \left[ H^{3/2+\delta}(\Omega)\right]^2 \times \bs{L}^2(\Omega)$ to $\mathbb{R}$. The second-order derivative at $\bar{\bs{w}} = (\bar{\bs{u}}, \bar{{\bs{y}}}, \bar{\bs{U}})$ fulfils together with the adjoint state $\bar{\bs{\Lambda}} = (\bar{\bs{\varphi}}, \bar{\bs{\eta}})$
	\begin{align*}
		\bs{\cl{L}}_{\bs{w} \bs{w}}(\bar{\bs{w}}, \bar{\bs{\Lambda}}) \left[(\bs{\zeta}_1, \bs{\mu}_1, \bs{h}_1), (\bs{\zeta}_2, \bs{\mu}_2, \bs{h}_2)\right] &= \bs{\cl{L}}_{\bs{u} \bs{u}} (\bar{\bs{w}}, \bar{\bs{\Lambda}})\left[\bs{\zeta}_1, \bs{\zeta}_2\right] + \bs{\cl{L}}_{\bs{y} \bs{y}} (\bar{\bs{w}}, \bar{\bs{\Lambda}})\left[\bs{\mu}_1, \bs{\mu}_2\right] + \bs{\cl{L}}_{\bs{U} \bs{U}} (\bar{\bs{w}}, \bar{\bs{\Lambda}})\left[\bs{h}_1, \bs{h}_2\right] \\
		&\quad + \bs{\cl{L}_{\bs{uy}}(\bar{w}, \bar{\Lambda})} \left[\bs{\zeta}_1, \bs{\mu}_2\right] + \bs{\cl{L}_{\bs{yu}}(\bar{w}, \bar{\Lambda})} \left[\bs{\mu}_1, \bs{\zeta}_2\right].
	\end{align*}
	and
	\begin{align*}
		|\bs{\cl{L}}_{\bs{u}\bs{u}}(\bar{\bs{w}}, \bar{\bs{\Lambda}})\left[\bs{\zeta}_1, \bs{\zeta}_2\right]|	&\leq  C_{\cl{L}_1} \norm{\bs{\zeta}_1}_{1,\Omega} \norm{\bs{\zeta}_2}_{1,\Omega}, \\
		|\bs{\cl{L}}_{\bs{y}\bs{y}}(\bar{\bs{w}}, \bar{\bs{\Lambda}})\left[\bs{\mu}_1, \bs{\mu}_2\right]|	&\leq C_{\cl{L}_2} \norm{\bs{\mu}_1}_{1,\Omega} \norm{\bs{\mu}_2}_{1,\Omega}, \\
		|\bs{\cl{L}_{y y y}}(\bar{\bs{w}}, \bar{\bs{\Lambda}}) [\bs{\mu}_1, \bs{\mu}_2, \bs{\mu}_3]| &\leq C_{\cl{L}_3} \norm{\bs{\mu}_1}_{1,\Omega} \norm{\bs{\mu}_2}_{1,\Omega} \norm{\bs{\mu}_3}_{1,\Omega}, \\
		 |\bs{\cl{L}_{\bs{uy}}(\bar{w}, \bar{\Lambda})} \left[\bs{\zeta}_1, \bs{\mu}_1\right]| &\leq C_{\cl{L}_4} \norm{\bs{\zeta}_1}_{1,\Omega} \norm{\bs{\mu}_1}_{1,\Omega}, \\
		|\bs{\cl{L}_{\bs{yu}}(\bar{w}, \bar{\Lambda})} \left[\bs{\mu}_1, \bs{\zeta}_1\right]| &\leq C_{\cl{L}_4} \norm{\bs{\mu}_1}_{1,\Omega} \norm{\bs{\zeta}_1}_{1,\Omega}, \\
		|\bs{\cl{L}_{u y y}}(\bar{\bs{w}}, \bar{\bs{\Lambda}}) [\bs{\zeta}_1, \bs{\mu}_1, \bs{\mu}_2]| &\leq C_{\cl{L}_5} \norm{\bs{\zeta}_1}_{1,\Omega} \norm{\bs{\mu}_1}_{1,\Omega} \norm{\bs{\mu}_2}_{1,\Omega}, 
	\end{align*}

	for all $(\bs{\zeta}_i, \bs{\mu}_i, \bs{h}_i) \in \bs{X} \times \left[H^1(\Omega)\right]^2 \times \bs{L}^2(\Omega)$ with  positive constants $C_{\cl{L}_i}$ independent of $\bar{\bs{w}}, \bs{\zeta}_i, \bs{\mu}_i$.  Here the last inequality in three dimensions holds for $\delta \in [\frac{1}{4}, \frac{1}{2}),$ which can be further relaxed to include the whole range of $\delta \in (0,\frac{1}{2})$ under reasonable assumptions (see Remark \ref{relaxing_delta}). 
\end{lemma}
\begin{proof}
	The first order derivatives with respect to $\bs{u}, \bs{y}$ and $\bs{U}$ in the direction $(\bs{\zeta}, \bs{\mu}, \bs{h}) \in \bs{X} \times \left[H^1(\Omega)\right]^2 \times \bs{L}^2(\Omega)$ are,
	\begin{align*}
		\cl{L}_{\bs{u}}(\bar{\bs{w}}, \bar{\bs{\Lambda}})[\bs{\zeta}] &= (\bs{\zeta}, \bar{\bs{u}} - \bs{u}_d)  - a(\bar{\bs{y}}; \bs{\zeta}, \bar{\bs{\varphi}}) - c(\bs{\zeta}, \bar{\bs{u}}, \bar{\bs{\varphi}}) - c(\bar{\bs{u}}, \bs{\zeta}, \bar{\bs{\varphi}}) - c_{\bs{y}}(\bs{\zeta}, \bar{\bs{y}}, \bar{\bs{\eta}}),\\
		\cl{L}_{\bs{y}}(\bar{\bs{w}}, \bar{\bs{\Lambda}})[\bs{\mu}] &= (\bs{\mu}, \bar{\bs{y}} - \bs{y}_d) - a_{\bs{y}}(\bs{\mu}, \bar{\bs{\eta}}) - c_{\bs{y}}(\bar{\bs{u}}, \bs{\mu}, \bar{\bs{\eta}}) - ((\nu_{T} (\bar{T})) \mu^T \; \nabla \bar{\bs{u}}, \nabla \bar{\bs{\varphi}}) + ((\bs{F}_{\bs{y}}(\bar{\bs{y}})) \bs{\mu}, \bar{\bs{\varphi}}),\\
		\cl{L}_{\bs{U}}(\bar{\bs{w}}, \bar{\bs{\Lambda}})[\bs{h}] &= \lambda (\bs{h}, \bar{\bs{U}}) + (\bs{h}, \bar{\bs{\varphi}}),
	\end{align*}
	where $\bs{\mu} := (\mu^T, \mu^S)$ and $(\bs{\zeta}, \bs{\mu}, \bs{h}) \in \bs{X} \times \left[H^1(\Omega)\right]^2 \times \bs{L}^2(\Omega)$. The mappings $\bar{\bs{u}} \mapsto \bs{\cl{L}_{u}}$ and $\bar{\bs{U}} \mapsto \bs{\cl{L}_{U}}$ are affine and the linear parts are bounded due to Lemma \ref{ContinuityProp} and hence continuous. On the other hand, the mapping $\bar{\bs{y}} \mapsto \bs{\cl{L}_{y}}$ is nonlinear. Using  the regularity of $\bar{\bs{u}}$ and energy estimates of $\bar{{\bs{u}}}$ and $\bar{\bs{\varphi}}$ (see Lemma \ref{State.Wp.Reg} and Lemma \ref{EnergyEstAdjoint}), we have 
	\begin{align}
		|((\bs{F}_{\bs{y}}(\bar{\bs{y}})) \bs{\mu}, \bs{v})| &\leq \norm{(\bs{F}_{\bs{y}}(\bar{\bs{y}}))}_{\infty,\Omega} \norm{\bs{\mu}}_{0,\Omega} \norm{\bs{v}}_{0,\Omega} \leq C_{F_{\bs{y}}} \norm{\bs{\mu}}_{1,\Omega} \norm{\bs{v}}_{1,\Omega}, \label{CtyFy} \\
		|((\nu_{T}(\bar{T}) \mu^T \nabla \bs{\bar{u}}, \nabla \bs{v}))| 
		&\leq \norm{\nu_T(\bar{T})}_{\infty,\Omega} \norm{\mu^T}_{L^{p_2}} \norm{\nabla \bar{\bs{u}}}_{L^{p_1}} \norm{\nabla \bar{\bs{v}}}_{0,\Omega} \leq C_{\nu_T} C_{gn} C_{p_{2_d}} M \norm{\bs{\mu}}_{1,\Omega} \norm{\bs{v}}_{1,\Omega}.  \label{CtynuT}
	\end{align}
	In order to bound \eqref{CtynuT} we used $p_1$ and $p_2$ values as in \cite[(2.10), (2.11)]{TKM_ncDDF}.
	Therefore, the mapping is bounded due to the bounds in Lemma \ref{ContinuityProp}, \eqref{CtyFy} and \eqref{CtynuT}. Thus the mappings are Fr\'echet-differentiable and twice Fr\'echet-differentiability of $\bs{\cl{L}}$ is implied. The second order derivatives with respect to $\bs{u}, \bs{y}$ and $\bs{U}$ are,
	\begin{align}
		\cl{L}_{\bs{u} \bs{u}}(\bar{\bs{w}}, \bar{\bs{\Lambda}})[\bs{\zeta}_1, \bs{\zeta}_2] &= (\bs{\zeta}_1, \bs{\zeta}_2) - c(\bs{\zeta}_1, \bs{\zeta}_2, \bar{\bs{\varphi}}) - c(\bs{\zeta}_2, \bs{\zeta}_1, \bar{\bs{\varphi}}), \label{Luu}\\
		\cl{L}_{\bs{u} \bs{y}}(\bar{\bs{w}}, \bar{\bs{\Lambda}})[\bs{\zeta}_1, \bs{\mu}_1] &= - (\nu_T(\bar{T}) \mu^{T_1} \nabla \bs{\zeta}_1, \nabla \bar{\bs{\varphi}}) - c_{\bs{y}}(\bs{\zeta_1}, \bs{\mu}_1, \bar{\bs{\eta}}), \label{Luy} \\
		\cl{L}_{\bs{y} \bs{u}}(\bar{\bs{w}}, \bar{\bs{\Lambda}})[\bs{\mu}_1, \bs{\zeta}_1] &= \cl{L}_{\bs{u} \bs{y}}(\bar{\bs{w}}, \bar{\bs{\Lambda}})[\bs{\zeta}_1, \bs{\mu}_1], \label{Lyu} \\	
		\cl{L}_{\bs{y} \bs{y}}(\bar{\bs{w}}, \bar{\bs{\Lambda}})[\bs{\mu}_1, \bs{\mu}_2] &= (\bs{\mu}_1, \bs{\mu}_2) - (\nu_{TT}(\bar{T}) \mu^{T_1} \mu^{T_2} \nabla \bar{\bs{u}}, \nabla \bar{\bs{\varphi}}) + (\bs{F}_{\bs{y} \bs{y}}(\bar{\bs{y}}) (\bs{\mu}_{1} \otimes\bs{\mu}_{2}), \bar{\bs{\varphi}}), \label{Lyy}\\
		\cl{L}_{\bs{U} \bs{U}}(\bar{\bs{w}}, \bar{\bs{\Lambda}})[\bs{h}_1, \bs{h}_2] &= \lambda (\bs{h}_1, \bs{h}_2), \label{LUU}
	\end{align}
	where  $\bs{\mu}_i := (\mu^{T_{i}}, \mu^{S_i})$ and the following estimates follow:
	\begin{align}|
		| \cl{L}_{\bs{u} \bs{u}}(\bar{\bs{w}}, \bar{\bs{\Lambda}})[\bs{\zeta}_1, \bs{\zeta}_2]| &\leq \norm{\bs{\zeta}_1}_{0,\Omega} \norm{\bs{\zeta}_2}_{0,\Omega} + 2 \; C_{6_d} C_{3_d} \norm{\bs{\zeta}_2}_{1,\Omega} \norm{\bs{\zeta}_1}_{1,\Omega} \norm{\bar{\bs{\varphi}}}_{1,\Omega}, \nonumber\\
		&\leq \max \left\{1, 2 \; C_{6_d} C_{3_d} M_{\bs{\varphi}}\right\} \norm{\bs{\zeta}_1}_{1,\Omega} \norm{\bs{\zeta}_2}_{1,\Omega}, \label{CtyLuu}\\
		|\cl{L}_{\bs{u} \bs{y}}(\bar{\bs{w}}, \bar{\bs{\Lambda}})[\bs{\zeta}_1, \bs{\mu}_1]| &\leq C_{\nu_T} C_{gn} C_{p_{2_d}} M \norm{\bs{\mu}_1}_{1,\Omega} \norm{\bs{\zeta}_1}_{1,\Omega} + 2 \; C_{6_d} C_{3_d} \norm{\bs{\zeta}_1}_{1,\Omega} \norm{\bs{\mu}_1}_{1,\Omega} \norm{\bar{\bs{\eta}}}_{1,\Omega}, \nonumber\\
		&\leq \max\{C_{\nu_T} C_{gn} C_{p_{2_d}} M, 2 \; C_{6_d} C_{3_d} M_{\bs{\eta}}\} \norm{\bs{\zeta}_1}_{1,\Omega} \norm{\bs{\mu}_1}_{1,\Omega}, \label{Cty_Luy}\\
		|(\bs{F}_{\bs{y} \bs{y}}(\bar{\bs{y}}) (\bs{\mu}_1 \otimes \bs{\mu}_2), \bar{\bs{\varphi}})| &\leq \|F_{\bs{y} \bs{y}} (\bar{\bs{y}})\|_{\infty,\Omega} \norm{\bs{\mu}_1}_{L^4(\Omega)} \norm{\bs{\mu}_2}_{L^4(\Omega)} \norm{\bar{\bs{\varphi}}}_{0,\Omega} \nonumber\\
		&\leq C_{F_{\bs{y}  \bs{y}}} C_{4_d}^2 M_{\bs{\varphi}} \norm{\bs{\mu}_1}_{1,\Omega} \norm{\bs{\mu}_2}_{1,\Omega}, \label{Cty_Fyy}\\
		|(\nu_{TT}(\bar{T}) \mu^{T_1} \mu^{T_2} \nabla \bar{\bs{u}}, \nabla \bar{\bs{\varphi}})| &\leq \|\nu_{TT}(\bar{T})\|_{\infty,\Omega} \norm{\nabla \bar{\bs{u}}}_{L^{r_1}(\Omega)} \|\mu^{T_1}\|_{L^{r_2}(\Omega)} \|\mu^{T_2}\|_{L^{r_3}(\Omega)} \norm{\nabla \bar{\bs{\varphi}}}_{L^{r_4}(\Omega)} \nonumber\\
		&\leq\begin{cases}
			C_{\nu_{TT}} C_{gn} C_{r_{2_d}} C_{r_{3_d}} \norm{\bar{\bs{u}}}_{H^{3/2 + \delta}(\Omega)} \|\mu^{T_1}\|_{1,\Omega} 
			\|\mu^{T_2}\|_{1,\Omega} \norm{\bar{\bs{\varphi}}}_{1,\Omega} \; \mbox{in 2D}, \label{Cty_nuTT} \\
			C_{\nu_{TT}} \norm{\bar{\bs{u}}}_{3/2+\delta,\Omega} \|\mu^{T_1}\|_{1,\Omega} \|\mu^{T_2}\|_{1,\Omega} \norm{\bar{\bs{\varphi}}}_{3/2+\delta,\Omega} \; \mbox{in 3D} ,
		\end{cases} 
	\end{align} 		
	where the last inequality is tackled by choosing $r_1 = r_2 = r_3 = 6$  and $r_4 = 2$ in two dimensions and applying \eqref{FractionalGagliardoNirenberg} with $r_1 = r_4 = 3$ and $r_2 = r_3 = 6$ in three dimensions. Similar to the previous case, one can see that the mappings $\bar{\bs{u}} \mapsto \bs{\cl{L}_{u u}}, \bar{\bs{y}} \mapsto \bs{\cl{L}_{y y}}$ and $\bar{\bs{U}} \mapsto \bs{\cl{L}_{U U}}$ are twice Fr\'echet-differentiable and thrice Fr\'echet-differentiability of $\bs{\cl{L}}$ is implied. The third order derivatives of Lagrangian are,
	\begin{align}
		\bs{\cl{L}_{u y y}}(\bar{\bs{w}}, \bar{\bs{\Lambda}}) [\bs{\zeta}_1, \bs{\mu}_1, \bs{\mu}_2] &= - (\nu_{TT}(\bar{T}) \mu^{T_1} \mu^{T_2} \nabla \bs{\zeta}_1, \nabla \bar{\bs{\varphi}}) \\
		\bs{\cl{L}_{y y y}}(\bar{\bs{w}}, \bar{\bs{\Lambda}}) [\bs{\mu}_1, \bs{\mu}_2, \bs{\mu}_3] &= (\bs{F}_{\bs{y} \bs{y} \bs{y}}(\bar{\bs{y}}) (\bs{\mu}_1 \otimes\bs{\mu}_2\otimes \bs{\mu}_3), \bar{\bs{\varphi}}) - (\nu_{TTT}(\bar{T}) \mu^{T_1} \mu^{T_2} \mu^{T_3} \nabla \bar{\bs{u}}, \nabla \bar{\bs{\varphi}}),
	\end{align}
	
	and their boundedness follows from \eqref{Cty_nuTT}, Remark \ref{adjointRegularity}, Lemma \ref{EnergyEstAdjoint} and,
	\begin{align}
		|(\bs{F}_{\bs{y} \bs{y} \bs{y}}(\bar{\bs{y}}) (\bs{\mu}_1 \otimes\bs{\mu}_2\otimes \bs{\mu}_3), \bar{\bs{\varphi}})| &\leq \|F_{\bs{y} \bs{y} \bs{y}} (\bar{\bs{y}})\|_{\infty,\Omega}
		\norm{\bs{\mu}_3}_{L^6(\Omega)} \norm{\bs{\mu}_2}_{L^6(\Omega)} \norm{\bs{\mu}_1}_{L^6(\Omega)} \norm{\bar{\bs{\varphi}}}_{0,\Omega} \nonumber\\
		&\leq C_{F_{\bs{y} \bs{y} \bs{y}}} M_{\bs{\varphi}} \norm{\bs{\mu}_1}_{1,\Omega} \norm{\bs{\mu}_2}_{1,\Omega} \norm{\bs{\mu}_3}_{1,\Omega}, \label{ToD_eq1}\\
		|(\nu_{TTT}(\bar{T})  \mu_{T_1} \mu_{T_2} \mu_{T_3} \nabla \bar{\bs{u}}, \nabla \bar{\bs{\varphi}})| &\leq \|\nu_{TTT}(\bar{T})\|_{\infty,\Omega} \norm{\nabla \bar{\bs{u}}}_{L^{r_1}(\Omega)} \|\mu^{T_1}\|_{L^{r_2}(\Omega)} 
		\|\mu^{T_2}\|_{L^{r_3}(\Omega)}  \|\mu^{T_3}\|_{L^{r_4}(\Omega)}\norm{\nabla \bar{\bs{\varphi}}}_{L^{r_5}(\Omega)} \nonumber\\
		&\leq\begin{cases}
			C_{\nu_{TTT}} C_{gn}  \norm{\bar{\bs{u}}}_{{3/2 + \delta},\Omega} \|\mu^{T_1}\|_{1,\Omega} \|\mu^{T_2}\|_{1,\Omega} \|\mu^{T_3}\|_{1,\Omega} \norm{\bar{\bs{\varphi}}}_{1,\Omega}  \; \mbox{in 2D} ,\\
			C_{\nu_{TTT}} C_{gn} \norm{\bs{\bar{u}}}_{3/2+\delta,\Omega} \|\mu^{T_1}\|_{1,\Omega} \|\mu^{T_2}\|_{1,\Omega} 
			\|\mu^{T_3}\|_{1,\Omega} \norm{\bar{\bs{\varphi}}}_{3/2+\delta,\Omega} \; \mbox{in 3D},
			&\end{cases}
	\end{align}
	where the last inequality holds true for $r_i = 8$ for $i = 1, \dots, 4$ and $r_5 = 2$ in two dimensions. In three dimensions, it can be bounded similar to \eqref{Cty_nuTT} by choosing $r_1 = r_5 = 4$ and $r_2 = r_3 = r_ 4 = 6,$ for $\delta \in [\frac{1}{4}, \frac{1}{2}).$
\end{proof}

Let $\bar{\bs{w}}= (\bar{\bs{u}}, \bar{\bs{y}}, \bar{\bs{U}})$ be a fixed admissible optimal triplet that satisfies the first order necessary optimality conditions. We define the set of strongly active constraints 
\begin{definition}[Strongly active sets]
	For a fixed $\varepsilon > 0$ and all $i = 1,\ldots , n,$
	$$\Omega_{\varepsilon,i} := \left\{x \in \Omega : |\lambda \bar{U}_i(x) + \bar{\varphi}_i(x)| > \varepsilon\right\},$$
	where, $v_{i}(x)$ denotes the $i^{th}$ component of a vector function $\bs{v}=(v_1,\ldots,v_d)$ at $x \in \Omega$. For $\bs{U} \in \bs{L}^p(\Omega)$ and $1 \leq p < \infty$ we define the $L^p$-norm with respect to the set of positivity by
	$$\norm{\bs{U}}_{L^{p,\Omega_{\varepsilon}}} := \left(\sum_{i=1}^{n} \norm{U_i}^{p}_{L^{p}(\Omega_{\varepsilon,i})}\right)^{\frac{1}{p}}.$$
\end{definition}
Then for all $\bs{U} \in \bs{\cl{U}}_{ad},$ the following holds (see  \cite[Corollary 3.14]{WachsmuthSSC} for a proof):
\begin{align}\label{StronglyActiveEstimate}
	\sum_{i=1}^{n} \int_{\Omega_{\varepsilon,i}} \left(\lambda \bar{U}_i(x) + \bar{\varphi}_i(x)\right) \left(U_i(x) -\bar {U}_i(x)\right) \geq \varepsilon \norm{\bs{U} - \bar{\bs{U}}}_{L^{1,\Omega_{\varepsilon}}}.
\end{align}
We abbreviate $[\bs{v},\bs{v}] = [\bs{v}]^2$ and $[\bs{v},\bs{v},\bs{v}] = [\bs{v}]^3$. Assume that the optimal triplet $\bar{\bs{w}}$ and the associated adjoint pair $\bar{\bs{\Lambda}}$ satisfy the following coercivity assumption or the second-order sufficient optimality condition:
\begin{align}\label{SSC}\tag{SSC}
	\left\{
	\begin{aligned}
		&\exists \;\; \varepsilon > 0 \;\; \mbox{and} \;\;  \sigma > 0 \;\; \mbox{such that}\\
		&\hspace{3.5cm}\bs{\cl{L}}_{\bs{w} \bs{w}} (\bar{\bs{w}}, \bar{\bs{\Lambda}}) \left[\left(\bs{\zeta}, \bs{\mu}, \bs{h}\right)\right]^2 \geq \sigma \norm{\bs{h}}^2_{L^{r}(\Omega)} \\
		&\mbox{holds for all} \;\; (\bs{\zeta}, \bs{\mu}, \bs{h}) \in \bs{X} \times \left[H_0^1(\Omega)\right]^2 \times \bs{L}^2(\Omega) \; \mbox{with} \\	&\hspace{2cm}\bs{h} = \bs{U} - \bar{\bs{U}} \;\;\forall\;\; \bs{U} \in \bs{\cl{U}}_{ad}, \bs{h}_i = 0 \;\; \mbox{on} \;\; \Omega_{\varepsilon,i} \; \mbox{for} \;\; i = 1,\ldots , n\\
		&\mbox{and} \; (\bs{\zeta}, \bs{\mu}) \in \bs{X} \times \left[H^1_0(\Omega) \right]^2, \; \mbox{where} \; \bs{\mu} := \left(\mu^T, \mu^S \right), \; \mbox{solves the linearized equation} \\
		&a(\bar{\bs{y}}; \bs{\zeta}, \bs{v}) + c(\bs{\zeta}, \bar{\bs{u}}, \bs{v}) + c(\bar{\bs{u}}, \bs{\zeta}, \bs{v}) + ((\nu_{T} (\bar{T})) \mu^T \; \nabla \bar{\bs{u}}, \nabla \bs{v}) - ((\bs{F}_{\bs{y}}(\bar{\bs{y}})) \bs{\mu}, \bs{v})  = (\bs{h}, \bs{v}) ,\\
		&\hspace{6cm}a_{\bs{y}}(\bs{\mu}, \bs{s}) + c_{\bs{y}}(\bar{\bs{u}}, \bs{\mu}, \bs{s}) + c_{\bs{y}}(\bs{\zeta}, \bar{\bs{y}}, \bs{s}) =  0,\\
		&\hspace{10cm}  \forall \; (\bs{v}, \bs{s}) \in \bs{X} \times \left[H^1_0(\Omega)\right]^2. 
	\end{aligned}
	\right.
\end{align}
 Now we prove that \eqref{SSC} and first-order necessary conditions together are sufficient for local optimality of $\bar{\bs{w}}$. In particular, we are able to show quadratic growth of the cost functional with respect to $\bs{L}^{\rparam}$-norm in a $\bs{L}^{\sparam}$-neighbourhood of the reference control.

\begin{theorem}[Growth rate of the cost functional]\label{SSC_thm1}
	Let $\bar{\bs{w}}=(\bar{\bs{u}}, \bar{\bs{y}}, \bar{\bs{U}})$ be the admissible triplet for our optimal control problem which satisfies the first-order necessary optimality condition with corresponding adjoint state $\bar{\bs{\Lambda}}$ given in Theorem \ref{Fopt_AdjointEqn}, such that Lemma \ref{LagrangianFrechetDifferentiability} holds. Assume further that \eqref{SSC} is satisfied at $\bar{\bs{w}}$. Then there exists $\vartheta > 0$ and $\rho > 0$ such that
	$$J(\bs{w}) \geq J(\bar{\bs{w}}) + \vartheta \norm{\bs{U} - \bar{\bs{U}}}^2_{L^{\rparam}(\Omega)},$$
	holds for all admissible triplets $\bs{w}=({\bs{u}}, {\bs{y}}, {\bs{U}})$ with  $\norm{\bs{U} - \bar{\bs{U}}}_{L^{\sparam}(\Omega)} \leq \rho$. 
\end{theorem}
\begin{proof}
	Suppose $\bar{\bs{w}}=(\bar{\bs{u}}, \bar{\bs{y}}, \bar{\bs{U}})$ satisfies the assumptions of the theorem. Let $\bs{w}=({\bs{u}}, {\bs{y}}, {\bs{U}})$ be another admissible triplet. Then we have
	\begin{align}\label{SSC_thm1_eq1}
		J(\bar{\bs{w}}) = \bs{\cl{L}}(\bar{\bs{w}}, \bar{\bs{\Lambda}}) \;\; \mbox{and} \;\; J(\bs{w}) = \bs{\cl{L}}(\bs{w}, \bar{\bs{\Lambda}}).
	\end{align}
	Taylor-expansion of the Lagrangian around $\bar{\bs{w}}$ yields
	\begin{align}
		\bs{\cl{L}}(\bs{w}, \bar{\bs{\Lambda}}) &= \bs{\cl{L}}(\bar{\bs{w}}, \bar{\bs{\Lambda}}) + \bs{\cl{L}}_{\bs{u}}(\bar{\bs{w}}, \bar{\bs{\Lambda}})(\bs{u} - \bar{\bs{u}}) +  \bs{\cl{L}}_{\bs{y}}(\bar{\bs{w}}, \bar{\bs{\Lambda}})(\bs{y}-\bar{\bs{y}})+ \bs{\cl{L}}_{\bs{U}}(\bar{\bs{w}}, \bar{\bs{\Lambda}})(\bs{U} - \bar{\bs{U}})  \label{TaylorExpansion}\\
		&~~~~~~+\frac{1}{2}\bs{\cl{L}}_{\bs{w} \bs{w}}(\bar{\bs{w}}, \bar{\bs{\Lambda}}) \left[\bs{w} - \bar{\bs{w}}\right]^2 + R_{\bs{\cl{L}}}.  \nonumber
	\end{align} 
	Due the nature of nonlinearities that present themselves in the diffusion coefficient and the forcing terms, we need to estimate the following higher order remainder term $R_{\bs{\cl{L}}}$ (see \cite[Theorem 7.9-1]{ciarlet2013linear}), 
	$$R_{\bs{\cl{L}}} := \frac{1}{3} \int_{0}^1  (1 - \varTheta) \bs{\cl{L}}_{\bs{w} \bs{w} \bs{w}} (\bar{\bs{w}} + \varTheta (\bs{w} - \bar{\bs{w}}),\bar{\bs{\Lambda}}) \left[\bs{w} - \bar{\bs{w}}\right]^3 \ d\varTheta.$$
	Furthermore, we can verify that the first-order necessary conditions are equivalently expressed by
	\begin{align}\label{Fonoc_}
		\begin{aligned}
			\bs{\cl{L}_u}(\bar{\bs{w}}, \bar{\bs{\Lambda}}) \bs{\zeta} &= 0 \;\; \forall \;\; \bs{\zeta} \in \bs{X}, \;\;
			\bs{\cl{L}_y}(\bar{\bs{w}}, \bar{\bs{\Lambda}}) \bs{\mu} = 0 \;\; \forall \;\; \bs{\mu} \in \left[H_0^1(\Omega)\right]^2, \\
			&\bs{\cl{L}_U}(\bar{\bs{w}}, \bar{\bs{\Lambda}}) (\bs{U} - \bar{\bs{U}}) \geq 0 \;\; \forall \;\; \bs{U} \in \bs{\cl{U}}_{ad}, 
		\end{aligned}
	\end{align}
	and are satisfied at $\bar{\bs{w}}$ with the adjoint state $\bar{\bs{\Lambda}}$. Using the above representation, the second and third terms vanish in \eqref{TaylorExpansion} and the fourth term is non-negative. Infact, using \eqref{StronglyActiveEstimate} we have the following estimate for the fourth term in \eqref{TaylorExpansion} on the subspace $\Omega_{\varepsilon,i}$,
	\begin{align}\label{SSC_thm1_eq18}
		\bs{\cl{L}}_{\bs{U}}(\bar{\bs{w}}, \bar{\bs{\Lambda}})(\bs{U} - \bar{\bs{U}}) \geq \varepsilon \norm{\bs{U} - \bar{\bs{U}}}_{L^{1,\Omega_{\varepsilon}}}.
	\end{align}
	Next, we investigate the second derivative of $\bs{\cl{L}}$. Invoking \eqref{SSC} on the subspace $\Omega_{\varepsilon,i}$, let us construct a new admissible control $\tilde{\bs{U}} \in \bs{L}^{\sparam}(\Omega)$ for $i = 1, ... , n$ by
	\begin{align*}
		\tilde{U}_i(x) = 
		\begin{cases}
			\bar{U}_i(x) \;\;\; \mbox{on} \;\;\; \Omega_{\varepsilon,i} \\
			U_i(x) \;\;\; \mbox{on} \;\;\; \Omega \backslash \Omega_{\varepsilon,i}.
		\end{cases}
	\end{align*}
	Then $\bs{U} - \bar{\bs{U}} = (\bs{U} - \tilde{\bs{U}}) + (\tilde{\bs{U}} - \bar{\bs{U}})$, 
	such that $\bs{h} := \tilde{\bs{U}} - \bar{\bs{U}} = 0$ on $\Omega_{\varepsilon,i}$ which is the assumption of \eqref{SSC}. The difference $(\bs{\zeta}, \bs{\mu}) = (\bs{u} - \bar{\bs{u}}, \bs{y} - \bar{\bs{y}})$ solves for all $(\bs{v}, \bs{s}) \in \bs{X} \times \left[H^1_0(\Omega)\right]^2$ the following equations,
	\begin{align*}
		\left\{
		\begin{aligned}
			& a(\bar{\bs{y}}; \bs{\zeta}, \bs{v}) + c(\bs{\zeta}, \bar{\bs{u}}, \bs{v}) + c(\bar{\bs{u}}, \bs{\zeta}, \bs{v}) 
			+ ((\nu_{T} (\bar{T})) \mu^T \; \nabla \bar{\bs{u}}, \nabla \bs{v}) - ((\bs{F}_{\bs{y}}(\bar{\bs{y}})) \bs{\mu}, \bs{v}) = (\bs{U} - \bar{\bs{U}}, \bs{v})_{r,r*} - \langle R_{\bs{u}}, \bs{v} \rangle, \nonumber\\
			&a_{\bs{y}}(\bs{\mu}, \bs{s}) + c_{\bs{y}}(\bar{\bs{u}}, \bs{\mu}, \bs{s}) + c_{\bs{y}}(\bs{\zeta}, \bar{\bs{y}}, \bs{s}) = \langle R_{\bs{y}}, \bs{s} \rangle,
		\end{aligned}
		\right.
	\end{align*}
	where the respective reminder terms $\langle R_{\bs{u}}, \bs{v} \rangle$ and $\langle R_{\bs{y}}, \bs{s} \rangle$ consist of quadratic and mixed higher-order remainder terms due to the inherent nonlinear nature of the coupled problem and are defined for $\Theta \in (0,1),$ respectively as follows:
	\begin{align*}
		&\langle R_{\bs{u}}, \bs{v} \rangle := c(\bs{u} - \bar{\bs{u}}, \bs{u} - \bar{\bs{u}}, \bs{v}) + \frac{1}{2!} \int_{0}^{1} (1 - \Theta)(F_{\bs{y}\bs{y}} (\bs{y} + \Theta (\bs{y} - \bar{\bs{y}})) (\bs{y} - \bar{\bs{y}})^2, \bs{v}) \;d \Theta \\
		&~~~~~~~~~- \frac{1}{2!}(\nu_{TT}(T + \Theta (T - \bar{T})) (T - \bar{T})^2 \nabla \bar{\bs{u}}, \nabla \bs{v}) - (\nu_T(T + \Theta (T - \bar{T})) (T - \bar{T}) \nabla (\bs{u} - \bar{\bs{u}}), \nabla \bs{v}), \\
		&\langle R_{\bs{y}}, \bs{s} \rangle := c_{\bs{y}}(\bs{u} - \bar{\bs{u}}, \bs{y} - \bar{\bs{y}}, \bs{s}).
	\end{align*} 
	In order to apply \eqref{SSC}, we split $\bs{\zeta}=\bs{u} - \bar{\bs{u}}$ and $\bs{\mu}= \bs{y} - \bar{\bs{y}}$ as follows:
	$\bs{\zeta} = \bs{u}_{h} + \bs{u}_{r} \;\; \mbox{and} \;\; \bs{\mu} = \bs{y}_{h} + \bs{y}_{r}.$
	Here, $(\bs{u}_h,\bs{y}_h)$ with $\bs{y}_h := (T_h, S_h)$ and $(\bs{u}_r,\bs{y}_r)$ with $\bs{y}_r := (T_r, S_r)$ for all $(\bs{v}, \bs{s}) \in \bs{X} \times \left[H_0^1(\Omega)\right]^2$  solve,
	\begin{align}\label{SSC_thm1_eq2}
		\left\{
		\begin{aligned}
			& a(\bar{\bs{y}}; \bs{u}_h, \bs{v}) + c(\bs{u}_h, \bar{\bs{u}}, \bs{v}) + c(\bar{\bs{u}}, \bs{u}_h, \bs{v}) + ((\nu_{T} (\bar{T})) T_{h} \; \nabla \bar{\bs{u}}, \nabla \bs{v}) - ((\bs{F}_{\bs{y}}(\bar{\bs{y}})) \bs{y}_h, \bs{v}) = (\bs{h}, \bs{v})_{r,r*}, \\
			&a_{\bs{y}}(\bs{y}_h, \bs{s}) + c_{\bs{y}}(\bar{\bs{u}}, \bs{y}_h, \bs{s}) + c_{\bs{y}}(\bs{u}_h, \bar{\bs{y}}, \bs{s}) = 0,
		\end{aligned}
		\right.
	\end{align}
	and
	\begin{align}\label{SSC_thm1_eq3}
		\left\{
		\begin{aligned}
			&a(\bar{\bs{y}}; \bs{u}_r, \bs{v}) + c(\bs{u}_r, \bar{\bs{u}}, \bs{v}) + c(\bar{\bs{u}}, \bs{u}_r, \bs{v}) \\
			&\hspace{2.5cm}+ ((\nu_{T} (\bar{T})) T_{r} \; \nabla \bar{\bs{u}}, \nabla \bs{v}) - ((\bs{F}_{\bs{y}}(\bar{\bs{y}})) \bs{y}_r, \bs{v}) = (\bs{U} - \tilde{\bs{U}}, \bs{v})_{r,r*} - \langle R_{\bs{u}}, \bs{v} \rangle, \\
			&a_{\bs{y}}(\bs{y}_r, \bs{s}) + c_{\bs{y}}(\bar{\bs{u}}, \bs{y}_r, \bs{s}) + c_{\bs{y}}(\bs{u}_r, \bar{\bs{y}}, \bs{s}) = \langle R_{\bs{y}}, \bs{s} \rangle,
		\end{aligned}
		\right.
	\end{align}	
	respectively. Equations \eqref{SSC_thm1_eq2} are linear and the auxiliary triplet $(\bs{u}_h, \bs{y}_h, \bs{h})$ belongs to the subspace  where \eqref{SSC} applies. Using Lemma \ref{EnergyEstLinearized} by taking $\langle \bs{\hat{f}}, \bs{v} \rangle = (\bs{h}, \bs{v})_{\rparam,\rdual}$ and $\langle \bs{\tilde{f}}, \bs{s} \rangle = 0$, we obtain the following estimates for the auxiliary states:
	\begin{align}\label{SSC_thm1_eq4}
		\begin{aligned}
			&\norm{\bs{u}_h}_{1,\Omega} + \norm{\bs{y}_h}_{1,\Omega} \leq C \norm{\bs{h}}_{L^{\rparam}(\Omega)} \leq C \left(\|\tilde{\bs{U}} - \bs{U}\|_{L^{\rparam}(\Omega)} + \|\bs{U} - \bar{\bs{U}}\|_{L^{\rparam}(\Omega)}\right).
		\end{aligned}
	\end{align}
	To bound the remainder terms,, we follow the same steps as in the proof of Lemma \ref{EnergyEstLinearized} and apply \eqref{stabilityH1}, to get
	\begin{align*}
		|c(\bs{u} - \bar{\bs{u}}, \bs{u} - \bar{\bs{u}}, \bs{v})| &\leq C_{6_d} C_{3_d} \norm{\bs{u} - \bar{\bs{u}}}^2_{1,\Omega} \norm{\bs{u}_r}_{1,\Omega} \leq C \|\bs{U} - \bar{\bs{U}}\|^2_{L^{\rparam}(\Omega)} \norm{\bs{v}}_{1,\Omega}, \\
		|(\nu_{TT} \left(\bar{T} + \Theta (T - \bar{T})\right) (T - \bar{T})^2 \nabla \bar{\bs{u}}, \nabla \bs{v})| &\leq C_{\nu_{TT}} C_{gn} C_{r_{2_d}} C_{r_{3_d}} M \|\bs{y} - \bar{\bs{y}}\|^2_{3/2+\delta,\Omega} \norm{\bs{v}}_{1,\Omega} \\
		&\leq C \|\bs{U} - \bar{\bs{U}}\|^2_{L^{\rparam}(\Omega)} \norm{\bs{v}}_{1,\Omega}, \\
		|(F_{\bs{y} \bs{y}} \left(\bar{\bs{y}} + \Theta (\bs{y} - \bar{\bs{y}})\right) \left(\bs{y} - \bar{\bs{y}})^2, \bs{v}\right)| &\leq C_{F_{\bs{y} \bs{y}}} C_{4_d}^2 \norm{\bs{y} - \bar{\bs{y}}}^2_{1,\Omega} \norm{\bs{v}}_{1,\Omega} \leq C \|\bs{U} - \bar{\bs{U}}\|^2_{L^{\rparam}(\Omega)} \norm{\bs{v}}_{1,\Omega},\\
		|(\nu_T(T + \Theta (T - \bar{T})) (T - \bar{T}) \nabla (\bs{u} - \bar{\bs{u}}), \nabla \bs{v})| &\leq C_{\nu_T} C_{6_d} C_{gn}  \norm{\bs{y} - \bar{\bs{y}}}_{1,\Omega} \norm{\bs{u} - \bar{\bs{u}}}_{3/2+\delta,\Omega} \norm{\bs{v}}_{1,\Omega}\\
		&\leq C \|\bs{U} - \bar{\bs{U}}\|^{2}_{L^{\rparam}(\Omega)} \norm{\bs{v}}_{1,\Omega}.
	\end{align*}
	Thus,
	\begin{align}\label{SSC_thm1_eq6}
		|\langle R_{\bs{u}}, \bs{v} \rangle|  \leq C \|\bs{U} - \bar{\bs{U}}\|^2_{L^{\rparam}(\Omega)}  \norm{\bs{v}}_{1,\Omega} .
	\end{align}
	Similarly, we can get the following estimates for $R_{\bs{y}}$:
	\begin{align}\label{SSC_thm1_eqRy}
		|\langle R_{\bs{y}}, \bs{s} \rangle| \leq C_{6_d} C_{3_d} \norm{\bs{u} - \bar{\bs{u}}}_{1,\Omega} \norm{\bs{y} - \bar{\bs{y}}}_{1,\Omega} \norm{\bs{s}}_{1,\Omega} \leq C \|\bs{U} - \bar{\bs{U}}\|^2_{L^{\rparam}(\Omega)} \norm{\bs{s}}_{1,\Omega}.
	\end{align}
	Following the steps analogous to the proof of  Lemma \ref{EnergyEstLinearized}, we can get the following estimates for the auxiliary states solving \eqref{SSC_thm1_eq3}:
	\begin{align}\label{SSC_thm1_eq7}
		\norm{\bs{u}_r}_{1,\Omega} + \norm{\bs{y}_r}_{1,\Omega} &\leq C\big(\|\tilde{\bs{U}} - \bs{U}\|_{L^{\rparam}(\Omega)} +\norm{R_{\bs{u}}}_{\bs{X}^*}+\norm{R_{\bs{y}}}_{-1,\Omega}\big)\nonumber\\&\leq C \big(\|\tilde{\bs{U}} - \bs{U}\|_{L^{\rparam}(\Omega)} + \|\bs{U} - \bar{\bs{U}}\|^2_{L^{\rparam}(\Omega)} \big).
	\end{align}
	Now we investigate the second order derivative of the Lagrangian.  Denoting by $\bs{w}^h = (\bs{u}_h, \bs{y}_h, \bs{h}),$ we can easily derive the following identities:
	\begin{align}
		\bs{\cl{L}}_{\bs{w} \bs{w}}(\bar{\bs{w}}, \bar{\bs{\Lambda}}) \big[\bs{w} - \bar{\bs{w}}\big]^2 &= \bs{\cl{L}}_{\bs{u} \bs{u}}(\bar{\bs{w}}, \bar{\bs{\Lambda}}) \big[\bs{u}_h + \bs{u}_r\big]^2 + \bs{\cl{L}}_{\bs{y} \bs{y}}(\bar{\bs{w}}, \bar{\bs{\Lambda}}) \big[\bs{y}_h + \bs{y}_r\big]^2 +
		\bs{\cl{L}}_{\bs{U} \bs{U}}(\bar{\bs{w}}, \bar{\bs{\Lambda}}) \big[\bs{U} - \tilde{\bs{U}} + \bs{h}\big]^2\nonumber\\
		&\quad+ 2 \bs{\cal{L}}_{\bs{uy}}(\bar{\bs{w}}, \bar{\bs{\Lambda}}) \big[ u_h + u_r, y_h + y_r\big] \nonumber\\
		&= \bs{\cl{L}}_{\bs{w} \bs{w}}(\bar{\bs{w}}, \bar{\bs{\Lambda}}) \left[\bs{w}^h\right]^2 + \bs{\cl{L}}_{\bs{u} \bs{u}}(\bar{\bs{w}}, \bar{\bs{\Lambda}}) \left[\bs{u}_r\right]^2 + 
		2 \bs{\cl{L}}_{\bs{u} \bs{u}}(\bar{\bs{w}}, \bar{\bs{\Lambda}}) \big[\bs{u}_h, \bs{u}_r\big] \nonumber\\ 
		&~~~+ \bs{\cl{L}}_{\bs{y} \bs{y}}(\bar{\bs{w}}, \bar{\bs{\Lambda}}) \left[\bs{y}_r\right]^2 + 2 \; \bs{\cl{L}}_{\bs{y} \bs{y}}(\bar{\bs{w}}, \bar{\bs{\Lambda}}) \left[\bs{y}_h, \bs{y}_r\right] + \bs{\cl{L}}_{\bs{U} \bs{U}}(\bar{\bs{w}}, \bar{\bs{\Lambda}}) \big[\bs{U} - \tilde{\bs{U}}\big]^2 \nonumber\\
		&~~~+ 2 \bs{\cl{L}}_{\bs{U} \bs{U}}(\bar{\bs{w}}, \bar{\bs{\Lambda}}) \big[\bs{U} - \tilde{\bs{U}}, \bs{h}\big] + 2 \bs{\cal{L}}_{\bs{uy}}(\bar{\bs{w}}, \bar{\bs{\Lambda}}) \big[ u_h,y_r\big] + 2 \bs{\cal{L}}_{\bs{uy}}(\bar{\bs{w}}, \bar{\bs{\Lambda}}) \big[ u_r, y_h \big] \nonumber\\ 
		&~~~+2 \bs{\cal{L}}_{\bs{uy}}(\bar{\bs{w}}, \bar{\bs{\Lambda}}) \big[u_r, y_r\big]. \label{SSC_thm1_eq9}
	\end{align}
	
	The first term is tackled by using \eqref{SSC}. The second order derivatives with respect to control satisfy
	$$\bs{\cl{L}}_{\bs{U} \bs{U}}(\bar{\bs{w}}, \bar{\bs{\Lambda}}) \big[\bs{U} - \tilde{\bs{U}}\big]^2 + 2 \bs{\cl{L}}_{\bs{U} \bs{U}}(\bar{\bs{w}}, \bar{\bs{\Lambda}}) \big[\bs{U} - \tilde{\bs{U}}, \bs{h}\big] = \lambda \|\bs{U} - \tilde{\bs{U}}\|^2_{0,\Omega} + 2 \lambda (\bs{U} - \tilde{\bs{U}}, \bs{h}).$$
	Using the definition of $\tilde{\bs{U}}$ and $\bs{h}$, we can deduce that they are orthogonal to each other, so their inner product is zero. Therefore, the following relation holds:
	\begin{align}\label{SSC_thm1_eq10}
		\bs{\cl{L}}_{\bs{U} \bs{U}}(\bar{\bs{w}}, \bar{\bs{\Lambda}}) \big[\bs{U} - \tilde{\bs{U}}\big]^2 + 2 \bs{\cl{L}}_{\bs{U} \bs{U}}(\bar{\bs{w}}, \bar{\bs{\Lambda}}) \big[\bs{U} - \tilde{\bs{U}}, \bs{h}\big] = \lambda \|\bs{U} - \bar{\bs{U}}\|^2_{0,\Omega_{\varepsilon}}\geq 0.
	\end{align}
	For our convenience, let use the following notations:
	$\tilde{\bs{z}} := \|\bs{U} - \tilde{\bs{U}}\|_{L^{\rparam}(\Omega)} \;\; \mbox{and} \;\; \bar{\bs{z}} := \|\bs{U} - \bar{\bs{U}}\|_{L^{\rparam}(\Omega)}.$
	Then the second order derivatives with respect to $\bs{y}$ in \eqref{SSC_thm1_eq9} are treated using Lemma \ref{LagrangianFrechetDifferentiability} and estimates \eqref{SSC_thm1_eq4}, \eqref{SSC_thm1_eq7}, as follows:
	\begin{align}
		|\bs{\cl{L}}_{\bs{y} \bs{y}}(\bar{\bs{w}}, \bar{\bs{\Lambda}}) \left[\bs{y}_r\right]^2 + 2 \bs{\cl{L}}_{\bs{y} \bs{y}}(\bar{\bs{w}}, \bar{\bs{\Lambda}}) \left[\bs{y}_h, \bs{y}_r\right]| &\leq C \big(\norm{\bs{y}_r}^2_{1,\Omega} + \norm{\bs{y}_r}_{1,\Omega} \norm{\bs{y}_h}_{1,\Omega}\big) \nonumber\\
		&\leq C \big(\bar{\bs{z}}^4 + \bar{\bs{z}}^3  + \bar{\bs{z}} \tilde{\bs{z}} + \tilde{\bs{z}}\bar{\bs{z}}^2  + \tilde{\bs{z}}^2\big). \label{SSC_thm1_eq11}
	\end{align}
	The same estimates hold for the second order derivatives with respect to $\bs{u}$ in \eqref{SSC_thm1_eq9} in an analogous way,
	\begin{align}
		|\bs{\cl{L}}_{\bs{u} \bs{u}}(\bar{\bs{w}}, \bar{\bs{\Lambda}}) \left[\bs{u}_r\right]^2 + 
		2 \bs{\cl{L}}_{\bs{u} \bs{u}}(\bar{\bs{w}}, \bar{\bs{\Lambda}}) \left[\bs{u}_h, \bs{u}_r\right]| &\leq C \big(\norm{\bs{u}_r}^2_{1,\Omega} + \norm{\bs{u}_r}_{1,\Omega} \norm{\bs{u}_h}_{1,\Omega}\big) \nonumber\\
		&\leq C \big(\bar{\bs{z}}^4 + \bar{\bs{z}}^3  + \bar{\bs{z}} \tilde{\bs{z}} + \tilde{\bs{z}}\bar{\bs{z}}^2 + \tilde{\bs{z}}^2\big). \label{SSC_thm1_eq12}
	\end{align}
	Similarly we estimate the mixed second order derivatives, 
	\begin{align}
		&2 \bs{\cal{L}}_{\bs{uy}}(\bar{\bs{w}}, \bar{\bs{\Lambda}}) \big[ u_h,y_r\big] + 2 \bs{\cal{L}}_{\bs{uy}}(\bar{\bs{w}}, \bar{\bs{\Lambda}}) \big[ u_r, y_h \big] + 2 \bs{\cal{L}}_{\bs{uy}}(\bar{\bs{w}}, \bar{\bs{\Lambda}}) \big[u_r, y_r\big] \nonumber\\
		&\qquad\leq C (\norm{\bs{u}_h}_{1,\Omega} \norm{\bs{y}_r}_{1,\Omega} + \norm{\bs{u}_r}_{1,\Omega} \norm{\bs{y}_h}_{1,\Omega} + \norm{\bs{u}_r}_{1,\Omega} \norm{\bs{y}_r}_{1,\Omega}) \nonumber\\
		&\qquad\leq C \big(\bar{\bs{z}}^4 + \bar{\bs{z}}^3  + \bar{\bs{z}} \tilde{\bs{z}} + \tilde{\bs{z}}\bar{\bs{z}}^2 + \tilde{\bs{z}}^2\big). \label{eq:mixed.der.bd}
	\end{align}

	Substituting \eqref{SSC_thm1_eq10}-\eqref{eq:mixed.der.bd} in \eqref{SSC_thm1_eq9} results to
	\begin{align}\label{SSC_thm1_eq!}
		\bs{\cl{L}}_{\bs{w} \bs{w}}(\bar{\bs{w}}, \bar{\bs{\Lambda}}) \left[\bs{w} - \bar{\bs{w}}\right]^2 \geq \sigma \norm{\bs{h}}^2_{L^{r}(\Omega)} - C (\bar{\bs{z}}^4 + \bar{\bs{z}}^3  + \bar{\bs{z}} \tilde{\bs{z}} + \tilde{\bs{z}}\bar{\bs{z}}^2  + \tilde{\bs{z}}^2 ).
	\end{align}
	Next we eliminate  $\bs{h}$ in \eqref{SSC_thm1_eq!} such that only terms containing $\bar{\bs{\Lambda}}$ and  $\tilde{\bs{z}}$  appear. It can be easily seen that 
	$$\|\bs{\bs{U} - \bar{\bs{U}}}\|^2_{L^{\rparam}(\Omega)} = \|\bs{U} - \tilde{\bs{U}} + \bs{h}\|^2_{L^{\rparam}(\Omega)} \leq 2 \big(\|\bs{U} - \tilde{\bs{U}}\|^2_{L^{\rparam}(\Omega)} + \norm{\bs{h}}^2_{L^{\rparam}(\Omega)}\big),$$
	which implies
	\begin{align}\label{SSC_thm1_eq13}
		\norm{\bs{h}}^2_{L^{r}(\Omega)} \geq \frac{1}{2} \bar{\bs{z}}^2 - \tilde{\bs{z}}^2 .
	\end{align}
	Invoking \eqref{SSC_thm1_eq11}, \eqref{SSC_thm1_eq12} and \eqref{SSC_thm1_eq13} in \eqref{SSC_thm1_eq!},  we get
	\begin{align}
		\frac{1}{2}\bs{\cl{L}}_{\bs{w} \bs{w}}(\bar{\bs{w}}, \bar{\bs{\Lambda}}) \left[\bs{w} - \bar{\bs{w}}\right]^2 &\geq \frac{\sigma}{4} \bar{\bs{z}}^2 - C \left(\bar{\bs{z}}^4 + \bar{\bs{z}}^3  + \bar{\bs{z}} \tilde{\bs{z}} + \tilde{\bs{z}}\bar{\bs{z}}^2 + \tilde{\bs{z}}^2\right). \label{SSC_thm1_eq14}
	\end{align}
	
	Similarly, we investigate the third order derivative which presents itself in the remainder term $R_{\bs{\cl{L}}}$ of \eqref{TaylorExpansion}. Proceeding analogously we first estimate the third order derivative with respect to $\bs{y}$ as follows:
	\begin{align}
		\bs{\cl{L}_{\bs{y} \bs{y} \bs{y}}} (\bar{\bs{w}} + \varTheta (\bs{w} - \bar{\bs{w}}), \bar{\bs{\Lambda}}) \left[\bs{y} - \bar{\bs{y}}\right]^3 &= \bs{\cl{L}_{\bs{y} \bs{y} \bs{y}}} (\bar{\bs{w}} + \varTheta (\bs{w} - \bar{\bs{w}}), \bar{\bs{\Lambda}}) \left[\bs{y}_h\right]^3 + \bs{\cl{L}_{\bs{y} \bs{y} \bs{y}}} (\bar{\bs{w}} + \varTheta (\bs{w} - \bar{\bs{w}}), \bar{\bs{\Lambda}}) \left[\bs{y}_r\right]^3 \nonumber\\
		&~~~+3 \; \bs{\cl{L}_{\bs{y} \bs{y} \bs{y}}} (\bar{\bs{w}} + \varTheta (\bs{w} - \bar{\bs{w}}), \bar{\bs{\Lambda}}) \left[\bs{y}_h, \bs{y}_h, \bs{y}_r\right] \nonumber\\
		&~~~+3 \; \bs{\cl{L}_{\bs{y} \bs{y} \bs{y}}} (\bar{\bs{w}} + \varTheta (\bs{w} - \bar{\bs{w}}), \bar{\bs{\Lambda}}) \left[\bs{y}_h, \bs{y}_r, \bs{y}_r\right]. 
	\end{align}
	Using Lemma \ref{LagrangianFrechetDifferentiability} along with \eqref{SSC_thm1_eq4} and \eqref{SSC_thm1_eq7} gives
	\begin{align*}
		|\bs{\cl{L}_{\bs{y} \bs{y} \bs{y}}} (\bar{\bs{w}} + \varTheta (\bs{w} - \bar{\bs{w}}), \bar{\bs{\Lambda}}) \left[\bs{y}_h\right]^3| &\leq C_{\bs{\cl{L}}_3} \norm{\bs{y}_h}^3_{1,\Omega} \leq C (\tilde{\bs{z}} + \bar{\bs{z}})^3, \\
		|\bs{\cl{L}_{\bs{y} \bs{y} \bs{y}}} (\bar{\bs{w}} + \varTheta (\bs{w} - \bar{\bs{w}}), \bar{\bs{\Lambda}}) \left[\bs{y}_r\right]^3| &\leq C_{\bs{\cl{L}}_3} \norm{\bs{y}_r}^3_{1,\Omega} \leq C (\tilde{\bs{z}} + \bar{\bs{z}}^2)^3, \\
		|\bs{\cl{L}_{\bs{y} \bs{y} \bs{y}}} (\bar{\bs{w}} + \varTheta (\bs{w} - \bar{\bs{w}}), \bar{\bs{\Lambda}}) \left[\bs{y}_h, \bs{y}_h, \bs{y}_r\right]| &\leq C_{\bs{\cl{L}}_3} \norm{\bs{y}_h}^2_{1,\Omega}  \norm{\bs{y}_r}_{1,\Omega} 
		\leq C \big[\left(\tilde{\bs{z}} + \bar{\bs{z}}\right)^2  \left(\tilde{\bs{z}} + \bar{\bs{z}}^2 \right)\big],  \nonumber\\
		|\bs{\cl{L}_{\bs{y} \bs{y} \bs{y}}} (\bar{\bs{w}} + \varTheta (\bs{w} - \bar{\bs{w}}), \bar{\bs{\Lambda}}) \left[\bs{y}_h, \bs{y}_r, \bs{y}_r\right]| &\leq C_{\bs{\cl{L}}_3} \norm{\bs{y}_h}_{1,\Omega}  \norm{\bs{y}_r}^2_{1,\Omega} 
		\leq C \big[\left(\tilde{\bs{z}} + \bar{\bs{z}}\right)  \left(\tilde{\bs{z}} + \bar{\bs{z}}^2\right)^2\big].  \nonumber
	\end{align*}	
	
		Similarly the  mixed third order derivative can be expanded as, 
		\begin{align*}
			&\bs{\cl{L}_{\bs{u} \bs{y} \bs{y}}} (\bar{\bs{w}} + \varTheta (\bs{w} - \bar{\bs{w}}), \bar{\bs{\Lambda}}) \left[\bs{u}-\bar{\bs{u}}, \bs{y} - \bar{\bs{y}}, \bs{y} - \bar{\bs{y}}\right] = \bs{\cl{L}_{\bs{u} \bs{y} \bs{y}}} (\bar{\bs{w}} + \varTheta (\bs{w} - \bar{\bs{w}}), \bar{\bs{\Lambda}}) \left[\bs{u_h}, \bs{y}_h, \bs{y}_h\right] \nonumber\\
			&\quad+ \bs{\cl{L}_{\bs{u} \bs{y} \bs{y}}} (\bar{\bs{w}} + \varTheta (\bs{w} - \bar{\bs{w}}), \bar{\bs{\Lambda}}) \left[\bs{u}_h, \bs{y}_h, \bs{y}_r\right] 
			+ \bs{\cl{L}_{\bs{u} \bs{y} \bs{y}}} (\bar{\bs{w}} + \varTheta (\bs{w} - \bar{\bs{w}}), \bar{\bs{\Lambda}}) \left[\bs{u}_h, \bs{y}_r, \bs{y}_r\right] \nonumber\\
			&\quad+ \bs{\cl{L}_{\bs{u} \bs{y} \bs{y}}} (\bar{\bs{w}} + \varTheta (\bs{w} - \bar{\bs{w}}), \bar{\bs{\Lambda}}) \left[\bs{u}_r, \bs{y}_h, \bs{y}_h\right] + \bs{\cl{L}_{\bs{u} \bs{y} \bs{y}}} (\bar{\bs{w}} + \varTheta (\bs{w} - \bar{\bs{w}}), \bar{\bs{\Lambda}}) \left[\bs{u}_r, \bs{y}_h, \bs{y}_r\right] \nonumber\\
			&\quad+ \bs{\cl{L}_{\bs{u} \bs{y} \bs{y}}} (\bar{\bs{w}} + \varTheta (\bs{w} - \bar{\bs{w}}), \bar{\bs{\Lambda}}) \left[\bs{u}_r, \bs{y}_r, \bs{y}_r\right]
		\end{align*}
		It is easy to see that each of the above terms can be estimated in the same way as the third order derivative with respect to $\bs{y}$.

	On further simplifying after invoking the above bounds in $R_{\bs{\cl{L}}}$, we can get the following estimate,
	\begin{align}
		R_{\bs{\cl{L}}} &\geq - C\big(\bar{\bs{z}}^{6} + \bar{\bs{z}}^5 + \bar{\bs{z}}^{3} + \tilde{\bs{z}}^2 \bar{\bs{z}}  + \tilde{\bs{z}}^2 \bar{\bs{z}}^2 + \tilde{\bs{z}} \bar{\bs{z}}^2  + \tilde{\bs{z}} \bar{\bs{z}}^4 +  \tilde{\bs{z}}^3\big). \label{SSC_thm1_eq15}
	\end{align}
	Adding \eqref{SSC_thm1_eq14} and \eqref{SSC_thm1_eq15} and applying Young's inequality to separate powers of $\bar{\bs{z}}$ and $\tilde{\bs{z}},$ we get 
	\begin{align}\label{SSC_thm1_eq16}
		\frac{1}{2}\bs{\cl{L}}_{\bs{w} \bs{w}}(\bar{\bs{w}}, \bar{\bs{\Lambda}}) \left[\bs{w} - \bar{\bs{w}}\right]^2 + R_{\bs{\cl{L}}} &\geq \frac{\sigma}{4} \bar{\bs{z}}^2  - C \big(\bar{\bs{z}}^{6} + \bar{\bs{z}}^5 +  \bar{\bs{z}}^4 + \bar{\bs{z}}^{3} + \bar{\bs{z}} \tilde{\bs{z}} + \tilde{\bs{z}}^2 \bar{\bs{z}}  + \tilde{\bs{z}}^2 \bar{\bs{z}}^2 + \tilde{\bs{z}} \bar{\bs{z}}^2  \nonumber\\
		&~~~+ \tilde{\bs{z}} \bar{\bs{z}}^4  + \tilde{\bs{z}}^2 + \tilde{\bs{z}}^3 \big)\nonumber \\
		&\geq \frac{\sigma}{8} \bar{\bs{z}}^2 - C \left(\bar{\bs{z}}^8 + \bar{\bs{z}}^6 + \bar{\bs{z}}^5 + \bar{\bs{z}}^4 + \bar{\bs{z}}^3 + \tilde{\bs{z}}^4 + \tilde{\bs{z}}^3 + \tilde{\bs{z}}^2\right) \nonumber\\
		&\geq \bar{\bs{z}}^2 \left(\frac{\sigma}{8} - C\left[\bar{\bs{z}}^6 + \bar{\bs{z}}^4 + \bar{\bs{z}}^3 + \bar{\bs{z}}^2 + \bar{\bs{z}}\right]\right) - C \left(\tilde{\bs{z}}^4 + \tilde{\bs{z}}^3 + \tilde{\bs{z}}^2\right). 
	\end{align}
	If $\bs{U}$ is sufficiently close to $\bar{\bs{U}}$, that is, $\bar{\bs{z}} \leq N_{\sparam,\rparam} \|\bs{U} - \bar{\bs{U}}\|_{L^{\sparam}(\Omega)} \leq N_{\sparam,\rparam} \; \rho_1$, then the term $$\left(\frac{\sigma}{8} - C\left[\bar{\bs{z}}^6 + \bar{\bs{z}}^4 + \bar{\bs{z}}^3 + \bar{\bs{z}}^2 + \bar{\bs{z}}\right]\right) \geq \frac{\sigma}{16}.$$ Thus we arrive at
	\begin{align}\label{SSC_thm1_eq17}
		\frac{1}{2}\bs{\cl{L}}_{\bs{w} \bs{w}}(\bar{\bs{w}}, \bar{\bs{\Lambda}}) \left[\bs{w} - \bar{\bs{w}}\right]^2 + R_{\bs{\cl{L}}} \geq \frac{\sigma}{16}	\bar{\bs{z}}^2 - C \left(\tilde{\bs{z}}^4 + \tilde{\bs{z}}^3 + \tilde{\bs{z}}^2\right).
	\end{align}
	Focusing on $\tilde{\bs{z}}$, by definition of our constructed control $\tilde{\bs{U}},$ $(U - \tilde{U})_i = 0$ on $\Omega \backslash \Omega_{\varepsilon,i}$ and $(\tilde{U} - \bar{U})_i = 0$ on $\Omega_{\varepsilon,i}.$ Hence using \eqref{ControlInterpolation},
	we conclude the following:
	\begin{align}
		\tilde{\bs{z}}^2 &\leq \|\bs{U} - \tilde{\bs{U}}\|_{L^1(\Omega)} \|\bs{U} - \tilde{\bs{U}}\|_{L^{\sparam}(\Omega)} = \norm{\bs{U} - \bar{\bs{U}}}_{L^{1,\Omega_{\varepsilon}}} \norm{\bs{U} - \bar{\bs{U}}}_{L^{\sparam,\Omega_{\varepsilon}}} \leq \rho_1 \norm{\bs{U} - \bar{\bs{U}}}_{L^{1,\Omega_{\varepsilon}}}. 
	\end{align}
	Using the above bound, \eqref{SSC_thm1_eq17}, \eqref{SSC_thm1_eq18} and \eqref{SSC_thm1_eq1} in \eqref{TaylorExpansion}, we obtain
	\begin{align*}
		J(\bs{w}) &\geq J(\bar{\bs{w}}) + \varepsilon \norm{\bs{U} - \bar{\bs{U}}}_{L^{1,\Omega_{\varepsilon}}} + \frac{\sigma}{16} \norm{\bs{U} - \bar{\bs{U}}}^2_{L^{r}(\Omega)} \\
		&\qquad- C \big(\rho_1 + \rho_1^{3/2} \norm{\bs{U} - \bar{\bs{U}}}^{1/2}_{L^{1,\Omega_{\varepsilon}}} +\rho_1^2 \norm{\bs{U} - \bar{\bs{U}}}_{L^{1,\Omega_{\varepsilon}}} +  \big) \norm{\bs{U} - \bar{\bs{U}}}_{L^{1,\Omega_{\varepsilon}}} \\
		&\geq J(\bar{\bs{w}}) + \varepsilon \norm{\bs{U} - \bar{\bs{U}}}^2_{L^{1,\Omega_{\varepsilon}}} - C \rho_2 \norm{\bs{U} - \bar{\bs{U}}}^2_{L^{1,\Omega_{\varepsilon}}} + \frac{\sigma}{16} \norm{\bs{U} - \bar{\bs{U}}}^2_{L^{r}(\Omega)} \nonumber\\
		&\geq J(\bar{\bs{w}}) + \left(\varepsilon - C \rho_2\right) \norm{\bs{U} - \bar{\bs{U}}}^2_{L^{1,\Omega_{\varepsilon}}} + \frac{\sigma}{16} \norm{\bs{U} - \bar{\bs{U}}}^2_{L^{r}(\Omega)}.
	\end{align*}
	For $\rho_2$ small enough such that $\varepsilon - C \rho_2 > 0$, the claim is proven with $\vartheta = \frac{\sigma}{16}$ and $\rho = \min (\rho_1, \rho_2)$.
\end{proof}

The following result is an immediate consequence of Theorem \ref{SSC_thm1}.

\begin{corollary}
	Let the assumptions of Theorem \ref{SSC_thm1} hold. Then $\bar{\bs{U}}$ is a locally optimal control in the sense of $\bs{L}^{\sparam}(\Omega)$.  
\end{corollary}

\begin{remark}\label{relaxing_delta}
	In three dimensions, the restriction on $\delta$ can be relaxed by choosing $(\bs{\zeta}, \bs{\mu}) \in \bs{X} \cap \bs{H}^{3/2+\delta}(\Omega) \times [H^{3/2+\delta}(\Omega)]^2.$ This is reasonable since we have already shown in Remark \ref{adjointRegularity} that the weak solution to the adjoint equation has this extra regularity, from which we can deduce that the linearized equations in \eqref{SSC} will also have this characteristic.
\end{remark}

In Theorem, \ref{SSC_thm1} we assumed that the reference control satisfies \eqref{SSC}. In the following theorem, we investigate conditions under which it can be ensured that \eqref{SSC} holds. 

\begin{theorem}\label{SSC_thm2}
	Let $\bar{\bs{w}}=(\bar{\bs{u}}, \bar{\bs{y}}, \bar{\bs{U}})$ be an admissible triplet for our optimal control problem. Suppose $\bar{\bs{w}}$ satisfies the first-order necessary optimality conditions with associated adjoint state $\bar{\bs{\Lambda}}$. Then \eqref{SSC} is fulfilled if the parameter $\lambda$ is sufficiently large or the residuals $\norm{\bar{\bs{y}} - \bs{y}_d}_{0,\Omega}$ and $\norm{\bar{\bs{u}} - \bs{u}_d}_{0,\Omega}$ are sufficiently small such that
	\begin{align}\label{SSC_thm2_eq2}
		\lambda &> \frac{1}{C_{2,r}} \big(M_{\bs{\psi}} M_{\bs{\varphi}} \big(C_{6_d} C_{3_d} + C_{F_{\bs{y} \bs{y}}} C_{4_d}^2) + C_{\nu_{TT}} \hat{M} \bar{M} M_{\bs{\psi}}\big)
	\end{align}  is satisfied.
\end{theorem}

\begin{proof}
	The second order derivative of the Lagrangian is given by
	\begin{align}
		\bs{\cl{L}}_{\bs{w} \bs{w}}(\bar{\bs{w}}, \bar{\bs{\Lambda}}) \left[(\bs{\zeta}, \bs{\mu}, \bs{h})\right]^2 &=  \norm{\bs{\zeta}}^2_{0,\Omega} - 2 \; c(\bs{\zeta}, \bs{\zeta}, \bar{\bs{\varphi}}) + \norm{\bs{\mu}}^2_{0,\Omega} - (\nu_{TT}(\bar{T}) (\mu^T)^2 \nabla \bar{\bs{u}}, \nabla \bar{\bs{\varphi}}) \label{SSC_thm2_eq1}\\
		&~~~~\hspace{1cm}+ (\bs{F}_{\bs{y} \bs{y}}(\bar{\bs{y}}) (\bs{\mu}\otimes\bs{\mu}), \bar{\bs{\varphi}}) + \lambda \norm{\bs{h}}^2_{0,\Omega}. \nonumber
	\end{align}
	The positivity of \eqref{SSC_thm2_eq1} can be disturbed by the second, fourth and fifth terms. The states $(\bs{\zeta}, \bs{\mu})$ are weak solutions of the linear equations \eqref{adjointEq}. We can estimate these terms, thanks to the precise estimates of the nonlinear and linearized system of equations, \eqref{Cty_Fyy}, and \eqref{Cty_nuTT}, as follows:
	\begin{align*}
		&|c(\bs{\zeta}, \bs{\zeta}, \bar{\bs{\varphi}})| \leq C_{6_d} C_{3_d} \norm{\nabla \bar{\bs{\varphi}}}_{0,\Omega} \norm{\bs{\zeta}}^2_{1,\Omega} \leq C_{6_d} C_{3_d} M_{\bs{\psi}} M_{\bs{\varphi}} \norm{\bs{h}}^2_{L^{\rparam}(\Omega)}, \\
		&|(\bs{F}_{\bs{y} \bs{y}}(\bar{\bs{y}}) (\bs{\mu} \otimes \bs{\mu}), \bar{\bs{\varphi}})| \leq C_{F_{\bs{y}  \bs{y}}} C_{4_d}^2 \norm{\bs{\mu}}^2_{1,\Omega} \norm{\bar{\bs{\varphi}}}_{1,\Omega} \leq C_{F_{\bs{y} \bs{y}}} C_{4_d}^2 M_{\bs{\psi}} M_{\bs{\varphi}} \norm{\bs{h}}^2_{L^{\rparam}(\Omega)}, \\
		&|(\nu_{TT}(\bar{T}) (\mu^T)^2 \nabla \bar{\bs{u}}, \nabla \bar{\bs{\varphi}})| \leq C_{\nu_{TT}} \hat{M} \bar{M} M_{\bs{\psi}}  \norm{\bs{h}}^2_{L^{\rparam}(\Omega)},
	\end{align*}
	where, $\hat{M}$ is defined according to \eqref{Cty_nuTT} in two and three dimensions, respectively. Thus using the embeddings of $L^{\rparam}$-spaces  with the embedding constant $C_{2,\rparam}$ depending only on the domain $\Omega$, we obtain 
	\begin{align*}
		&	\bs{\cl{L}}_{\bs{w} \bs{w}}(\bar{\bs{w}}, \bar{\bs{\Lambda}}) \left[(\bs{\zeta}, \bs{\mu}, \bs{h})\right]^2 \nonumber\\&\geq \norm{\bs{\zeta}}^2_{0,\Omega} + \norm{\bs{\mu}}^2_{0,\Omega} + \big(\lambda C_{2,\rparam} - \big(M_{\bs{\psi}} M_{\bs{\varphi}} \big(C_{6_d} C_{3_d} + C_{F_{\bs{y} \bs{y}}} C_{4_d}^2 \big) + C_{\nu_{TT}} \bar{M} \hat{M} M_{\bs{\psi}} \big)\big) \norm{\bs{h}}^2_{L^{\rparam}(\Omega)}.	
	\end{align*}
	This implies that $\bs{\cl{L}}_{\bs{w} \bs{w}}$ is positive definite, if 
	\eqref{SSC_thm2_eq2} holds,
	that is, either $\lambda$ is large enough or the residuals $\norm{\bar{\bs{y}} - \bs{y}_d}_{0,\Omega}$ and $\norm{\bar{\bs{u}} - \bs{u}_d}_{0,\Omega}$ are small enough.
\end{proof}

\section{Numerical analysis of the state and adjoint equations}\label{NA:SE_AE}
In this section we propose a discretization of the governing equation \eqref{P:S} using non-conforming finite elements and study the well-posedness of resulting discrete system. To this aim, let $\mathcal{T}_h$ denote a partition of the domain $\Omega$ into a finite collection of open subdomains $K_i, \; i = 1, ..., N^{E}$ (triangles or quadrilaterals in (2D) and tetrahedra or hexahedra  in (3D)) such that:
$\bar{\Omega} = \cup_{K_i \in \mathcal{T}_h} \bar{K}_{i} \;\; \mbox{and} \;\; K_i \cap K_j = \phi, \; i \neq j.$
The mesh parameter $h$ associated with $\mathcal{T}_h$ is defined as $h = \max_{K \in \mathcal{T}_h} h_K,$ where $h_K$ denotes the diameter of $K$. We assume that $\mathcal{T}_h$ is quasi-uniform in the sense that there exists a real number $\bar{\rho} > 0$, independent of $h$ and $K$ such that as $h \rightarrow 0$,
$\frac{h_K}{\rho_K} \leq \bar{\rho}, \;\; \forall \;\; K \in \mathcal{T}_h,$
where $\rho_K := \sup \left\{diam(\mathcal{B}): \mathcal{B} \; \mbox{is a ball contained in K}\right\}.$  Let $\mathcal{E}(\mathcal{T}_h)$ be the set of all edges in (2D) or the set of all faces in (3D), which is divided into $\mathcal{E}_I(\mathcal{T}_h)$ and $\mathcal{E}_B(\mathcal{T}_h)$, the set of all interior and boundary edges, respectively. For any  $e \in \mathcal{E}_I(\mathcal{T}_h)$, we label by $K^{+}$ and $K^{-}$  the elements adjacent to $e$. Furthermore, $h_e$ denotes the edge length or the maximum diameter of the facet in (2D) and (3D), respectively. Let $\bs{v}, w$ represent smooth vector and scalar fields, respectively, defined over $\mathcal{T}_h$. Then, by $(\bs{v}^+, \bs{v}^-)$ and $(w^+, w^-)$ we will denote the traces of $\bs{v}$ and $w$ on $e$ taken from inside $K^+$ and $K^-$, respectively. Similarly, we denote by $\bs{n}_e^+$ and $\bs{n}_e^-$ the outward unit normal vectors to $e$ corresponding to $K^+$ and $K^-$, respectively. The averages and jumps of $\bs{v}$ and $w$ across $e$ are defined, respectively, as follows 
\begin{align*}
	\avg{\bs{v}} := \frac{(\bs{v}^+ + \bs{v}^-)}{2}, \;\; \jump{\bs{v}} := \left(\bs{v}^- \cdot \bs{n}^-_e + \bs{v}^+ \cdot \bs{n}^+_e\right),  
	\avg{w} := \frac{(w^+ + w^-)}{2}, \;\; \jump{w} := \left(w^- \bs{n}^-_e + w^+  \bs{n}^+_e\right).
\end{align*}
The respective notations for boundary jumps and averages are, $\avg{\bs{v}} = \bs{v},  \jump{\bs{v}} = \bs{v} \cdot \bs{n}_{\Gamma},  \avg{w} = w, \; \jump{w} = w \; \bs{n}_{\Gamma}. $ The labeling of the two elements adjacent to an interior edge/face \(e\) by \(K^+\) and \(K^-\) is arbitrary and fixed once and for all. The normals \(\boldsymbol n_e^\pm\) are the outward unit normals to \(e\) with respect to \(K^\pm\), respectively, so that \(\bs{n}_e^- = -\bs{n}_e^+\). The definitions of the jumps and averages are independent of this arbitrary labeling.
In addition, $\nabla_h$ denotes the broken gradient operator, that is, the operator $\nabla$ taken element wise. For any natural number $k \geq 1$, $\mathbb{P}_k(K)$ denotes the space of piecewise polynomials of degree less than or equal to $k$. In this section, for some $\bs{U} \in \bs{L}^{\rparam}(\Omega)$, we use the following notations $\bs{u}_h := \bs{u}_h(\bs{U}), p_h := p_h(\bs{U}), \bs{y}_h := \bs{y}_h(\bs{U}).$  

\subsection{Numerical analysis of the state equation}\label{Sec:nc:S} 

Consider the following lowest order Crouziex-Raviart space on the partition $\mathcal{T}_h$,
\begin{align*}
	&\bs{V}^{cr}_h := \left\{ \bs{v}_h \in \bs{L}^2(\Omega) : \bs{v}_h |_K \in \left[\mathbb{P}_1(K)\right]^d, \bs{v}_h \; \mbox{is continuous at midpoint} \; \forall \; e \in \mathcal{E}_I(\mathcal{T}_h) \right\}, \\
	&\bs{V}^{cr_0}_h := \left\{ \bs{v}_h \in \bs{V}^{cr}_h : \bs{v}_h  \; \mbox{at midpoint is zero} \; \forall \; e \in \mathcal{E}_B(\mathcal{T}_h) \right\},
\end{align*}
and let $\mathcal{M}_h^{cr}$ and $\mathcal{M}_h^{cr_0}$ denote their counterparts for all $\bs{s}_h \in \left[L^2(\Omega)\right]^2$. Moreover, we also define the following discontinuous finite element space
 $$Q_h^{k} := \left\{q_h \in L_0^2(\Omega) : q_h|_K \in \mathbb{P}_{k}(K) \;\; \forall \;\; K \in \mathcal{T}_h\right\}.$$
Corresponding to these finite dimensional spaces, we consider the following nonconforming finite element method which is based on an upwind discretization to tackle the convective terms (cf. \cite{DS_upwind}).  For some $\bs{U} \in \bs{L}^{\rparam}(\Omega)$, find $(\bs{u}_h,p_h,\bs{y}_h) \in \bs{V}^{cr_0}_h \times Q^0_h \times \mathcal{M}^{cr}_h$ such that $\bs{y}_h |_{\Gamma} = \bs{y}_h^D$ and 
\begin{align}\label{ncPh:S}
	\left\{
	\begin{aligned}
		a^{h}(\bs{y}_h; \bs{u}_h, \bs{v}_h) + c^{h}(\bs{u}_h, \bs{u}_h, \bs{v}_h) + b^{h}(\bs{v}_h,p_h) &= d^h(\bs{y}_h,\bs{v}_h) + (\bs{U},\bs{v}_h)_{\rparam,\rdual}  \;\; \forall \;\; \bs{v}_h \in \bs{V}^{cr_0}_h, \\
		b^h(\bs{u}_h,q_h) &= 0 \;\; \forall \;\; q_h \in Q^0_h, \\
		a^h_{\bs{y}}(\bs{y}_h,\bs{s}_h) + c^h_{\bs{y}}(\bs{u}_h,\bs{y}_h,\bs{s}_h) &= 0 \;\; \forall \;\; \bs{s}_h \in \mathcal{M}^{cr_0}_h, 
	\end{aligned}
	\right.
\end{align}
with
\begin{align*}
	a^h(\bs{s}_h; \bs{u}_h, \bs{v}_h) &:= \int_{\Omega} \left(\bs{K}^{-1} \bs{u}_h \cdot \bs{v}_h + \nu(\bs{s}_h) \nabla_h \bs{u}_h : \nabla_h \bs{v}_h \right)  dx,\;\;
	 a^h_{\bs{y}}(\bs{y}_h, \bs{s}_h) := \int_{\Omega} \bs{D}~ \nabla_h\bs{y}_h : \nabla_h\bs{s}_h ~dx, \\
	c^h(\bs{w}_h, \bs{u}_h, \bs{v}_h) &:= \int_{\Omega} \left(\bs{w}_h \cdot \nabla_h \right) \bs{u}_h  \cdot \bs{v}_h ~dx + \sum_{K \in \mathcal{T}_h} \int_{\partial K \backslash \Gamma} \hat{\bs{w}}^{up}_{h} (\bs{u}_h) \cdot \bs{v}_h ~ds, \\
	c^h_{\bs{y}}(\bs{w}_h, \bs{y}_h, \bs{s}_h) &:=  \int_{\Omega} \left(\bs{w}_h \cdot \nabla_h \right) \bs{y}_h \cdot \bs{s}_h ~dx + \sum_{K \in {\mathcal{T}_h}} \int_{\partial K \backslash \Gamma} \hat{\bs{w}}^{up}_{h} (\bs{y}_h) \cdot \bs{s}_h ~ds,
\end{align*}
where $b^h(\cdot,\cdot), \; d^h(\cdot,\cdot)$ allows for discontinuous elements and are defined analogously to $b(\cdot,\cdot), \; d(\cdot,\cdot)$, respectively. The fluxes are defined as follows:
$$\small \hat{\bs{w}}^{up}_h(\bs{u}_h) := \frac{1}{2} \left(\bs{w}_h \cdot  \bs{n}_K - |\bs{w}_h \cdot \bs{n}_K|\right)(\bs{u}_h^e - \bs{u}_h), \hat{\bs{w}}^{up}_h(\bs{y}_h) := \frac{1}{2} \left(\bs{w}_h \cdot  \bs{n}_K - |\bs{w}_h \cdot \bs{n}_K|\right)(\bs{y}_h^e - \bs{y}_h),$$
and $\bs{u}_h^e$ and $\bs{y}_h^e$ are the traces of $\bs{u}$ and $\bs{y} = (T,S)^\top$, respectively, taken from within the exterior of $K$ over the edge under consideration and is zero on the boundary.  

The space $\bs{V}_h^{cr_0}$ is a subset of the piecewise Sobolev space
$$\bs{H}^1(\mathcal{T}_h) := \left\{\bs{v} \in \bs{L}^2(\Omega) : \bs{v}|_{K} \in \bs{H}^1(K) \;\; \forall \;\; K \in \mathcal{T}_h\right\}.$$
Similarly, $\mathcal{M}_h^{cr}$ is also a subset of $\left[H^1(\mathcal{T}_h)\right]^2$ which is defined analogously for all $\bs{s}_h \in \left[L^2(\Omega)\right]^2$. Now we define the following mesh and parameter dependent norms over $\bs{V}_h^{cr_0}$ and $\mathcal{M}^{cr}_h$ as follows:
\begin{align*}
	\norm{\bs{v}}^2_{1,\mathcal{T}_h^{nc}} := \sum_{K \in {\mathcal{T}_h}} \sigma \norm{\bs{v}}^2_{0,K} + \nu_2 \norm{\nabla \bs{v}}^2_{0,K}, \; \forall \; \bs{v} \in \bs{H}^1({\mathcal{T}_h}),  \;
	\norm{\bs{s}}^2_{1,\mathcal{T}_h^{nc}} := \sum_{K \in {\mathcal{T}_h}} \bar{\sigma} \norm{\nabla \bs{s}}^2_{0,K}, \; \forall \; \bs{s} \in \left[H^1(\mathcal{T}_h)\right]^2,
\end{align*}
where $\sigma = \norm{\bs{K}^{-1}}_{\infty,\Omega}$, $\bar{\sigma} = \norm{\bs{D}}_{\infty, \Omega}$ and $\nu_2$ is as defined in Section \ref{Intro}. Since the above norms are parameter-dependent, the constants appearing in the discrete stability estimates may depend on the fixed physical parameters, but remain independent of $h$. This dependance was absorbed into generic constants in the analysis of the state equation in \cite{TKM_ncDDF}. Here we make it explicit. In particular
\begin{align*}
	\left( \sum_{K \in \mathcal T_h} \norm{\nabla \bs{v}_h}^2_{0,K} \right)^{1/2} \leq \nu_2^{-1/2} \norm{\bs{v}_h}_{1,\mathcal{T}_h^{nc}}, \quad \left(\sum_{K \in \mathcal T_h} \norm{\nabla \bs{s}_h}^2_{0,K}\right)^{1/2} \leq \bar{\sigma}^{-1/2} \norm{\bs{s}_h}_{1,\mathcal{T}_h^{nc}}.
\end{align*}
and the discrete Sobolev embedding \cite{Buffa_IMA, TKM_ncDDF}: For $r \in [1,\infty)$ if $d=2$, and $r \in [1,6]$ if $d=3$,
\begin{align}\label{eq:discSob.para.dep}
	\norm{\bs{v}}_{L^r(\Omega)} \lesssim \nu_2^{-1/2} \norm{\bs{v}}_{1,\mathcal T_h^{nc}} \quad \forall \bs{v} \in \bs{H}^1(\mathcal T_h).
\end{align}
Consequently, tracking these factors in the arguments used in \cite[Lemma 4.1]{TKM_ncDDF}, we infer 

\begin{lemma}\label{ncStateCty}
	There holds for all $ \bs{w_h}, \bs{u}_h, \bs{v}_h \in \bs{V}_h^{cr_0}$, $\bs{y}_h \in \mathcal{M}_h^{cr}, \bs{s}_h \in \mathcal{M}^{cr_0}_h$, 
	\begin{align*}
		|a^h(\bs{y}_h; \bs{u}_h, \bs{v}_h)| &\leq C_a^{nc} \norm{\bs{u}_h}_{1,\mathcal{T}_h^{nc}} \norm{\bs{v}_h}_{1,\mathcal{T}_h^{nc}}, \\ 
		|a^h_{\bs{y}}(\bs{y}_h, \bs{s}_h)| &\leq \hat{C}_a^{nc} \norm{\bs{y}_h}_{1,\mathcal{T}_h^{nc}} \norm{\bs{s}_h}_{1,\mathcal{T}_h^{nc}}, \\
		|c^h(\bs{w}_h, \bs{u}_h, \bs{v}_h)| &\leq C_{e}^{nc} \norm{\bs{w}_h}_{1,\mathcal{T}_h^{nc}} \norm{\bs{u}_{h}}_{1,\mathcal{T}_h^{nc}} \norm{\bs{v}_h}_{1,\mathcal{T}_h^{nc}}, \\
		|c^h_{\bs{y}}(\bs{w}_h, \bs{y}_h, \bs{s}_h)| &\leq C_{e}^{nc} \norm{\bs{w}_h}_{1,\mathcal{T}_h^{nc}} \norm{\bs{y}_{h}}_{[L^4(\Omega)]^2} \norm{\bs{s}_h}_{1,\mathcal{T}_h^{nc}}, \\
		a^h(\cdot; \bs{u}_h, \bs{u}_h) &\geq \alpha_a^{nc} \norm{\bs{v}_h}^2_{1, \mathcal T_h^{nc}}, \\
		a^h_{\bs{y}}(\bs{s}_h, \bs{s}_h) &\geq \alpha_2^{nc} \norm{\bs{s}_h}^2_{1, \mathcal T_h^{nc}}.
	\end{align*}
	Moreover, $|b^h(\bs{v}_h, q_h)| \leq C_d \norm{\bs{v}_h}_{1,\mathcal{T}_h^{nc}} \norm{q_h}_{0,\Omega}$  for all $\bs{v}_h \in \bs{H}^1(\mathcal{T}_h), q_h \in L_0^2(\Omega).$ The constants above can be chosen independently of $h$ such that
	\begin{align*}
		&C_e^{nc} \lesssim \max \left\{ \nu_2^{-1/2}, \bar \sigma^{-1/2}, \nu_2^{-3/2}, (\nu_2 \bar \sigma)^{-1/2} \right\}, \quad C_d \lesssim \nu_2^{-1/2},\\
		&C_a^{nc}, \hat C_a^{nc} \sim \mathcal O(1), \quad
		\alpha_a^{nc} = \min \{ \alpha_1 / \sigma, \nu_1 / \nu_2 \}, \quad
		\alpha_2^{nc} = \alpha_2/ \bar \sigma.
	\end{align*}
\end{lemma}
We now discuss the existence of solution for the state equation. We begin by proving energy estimates for the equivalent reduced state equation which is obtained similar to \cite[Lemma 5.1]{TKM_ncDDF}.
\begin{lemma}\label{ncEnergyEstState}
	Let $(\boldsymbol{u}_h, \boldsymbol{y}_h)$ be a solution of the reduced state equation
	\begin{align}\label{ncPh:Sred}
		\left\{
		\begin{aligned}
			a^h(\boldsymbol{y}_h;\boldsymbol{u}_h,\boldsymbol{v}_h) + c^h(\boldsymbol{u}_h, \boldsymbol{u}_h, \boldsymbol{v}_h) - d^h(\boldsymbol{y}_h, \boldsymbol{v}_h) - (\boldsymbol{U},\boldsymbol{v}_h)_{4/3,4} &= 0 \;\; \forall \;\; \boldsymbol{v}_h \in \boldsymbol{X}_h^{cr},\\
			a^h_{\boldsymbol{y}}(\boldsymbol{y}_h,\boldsymbol{s}_h) + c^h_{\boldsymbol{y}}(\boldsymbol{u}_h, \boldsymbol{y}_h, \boldsymbol{s}_h) &= 0 \;\; \forall \;\; \boldsymbol{s}_h \in \mathcal{M}_h^{cr_0}.
		\end{aligned}
		\right.
	\end{align}
	with the discrete boundary lifting as defined in \cite[(5.2)]{TKM_ncDDF}. The discrete kernel space $\bs{X}_h^{cr}$ is defined as
	$$\boldsymbol{X}^{cr}_h := \left\{\boldsymbol{v}_h \in \boldsymbol{V}^{cr_0}_h : b^h(\boldsymbol{v}_h, q_h) = 0 \;\; \forall \;\; q_h \in Q^0_h\right\} = \left\{\boldsymbol{v}_h \in \boldsymbol{V}^{cr_0}_h : \bdiv \boldsymbol{v}_h = 0\right\}.$$
	Then there exist positive constants $C_{\boldsymbol{u}}^{nc}, C_{\boldsymbol{y}}^{nc}$ independent of the mesh parameter $h$, such that
	$$\norm{\boldsymbol{u}_h}_{1,\mathcal{T}_h^{nc}} \leq C_{\boldsymbol{u}}^{nc} \big(\norm{\boldsymbol{y}^D}_{1/2,\Gamma} + \norm{\boldsymbol{U}}_{L^{4/3}(\Omega)}\big) := M_{\boldsymbol{u}}^{nc}, \;\;\;\; \norm{\boldsymbol{y}_{h,0}}_{1,\mathcal{T}_h^{nc}} \leq C_{\boldsymbol{y}}^{nc} \big(\norm{\boldsymbol{y}^D}_{1/2,\Gamma} + \norm{\boldsymbol{U}}_{L^{4/3}(\Omega)}\big), \mbox{and}$$
	$$\norm{\boldsymbol{y}_h}_{1,\mathcal{T}_h^{nc}} \leq C_{\boldsymbol{y}}^{nc} \big(\norm{\boldsymbol{y}^D}_{1/2,\Gamma} + \norm{\boldsymbol{U}}_{L^{4/3}(\Omega)}\big) := M_{\boldsymbol{y}}^{nc}.$$
	Here, a discrete boundary lifting of $\bs{y}_{h}$ is defined in the following manner:
	\begin{align}\label{discreteLifitingdefn}
		\boldsymbol{y}_h = \boldsymbol{y}_{h,0} + \boldsymbol{y}_{h,1} \; \mbox{with} \; \boldsymbol{y}_{h,0} \in \mathcal{M}_{h}^{cr_0} \; \mbox{such that} \; \boldsymbol{y}_{h,1} \in \mathcal{M}_h^{cr} \; \mbox{and} \; \boldsymbol{y}_{h,1}|_{\Gamma} = \boldsymbol{y}_h^D.
	\end{align}
\end{lemma}
\begin{proof}
	The proof is analogous to \cite[Lemma 5.3]{TKM_ncDDF}.
\end{proof}
We make use of the following proposition which is a consequence of the Brouwer fixed point theorem to show the existence of discrete solutions \cite{temam2001navier}:
\begin{proposition}\label{ConseqOfBFPT}
	Let H be a finite dimensional Hilbert space with norm $|\cdot|_H$ and inner product $(\cdot,\cdot)_H$. Let $R: H \longrightarrow H$ be a continuous map such that there exists a positive constant $\rho$ satisfying
	$$(R(v),v)_H > 0 \;\; \forall \; v \in H \;\; \mbox{with} \;\; |v|_H = \rho.$$
	Then there exists a $w \in H$ such that $R(w) = 0, |w|_H \leq \rho$. 
\end{proposition}
\begin{theorem}\label{ncExistenceState}
	Let $\boldsymbol{y}_{h,1}$ be the discrete lifting of $\boldsymbol{y}_h$. Then there exists a discrete solution $(\boldsymbol{u}_h, \boldsymbol{y}_h) \in \boldsymbol{X}_h^{cr} \times \mathcal{M}_h^{cr}$ to \eqref{ncPh:Sred} satisfying the bounds of Lemma \ref{ncEnergyEstState}.
\end{theorem}
\begin{proof}
	Define the maps $Q^{\boldsymbol{u}}$ and $Q^{\boldsymbol{y}}$ on $\boldsymbol{X}_h^{cr}$ and $\mathcal{M}_h^{cr_0}$ for all $\boldsymbol{v}_h \in \boldsymbol{X}_h^{cr}$ and $\boldsymbol{s}_h \in \mathcal{M}_h^{cr_0}$, respectively, as follows:
	\begin{align*}
		(Q^{\boldsymbol{u}}(\boldsymbol{u}_h), \boldsymbol{v}_h) &= a^h(\boldsymbol{y}_h; \boldsymbol{u}_h, \boldsymbol{v}_h) + c^h(\boldsymbol{u}_h, \boldsymbol{u}_h, \boldsymbol{v}_h) - d^h(\boldsymbol{y}_h, \boldsymbol{v}_h) - (\boldsymbol{U}, \boldsymbol{v}_h)_{4/3,4},\\
		(Q^{\boldsymbol{y}}(\boldsymbol{y}_{h,0}), \boldsymbol{s}_h) &= a^h(\boldsymbol{y}_{h,0}, \boldsymbol{s}_h) + c^h_{\boldsymbol{y}}(\boldsymbol{u}_h, \boldsymbol{y}_{h,0}, \boldsymbol{s}_h) - a^h(\boldsymbol{y}_{h,1}, \boldsymbol{s}_h)  - c^h_{\boldsymbol{y}}(\boldsymbol{u}_h, \boldsymbol{y}_{h,1}, \boldsymbol{s}_h).
	\end{align*} 	
	The Continuity of $Q^{\boldsymbol{u}}$ and $Q^{\boldsymbol{y}}$ readily follows by using the boundedness properties of Lemma \ref{ncStateCty}. Let $Q^{\boldsymbol{u},\boldsymbol{y}} := Q^{\boldsymbol{u}} \times Q^{\boldsymbol{y}}$ with $\norm{(\boldsymbol{u}_h, \boldsymbol{y}_{h,0})} = \norm{\boldsymbol{u}_h}_{\boldsymbol{X}_h^{cr}} + \norm{\boldsymbol{y}_{h,0}}_{\mathcal{M}_h^{cr_0}}.$ Then the continuity of $Q^{\boldsymbol{u}, \boldsymbol{y}}$ follows from the continuity of maps  $Q^{\boldsymbol{u}}$ and $Q^{\boldsymbol{y}}$ and the definition of $Q^{\boldsymbol{u}, \boldsymbol{y}}$. In order to make use of Proposition \ref{ConseqOfBFPT}, which would imply that $Q^{\boldsymbol{u,\boldsymbol{y}}}$ has a fixed point, it remains to show the coercivity of $Q^{\boldsymbol{u},\boldsymbol{y}}$. Since $\boldsymbol{u}_h \in \boldsymbol{X}_h^{cr}$, Lemma \ref{ncStateCty} and \cite[(4.5))]{TKM_ncDDF}, the form $a^h(\cdot; \cdot, \cdot) + c^h(\cdot, \cdot, \cdot)$ is coercive, and on using the boundedness properties, we get for any $\bs{y_h}$ and $\bs{U}$
	\begin{align*}
		(Q^{\boldsymbol{u}}(\boldsymbol{u}_h), \boldsymbol{u}_h) &\geq \alpha_a^{nc} \norm{\boldsymbol{u}_h}^2_{1,\mathcal{T}_h^{nc}} - C_F C_e^{nc} \norm{\boldsymbol{y}_h}_{1,\mathcal{T}_h^{nc}} \norm{\boldsymbol{u}_h}_{1,\mathcal{T}_h^{nc}} - C_e^{nc} \norm{\boldsymbol{U}}_{L^{4/3}(\Omega)} \norm{\boldsymbol{u}_h}_{1,\mathcal{T}_h^{nc}} \nonumber\\
		&\geq \norm{\boldsymbol{u}_h}_{1,\mathcal T_h^{nc}} \left( \alpha_a^{nc}\norm{\boldsymbol{u}_h}_{1,\mathcal T_h^{nc}} - C_F C_e^{nc} \norm{\boldsymbol{y}_h}_{1,\mathcal T_h^{nc}} - C_e^{nc} \norm{\boldsymbol{U}}_{L^2(\Omega)} \right).
	\end{align*}
	It follows that $(Q^{\boldsymbol{u}}(\boldsymbol{u}_h), \boldsymbol{u}_h) > 0$ for a ball of radius $\norm{\boldsymbol{u}_h}_{1,\mathcal{T}_h^{nc}} = \kappa_1^{nc}$, and $\kappa_1^{nc}$ is sufficiently large: more precisely
	$$\kappa_1^{nc} > \frac{C_e^{nc}}{\alpha_a^{nc}}\left(C_F \norm{\boldsymbol{y}_h}_{1,\mathcal T_h^{nc}} + \norm{\boldsymbol{U}}_{L^{4/3}(\Omega)}\right).$$
	Similarly, invoking Lemma \eqref{ncStateCty} and \cite[(4.6))]{TKM_ncDDF}, the form $a^h_{\boldsymbol{y}}(\cdot;\cdot,\cdot) + c^h_{\boldsymbol{y}}(\cdot,\cdot,\cdot)$ is coercive, yields for a fixed $\boldsymbol{u}_h$
	\begin{align*}
		(Q^{\boldsymbol{y}}(\boldsymbol{y}_{h,0}), \boldsymbol{y}_{h,0}) &\geq \alpha_2^{nc} \norm{\boldsymbol{y}_{h,0}}^2_{1,\mathcal{T}_h^{nc}} - \hat{C}_a^{nc} \norm{\boldsymbol{y}_{h,1}}_{1,\mathcal{T}_h^{nc}} \norm{\boldsymbol{y}_{h,0}}_{1,\mathcal{T}_h^{nc}} - C_e^{nc} \norm{\boldsymbol{u}_h}_{1,\mathcal{T}_h^{nc}} \norm{\boldsymbol{y}_{h,1}}_{1,\mathcal{T}_h^{nc}} \norm{\boldsymbol{y}_{h,0}}_{1,\mathcal{T}_h^{nc}} \\
		&\geq \norm{\boldsymbol{y}_{h,0}}_{1,\mathcal{T}_h^{nc}} \left(
			\alpha_2^{nc} \norm{\boldsymbol{y}_{h,0}}_{1,\mathcal{T}_h^{nc}} - \hat{C}_a^{nc} \norm{\boldsymbol{y}_{h,1}}_{1,\mathcal{T}_h^{nc}}  - C_e^{nc} \norm{\boldsymbol{u}_h}_{1,\mathcal T_h^{nc}} \norm{\boldsymbol{y}_{h,1}}_{1,\mathcal{T}_h^{nc}}
		\right).
	\end{align*}
	Then, we conclude that $(Q^{\boldsymbol{y}}(\boldsymbol{y}_{h,0}), \boldsymbol{y}_{h,0}) > 0$ for a ball of radius $\norm{\boldsymbol{y}_{h,0}}_{1,\mathcal{T}_h^{nc}} = \kappa_2^{nc}$, and $\kappa_2^{nc}$ is sufficiently large: more precisely
	$$\kappa_2^{nc} > \frac{\norm{\boldsymbol{y}_{h,1}}_{1,\mathcal T_h^{nc}}}{\alpha_2^{nc}} \left( \hat C_a^{nc} + C_e^{nc} \norm{\boldsymbol{u}_h}_{1,\mathcal T_h^{nc}}\right).$$
	From the coercivity of $Q^{\boldsymbol{y}}$ and $Q^{\boldsymbol{u}}$ and the definition of $Q^{\boldsymbol{u}, \boldsymbol{y}}$, we deduce that $(Q^{\boldsymbol{u}, \boldsymbol{y}}(\boldsymbol{u}_h, \boldsymbol{y}_{h,0}),(\boldsymbol{u}_h, \boldsymbol{y}_{h,0})) > 0$ for $\norm{(\boldsymbol{u}_h, \boldsymbol{y}_{h,0})} = \kappa_1^{nc} + \kappa_2^{nc} > 0.$
\end{proof}

\begin{remark}\label{rem:state_uniqueness}
	The uniqueness of discrete solutions to \eqref{ncPh:S} can be established under a smallness assumption on the data analogous to \cite{TKM_ncDDF}. Furthermore, let $(\bs{u}_h(\bs{U}_i), \bs{y}_h(\bs{U}_i))$ be solutions of \eqref{ncPh:Sred} corresponding to the data $\bs{U}_i \in \bs{L}^{4/3}(\Omega)$ for $i = 1, 2$, then the following stability result holds under the smallness assumption:
	\begin{align}
		\norm{\bs{u}_h(\bs{U}_1) - \bs{u}_h(\bs{U}_2)}_{1, \mathcal T_h^{nc}} + \norm{\bs{y}_h(\bs{U}_1) - \bs{y}_h(\bs{U}_2)}_{1, \mathcal T_h^{nc}} \leq C \norm{\bs{U}_1 - \bs{U}_2}_{L^{4/3}(\Omega)}.
	\end{align}
\end{remark}

We state here the {\it a priori} error estimate for the governing equation which will be used in the numerical analysis of the control problem. The proof is detailed in \cite[Theorem 5.6]{TKM_ncDDF}.
\begin{theorem}\label{NA_State}
	 For some $\bs{U} \in \bs{L}^{4/3}(\Omega)$, let $(\bs{u}, p, \bs{y})$ and $(\bs{u}_h, p_h, \bs{y}_h)$ be the solutions of \eqref{P:S} and \eqref{ncPh:S}, respectively. Furthermore, for $\delta \in (0,\frac{1}{2})$ and $\bs{u} \in \bs{H}_0^1(\Omega) \cap \bs{H}^{3/2 + \delta}(\Omega),\;p \in L_0^2(\Omega) \cap H^{1/2 + \delta}(\Omega),\;\bs{y} \in \left[H^{3/2 + \delta}(\Omega) \right]^2$. Then under the smallness assumption \cite[Theorem 5.6]{TKM_ncDDF}, there exists a positive constant $C$ independent of the mesh parameter $h$ such that
	\begin{align}
		\norm{\bs{u} - \bs{u}_h}_{1,\mathcal{T}_h^{nc}} + \norm{\bs{y} - \bs{y}_h}_{1,\mathcal{T}_h^{nc}} &\leq C h^{1/2+\delta} \big(\norm{\bs{u}}_{3/2+\delta,\Omega} + \norm{\bs{y}}_{3/2+\delta,\Omega} \big), \label{aPriori_uy} \\
		\norm{p - p_h}_{0,\Omega} &\leq C h^{1/2+\delta} \big(\norm{\bs{u}}_{3/2+\delta,\Omega} + \norm{\bs{y}}_{3/2+\delta,\Omega} + \norm{p}_{1/2+\delta,\Omega}\big). \label{aPriori_p}
	\end{align} 	
\end{theorem}

\section{Numerical analysis of the adjoint equation}\label{NA:Adjoint}
The goal of this section is to propose a nonconforming discretization of the adjoint equation \eqref{adjointEq}  and establish the well-posedness and a priori estimates of the resulting discrete system. 
For some $\bs{U} \in \bs{\mathcal{U}}_{ad}$ and $\bs{u}_h := \bs{u}_h(\bs{U}) \in \bs{V}^{cr_0}_h$, $\bs{y}_h := \bs{y}_h(\bs{U}) \in \mathcal{M}_h^{cr}$, find $(\bs{\varphi}_h,\xi_h,\bs{\eta}_h) \in \bs{V}^{cr_0}_h \times Q^0_h \times \mathcal{M}^{cr_0}_h$ for all $(\bs{v}_h, q_h,\bs{s}_h) \in \bs{V}^{cr_0}_h \times Q^0_h \times \mathcal{M}^{cr_0}_h$ such that
	\begin{align}\label{nc:P:A}
		\small
		\left\{
		\begin{aligned}
			a^h(\bs{y}_h;\bs{\varphi}_h,\bs{v}_h) - c^h(\bs{u}_h, \bs{\varphi}_h,\bs{v}_h) + c^h(\bs{v}_h,\bs{u}_h,\bs{\varphi}_h) + b(\bs{v}_h, \xi_h) +c^h_{\bs{y}}(\bs{v}_h,\bs{y}_h,\bs{\eta}_h) &= (\bs{u}_h-\bs{u}_d,\bs{v}_h), \\
			b^h(\bs{\varphi}_h,q_h) &= 0,\\
			a^h_{\bs{y}}(\bs{\eta}_h,\bs{s}_h) - c^h_{\bs{y}}(\bs{u}_h,\bs{\eta}_h,\bs{s}_h) + (\nu_{T}(T_h) \nabla_h \bs{u}_h : \nabla_h \bs{\varphi}_h, \bs{s}_h) - ((\bs{F}_{\bs{y}}(\bs{y}_h))^{\top} \bs{\varphi}_h, \bs{s}_h) &= (\bs{y}_h-\bs{y}_d,\bs{s}_h),
	\end{aligned}	
	\right.
\end{align}
where, the discrete bilinear and trilinear forms are defined as in Section \ref{Sec:nc:S}.
\subsection{Discrete stability properties}\label{Sec:DSP:A}
The stability properties discussed in Lemma \ref{ncStateCty} hold alongwith the following properties which can be derived by applying Cauchy-Schwarz, H\"older's and discrete Poincar\'e's inequality and discrete Sobolev embedding \eqref{eq:discSob.para.dep},
\begin{align*}
	|(\bs{u}_h-\bs{u}_d, \bs{v}_h)| &\leq \bar{C}_d^{nc} \norm{\bs{u}_h-\bs{u}_d}_{0,\Omega} \norm{\bs{v}_h}_{1,\mathcal{T}_h^{nc}}, \\
	|(\bs{y}_h-\bs{y}_d, \bs{s}_h)| &\leq \hat{C}_d^{nc} \norm{\bs{y}_h-\bs{y}_d}_{0,\Omega} \norm{\bs{s}_h}_{1,\mathcal{T}_h^{nc}},  \\
	|((\bs{F}_{\bs{y}}(\bs{y}_h))^{\top} \bs{\varphi}_h, \bs{s}_h)| &\leq C_{F_{\bs{y}}} C_p^{nc} \norm{\bs{\varphi}_h}_{1,\mathcal{T}_h^{nc}} \norm{\bs{s}_h}_{1,\mathcal{T}_h^{nc}}, \\
	|(\nu_{T}(T_h) \nabla_h \bs{u}_h : \nabla_h \bs{\varphi}_h, \bs{s}_h)| &\leq \sum_{K \in {\mathcal{T}_h}} \|(\nu_T(T_h))^{\top}\|_{\infty, K} \norm{\nabla \bs{\varphi}_h}_{0,K} \norm{\nabla \bs{u}_h}_{L^3(K)} \norm{\bs{s}_h}_{L^6(K)},\\
	&\leq C_{\nu_{T}} C_e^{nc} \| \bs{\varphi}_h\|_{1,\mathcal{T}_h^{nc}} \norm{\nabla_h \bs{u}_h}_{L^3(\Omega)} \norm{ \bs{s}_h}_{1, \mathcal{T}_h^{nc}}.
\end{align*}

The constants above can be chosen independent of the mesh size $h$ such that
\begin{align*}
	\bar C_d^{nc} \lesssim \nu_2^{-1/2}, \quad \hat C_d^{nc} \lesssim \bar \sigma^{-1/2}, \quad C_p^{nc} \lesssim (\nu_2 \bar \sigma)^{-1/2}.
\end{align*}
Moreover, the generic constant $C_e^{nc}$ in the last estimate is understood in the sense specified in Lemma \ref{ncStateCty} and, for this particular estimate, contains the factor $(\nu_2 \bar \sigma)^{-1/2}$. The constants $C_{F_{\bs{y}}}$ and $C_{\nu_{T}}$ are specified in \eqref{CtyFy} and \eqref{CtynuT}.

\subsection{Well-posedness of discrete solutions}
We have the following equivalence result.
\begin{lemma}\label{Discrete:Adjoint:equivalence}
	If $(\bs{\varphi}_h,\xi_h,\bs{\eta}_h) \in \bs{V}_h^{cr_0} \times Q_h^0 \times \mathcal{M}_h^{cr_0}$ solves \eqref{nc:P:A}, then for some $\bs{U} \in \bs{\mathcal{U}}_{ad}$, $(\bs{u}_h,\bs{y}_h) \in \bs{X}_h^{cr} \times \mathcal{M}_h^{cr_0}$ satisfies
	\begin{align}\label{nc:P:Ared}
		\small
		\left\{
		\begin{aligned}
			a^h(\bs{y}_h,\bs{\varphi}_h,\bs{v}_h) - c^h(\bs{u}_h, \bs{\varphi}_h,\bs{v}_h) + c^h(\bs{v}_h,\bs{u}_h,\bs{\varphi}_h) + c^h_{\bs{y}}(\bs{v}_h,\bs{y}_h,\bs{\eta}_h) &= (\bs{u}_h-\bs{u}_d,\bs{v}_h), \\
			a^h_{\bs{y}}(\bs{\eta}_h,\bs{s}_h) - c^h_{\bs{y}}(\bs{u}_h,\bs{\eta}_h,\bs{s}_h) + (\nu_{T}(T_h) \nabla_h \bs{u}_h : \nabla_h \bs{\varphi}_h, \bs{s}_h) - ((\bs{F}_{\bs{y}}(\bs{y}_h))^{\top} \bs{\varphi}_h, \bs{s}_h) &= (\bs{y}_h-\bs{y}_d,\bs{s}_h).
		\end{aligned}
		\right.
	\end{align}
	Conversely, if $(\bs{\varphi}_h,\bs{\eta}_h) \in \bs{X}_h^{cr} \times \mathcal{M}_h^{cr_0}$ is a solution of the reduced problem \eqref{ncPh:Sred}, then for some $\bs{U} \in \bs{\mathcal{U}}_{ad}$ there exists $\xi_h \in Q_h^0$ such that $(\bs{\varphi}_h, \xi_h ,\bs{\eta}_h)$ is a solution of \eqref{nc:P:A}.
\end{lemma}

The following lemma states the energy estimates for the discrete adjoint equation. It can be proven along the lines of Lemma \cite[Lemma 5.2]{TKM_ncDDF} and following the steps analogous to the proof of Theorem \cite[Theorem 5.6]{TKM_ncDDF}.
\begin{lemma}\label{ncEnergyEstAdjoint}
	Let $(\bar{\bs{u}}_h, \bar{\bs{y}}_h) \in \bs{X}_h^{cr} \times \mathcal{M}_h^{cr}$ be the states associated with the control $\bar{\bs{U}} \in \bs{\mathcal{U}}_{ad}$ and 
	\begin{align}\label{ncUniquenessAssumption}
		\norm{\nabla_h \bs{u}_h}_{L^3(\Omega)} \leq \tilde{M}, 
	\end{align}
	 where $\tilde{M}>0$ is chosen such that following smallness assumption \cite[(5.8)]{TKM_ncDDF} holds. Then the solution 
	 $(\bs{\varphi}_h, \bs{\eta}_h) \in \bs{X}_h^{cr} \times \mathcal{M}_h^{cr_0}$ of \eqref{nc:P:Ared}, satisfies the following a priori estimate:
	 \begin{align}
	 	&\norm{\bs{\varphi}_h}_{1,\mathcal{T}_h^{nc}} \leq C_{\bs{\varphi}}^{nc} \left(\norm{\bs{u}_h - \bs{u}_d}_{0,\Omega} + \norm{\bs{y}_h - \bs{y}_d}_{0,\Omega}\right) := M_{\bs{\varphi}}^{nc}, \label{energyEst_varphi}\\ &\norm{\bs{\eta}_{h}}_{1,\mathcal{T}_h^{nc}} \leq C_{\bs{\eta}}^{nc} \left(\norm{\bs{u}_h - \bs{u}_d}_{0,\Omega} + \norm{\bs{y}_h - \bs{y}_d}_{0,\Omega}\right) := M_{\bs{\eta}}^{nc}, \label{energyEst_eta}
	 \end{align}
	where, $C_{\bs{\varphi}}^{nc}$ and $C_{\bs{\eta}}^{nc}$ are positive constants independent of $h$such that
	\begin{align}\label{adjoint_condition}
		\alpha_a^{nc} > \frac{C_e^{nc} M^{nc}_{\bs{y}}}{\alpha_2^{nc}} \left(C_{\nu_T} C_e^{nc} \tilde{M} + C_{F_{\bs{y}}} C_p^{nc} \right) + C_e^{nc}  M_{\bs{u}}^{nc}.
	\end{align}
	Here $M_{\bs{u}}^{nc}$ and $M_{\bs{y}}^{nc}$ are the energy estimates for the discrete state variables $\bs{u}_h$ and $\bs{y}_h$ and are explicitly defined in \cite[Lemma 5.3]{TKM_ncDDF}.
	Furthermore, let $(\bs{\varphi}_h(\bs{U}_i), \bs{\eta}_h(\bs{U}_i))$ be solutions of \eqref{nc:P:Ared} corresponding to $\bs{U}_i \in \bs{L}^r(\Omega)$ for $i = 1, 2$, then the following stability result holds:
	\begin{align}\label{adjoint_stability}
		&\norm{\bs{\varphi}_h(\bs{U}_1) - \bs{\varphi}_h(\bs{U}_2)}_{1,\mathcal{T}_h^{nc}} + \norm{\bs{\eta}_h(\bs{U}_1) - \bs{\eta}_h(\bs{U}_2)}_{1,\mathcal{T}_h^{nc}} \nonumber\\&\leq C \left(\norm{\bs{u}_h(\bs{U}_1) - \bs{u}_h(\bs{U}_2)}_{1,\mathcal{T}_h^{nc}}+ \norm{\bs{y}_h(\bs{U}_1) - \bs{y}_h(\bs{U}_2)}_{1,\mathcal{T}_h^{nc}}\right).
	\end{align}
\end{lemma}

We rewrite the discrete adjoint equation \eqref{nc:P:Ared} as follows:
\begin{align}\label{nc:P:A:rewrite}
	B^{nc}((\bs{\varphi}_h, \bs{\eta}_h), (\bs{v}_h, \bs{s}_h)) = \bs{f}^{nc}(\bs{v}_h,  \bs{s}_h),
\end{align}
where,
\begin{align*}
	B^{nc}((\bs{\varphi}_h, \bs{\eta}_h), (\bs{v}_h, \bs{s}_h)) &:=  a^h({\bs{y}_h},{\bs{\varphi}_h},\bs{v}_h) - c^h({\bs{u}_h}, {\bs{\varphi}_h},\bs{v}_h) + c^h(\bs{v}_h,{\bs{u}_h},\bs{\varphi}_h)  \\ 
	&\hspace{-1cm}+ a^h_{\bs{y}}({\bs{\eta}_h},\bs{s}_h) - c^h_{\bs{y}}(\bs{u}_h, \bs{\eta}_h,\bs{s}_h) + c^h_{\bs{y}}(\bs{v}_h, \bs{y}_h, \bs{\eta}_h) - ((\bs{F}_{\bs{y}}({\bs{y}_h}))^{\top} {\bs{\varphi}_h}, \bs{s}_h)\\
	&\hspace{-1cm}+ (\nu_{T}(T_h) \nabla_h \bs{u}_h : \nabla_h \bs{\varphi}_h, \bs{s}_h), \\
	\bs{f}^{nc}(\bs{v}_h, \bs{s}_h) &:= (\bs{u}_h-\bs{u}_d, \bs{v}_h) + (\bs{y}_h-\bs{y}_d, \bs{s}_h).
\end{align*}

	The boundedness of $B^{nc}$ and $\bs{f}^{nc}$ follows from the boundedness properties discussed in Lemma \ref{ncStateCty} and Section \ref{Sec:DSP:A}, assumption \eqref{ncUniquenessAssumption} and energy estimates in \cite[Lemma 5.2]{TKM_ncDDF}. The well-posedness of the adjoint equation \eqref{nc:P:A:rewrite} is a direct consequence of the Lax-Milgram theorem:
\begin{theorem}\label{nc:AdjointWellPosedness}
	Let $(\bar{\bs{u}}_h, \bar{\bs{y}}_h) \in \bs{X}_h^{cr} \times \mathcal{M}_h^{cr}$ be the states associated with the control $\bar{\bs{U}} \in \bs{\mathcal{U}}_{ad}$, assume \eqref{ncUniquenessAssumption} and $\alpha_a^{nc} > \mathcal M_1$, and $\alpha_2^{nc} > \mathcal M_2$ (see \eqref{eq:M1M2}). For every  $\bs{f}^{nc}$, there exists unique $(\bs{\varphi}_h, \bs{\eta}_h) \in \bs{X}_h^{cr}  \times \mathcal{M}_h^{cr_0}$  such that \eqref{nc:P:A:rewrite} holds
	with $\tnorm{(\bs{v}_h, \bs{s}_h)}^2_{nc} := \norm{\bs{v}_h}^2_{1,\mathcal{T}_{h}^{nc}} + \norm{\bs{s}_h}^2_{1,\mathcal{T}_{h}^{nc}}.$
\end{theorem}
\begin{proof}
		It is enough to show the coercivity of $B^{nc}.$ Choosing $\bs{v}_h = \bs{\varphi}_h, \bs{s}_h = \bs{\eta}_h$, and using the continuity properties discussed in Section \ref{Sec:DSP:A}, Lemmas \ref{ncStateCty} and energy estimates for the state equation \cite[Lemma 5.3]{TKM_ncDDF} as follows:
	\begin{align*}
		\small
		B^{nc}((\bs{\varphi}_h, \bs{\eta}_h), (\bs{\varphi}_h, \bs{\eta}_h)) &\geq \alpha_a^{nc} \norm{\bs{\varphi}}^2_{1,\mathcal{T}_h^{nc}} - C_e^{nc} M_{\bs{u}}^{nc} \norm{\bs{\varphi}_h}^2_{1,\mathcal{T}_h^{nc}} + \alpha_2^{nc} \norm{\bs{\eta}_h}^2_{1,\mathcal{T}_h^{nc}} \\
		&\hspace{-1.5cm}- \frac{C_e^{nc} M_{\bs{y}}^{nc}}{2} \left(\norm{\bs{\varphi}_h}^2_{1,\mathcal{T}_h^{nc}} + \norm{\bs{\eta}_h}^2_{1,\mathcal{T}_h^{nc}}\right) - \frac{C_{F_{\bs{y}}} C_p^{nc}}{2} \left(\norm{\bs{\varphi}_h}^2_{1,\mathcal{T}_h^{nc}} + \norm{\bs{\eta}_h}^2_{1,\mathcal{T}_h^{nc}}\right) \\
		&\hspace{-1.5cm}- \frac{C_{\nu_T} C_e^{nc} \tilde{M}}{2} \left(\norm{\bs{\varphi}_h}^2_{1,\mathcal{T}_h^{nc}} + \norm{\bs{\eta}_h}^2_{1,\mathcal{T}_h^{nc}}\right)\\
		&\hspace{-1.5cm}\geq (\alpha_a^{nc} - \mathcal{M}_1) \norm{\bs{\varphi}_h}^2_{1,\mathcal{T}_h^{nc}} + (\alpha_2^{nc} - \mathcal{M}_2) \norm{\bs{\eta}_h}^2_{1,\mathcal{T}_h^{nc}} \geq C_2^{nc} \tnorm{(\bs{\varphi}_h, \bs{\eta}_h)}_{nc}^2
	\end{align*} 
	such that $\alpha_a^{nc} \geq \mathcal{M}_1$ and $\alpha_2^{nc} \geq \mathcal{M}_2$ where,
	\begin{align*}\label{eq:M1M2}
		\mathcal{M}_2 := \frac{C_e^{nc} M_{\bs{y}}^{nc}}{2} + \frac{C_{F_{\bs{y}}} C_p^{nc}}{2} + \frac{C_{\nu_T} C_e^{nc} \tilde{M}}{2} \;\; \mbox{and} \;\; \mathcal{M}_1 := C_e^{nc} M_{\bs{u}}^{nc} + \mathcal{M}_2.
	\end{align*}
\end{proof}
\subsection{A priori error analysis} Let $\Pi_h: Q \longrightarrow Q_h^0$ be such that $(\Pi_h p - p, q_h) = 0$ for all $p \in Q$ and $q_h \in Q_h^0$ denote the pressure interpolation operator, $\bs{\Pi}_{nc}$ defined in \cite[(4.2)]{TKM_ncDDF} denote the nonconforming interpolation operator and without loss of generality let $\bs{P}_e$ denote the standard edge projection on to piecewise constants on $e$. Then under adequate regularity assumptions, the following approximation properties hold (cf. \cite{MR3213584,Ern_Book,EdgeProjection}):
\begin{align}\label{nc:approximation_properties}
	\begin{aligned}
		&\norm{\bs{\varphi} - \bs{\Pi}_{nc} \bs{\varphi}}_{1,\mathcal{T}_h^{nc}} \leq C (h^{3/2+\delta} \sqrt{\sigma} + h^{1/2+\delta} \sqrt{\nu_2}) \norm{\bs{\varphi}}_{3/2+\delta,\Omega}, \\
		&\norm{\bs{\eta} - \bs{\Pi}_{nc} \bs{\eta}}_{1,\mathcal{T}_h^{nc}} \leq C h^{1/2+\delta} \sqrt{\bar{\sigma}} \norm{\bs{\eta}}_{3/2+\delta,\Omega}, \;\; \norm{\xi - \Pi_h \xi}_{0,\Omega} \leq C h^{1/2+\delta} \norm{\xi}_{1/2+\delta,\Omega}, \\
		& \norm{\bs{v} - \bs{P}_e \bs{v}}_{0,e} \leq C h^{s-1/2}_{K} |\bs{v}|_{s,K} \;\; (s > 1/2),  \; \forall \; \bs{v} \in \bs{H}^1(K) \;\left(\mbox{or,}\; \bs{v} \in \left[H^1(K)\right]^2\right),  \; e \in \mathcal{E}(\mathcal{T}_h).
	\end{aligned}
\end{align}
\begin{theorem}\label{NA_Adjoint}
 Assume \eqref{ncUniquenessAssumption} with $\tilde{M} > 0$ such that smallness condition \cite[(5.8)]{TKM_ncDDF} holds.  Also, for some $\bs{U} \in \bs{L}^r(\Omega)$, $(\bs{u}, \bs{y}) \in \bs{H}_0^1(\Omega) \cap \bs{H}^{3/2 + \delta}(\Omega) \times \left[H^{3/2 + \delta}(\Omega)\right]^2$ and $(\bs{u}_h, \bs{y}_h) \in \bs{V}_h^{cr_0} \times \mathcal{M}_{h}^{cr}$, let $(\bs{\varphi}, \xi, \bs{\eta})$ and $(\bs{\varphi}_h, \xi_h, \bs{\eta}_h)$ be the solutions of \eqref{nc:P:Ared} and \eqref{nc:P:A}, respectively with the regularity of $(\bs{\varphi}, \xi, \bs{\eta})$ given in Remark \ref{adjointRegularity}. 
 Then under \eqref{adjoint_condition} and for $\bs{\varphi}_h \in \bs{V}_h^{cr_0} \cap \bs{W}^{1,3}(\mathcal{T}_h^{nc}),$
 there exists positive constants $C$ independent of the mesh parameter $h$ such that
\begin{align}
	\norm{\bs{\varphi} - \bs{\varphi}_h}_{1,\mathcal{T}_h^{nc}} + \norm{\bs{\eta} - \bs{\eta}_h}_{1,\mathcal{T}_h^{nc}} &\leq C h^{1/2+\delta} \big(\norm{\bs{\varphi}}_{3/2+\delta,\Omega} + \norm{\bs{\eta}}_{3/2+\delta,\Omega}\big), \label{aPriori_varphieta} \\
	\norm{\xi - \xi_h}_{0,\Omega} &\leq C h^{1/2+\delta} \big(\norm{\bs{\varphi}}_{3/2+\delta,\Omega} + \norm{\bs{\eta}}_{3/2+\delta,\Omega} + \norm{\xi}_{1/2+\delta,\Omega}\big), \label{aPriori_zeta}
\end{align} 
\end{theorem}
\begin{proof}
	Let us decompose the errors into
	\begin{align}\label{aPriori_eq0:Ad}
		\begin{aligned}
			e_{\bs{\varphi}} := (\bs{\varphi} - \bs{\chi}) + (\bs{\chi} - \bs{\varphi}_h) &= \hat{e}_{\bs{\varphi}} + \tilde{e}_{\bs{\varphi}}, \;\; e_{\bs{\eta}} := (\bs{\eta} - \bs{\theta}) + (\bs{\theta} - \bs{\eta}_h) = \hat{e}_{\bs{\eta}} + \tilde{e}_{\bs{\eta}}, \\ 
			&\hspace{-1cm}e_{\xi} := (\xi - \psi) + (\psi - \xi_h) = \hat{e}_{\xi} + \tilde{e}_{\xi}, 
		\end{aligned}
	\end{align}
	where, $(\bs{\chi}, \psi, \bs{\theta}) \in \bs{V}_h^{cr} \times Q_h^0 \times \mathcal{M}_h^{cr}$. Now an application of integration by parts implies that $(\bs{\varphi}, \xi, \bs{\eta})$ satisfies the following equation:
	\begin{align}
		a^h(\bs{y},\bs{\varphi},\bs{v}_h) - c^h(\bs{u}, \bs{\varphi},\bs{v}_h) + c^h(\bs{v}_h,\bs{u},\bs{\varphi}) + b^h(\bs{v}_h, \zeta)- c^h_{\bs{y}}(\bs{v}_h,\bs{y},\bs{\eta}) &= (\bs{u}-\bs{u}_d,\bs{v}_h) \nonumber \\
		&\hspace{-2.5cm}+ \sum_{K \in \mathcal{T}} \int_{\partial K} \nu(T_h) (\nabla \bs{\varphi} \bs{n}_K) \cdot \bs{v}_h \; ds,\label{aPriori_eq1:a:Ad} \\
		a^h_{\bs{y}}(\bs{\eta},\bs{s}_h) - c^h_{\bs{y}}(\bs{u},\bs{\eta},\bs{s}_h) + (\nu_{T}(T) \nabla_h \bs{u} : \nabla_h \bs{\varphi}, \bs{s}_h) - ((\bs{F}_{\bs{y}}(\bs{y}))^{\top} \bs{\varphi}, \bs{s}_h) &= (\bs{y}-\bs{y}_d,\bs{s}_h) \nonumber\\
		&\hspace{-2cm}+ \sum_{K \in \mathcal{T}_h} \int_{\partial K} \bs{D} (\nabla \bs{\eta}  \bs{n}_K) \cdot \bs{s}_h \; ds.\label{aPriori_eq1:b:Ad}
	\end{align}
	
	Now similar to the proof of Theorem \ref{NA_State}, we subtract \eqref{aPriori_eq1:a:Ad} and \eqref{aPriori_eq1:b:Ad} from their respective discrete counterparts, testing the resulting equations with $\bs{v}_h = \tilde{e}_{\bs{\varphi}}$ and $\bs{s}_h = \tilde{e}_{\bs{\eta}}$, respectively and rearranging the terms using the error decomposition in \eqref{aPriori_eq0:Ad} yields the following error equations:
	\begin{align*}
		a^h(\bs{y}_h,\tilde{e}_{\bs{\varphi}},\tilde{e}_{\bs{\varphi}}) - c^h(\bs{u}_h,\tilde{e}_{\bs{\varphi}},\tilde{e}_{\bs{\varphi}}) &= T_a + T_b + T_c + T_d + T_e + T_f,  \\
		a^h_{\bs{y}}(\tilde{e}_{\bs{\eta}}, \tilde{e}_{\bs{\eta}}) - c^h_{\bs{y}}(\bs{u}_h, \tilde{e}_{\bs{\eta}}, \tilde{e}_{\bs{\eta}}) &= I_a + I_b + I_c + I_d + I_e.
	\end{align*}
	where
\begin{align*}
	&T_a := a^h(\bs{y}_h; \bs{\varphi}_h, \tilde{e}_{\bs{\varphi}}) -  a^h(\bs{y}; \bs{\varphi}_h, \tilde{e}_{\bs{\varphi}}), \quad T_b := - c^h(\bs{u}_h, \hat{e}_{\bs{\varphi}},  \tilde{e}_{\bs{\varphi}}) + c^h(\bs{u}_h - \bs{u}, \bs{\varphi},  \tilde{e}_{\bs{\varphi}}), \\
	&T_c :=  c^h(\tilde{e}_{\bs{\varphi}}, \bs{u}_h, \tilde{e}_{\bs{\varphi}}), \quad T_d := c^h_{\bs{y}}(\tilde{e}_{\bs{\varphi}}, \bs{y} - \bs{y}_h, \bs{\eta}) + c^h_{\bs{y}}(\tilde{e}_{\bs{\varphi}}, \bs{y}_h, \hat{e}_{\bs{\eta}}) + c^h_{\bs{y}}(\tilde{e}_{\bs{\varphi}}, \bs{y}_h, \tilde{e}_{\bs{\eta}}), \\ &T_e := (\bs{u}_h - \bs{u}, \tilde{e}_{\bs{\varphi}}), \quad
	T_f := \sum_{K \in {\mathcal{T}_h}} \int_{\partial K} \nu(T) (\nabla \bs{\varphi}  \bs{n}_K) \cdot \tilde{e}_{\bs{\varphi}} \; ds, \\ &I_a := -c^h_{\bs{y}}(\bs{u}, \hat{e}_{\bs{\eta}}, \tilde{e}_{\bs{\eta}}) - c^h_{\bs{y}}(\bs{u} - \bs{u}_h, \bs{\eta}_h, \tilde{e}_{\bs{\eta}}), \quad I_d := (\bs{y} - \bs{y}_h, \tilde{e}_{\bs{\eta}}) ,\\
	& I_b := (\nu_T(T) \nabla_h(\bs{u} - \bs{u}_h) : \nabla_h \bs{\varphi}, \tilde{e}_{\bs{\eta}}) + ((\nu_T(T) - \nu_T(T_h)) \nabla_h \bs{u}_h : \nabla_h \bs{\varphi}_h, \tilde{e}_{\bs{\eta}}) + (\nu_T(T) \nabla_h \bs{u}_h : \nabla_h e_{\bs{\varphi}}, \tilde{e}_{\bs{\eta}}),\\
	& I_c := (\bs{F}_{\bs{y}}(\bs{y})^{\top}(\bs{\varphi} - \bs{\varphi}_h), \tilde{e}_{\bs{\eta}}) + (((\bs{F}_{\bs{y}}(\bs{y}))^{\top} - ((\bs{F}_{\bs{y}}(\bs{y}_h))^{\top})\; \bs{\varphi}_h, \tilde{e}_{\bs{\eta}}), \quad I_e :=  \sum_{K \in {\mathcal{T}_h}} \int_{\partial K} \bs{D} (\nabla \bs{\eta}  \bs{n}_K) \cdot \tilde{e}_{\bs{\eta}} \; ds.
\end{align*}

For $T_a$, using the Lipschitz continuity of $\nu$, H\"older's
inequality with exponents $(6,3,2)$, and the discrete Sobolev
embedding, we have
\begin{align}
|T_a| &= \left|\sum_{K\in\mathcal T_h} \int_K \big(\nu(T_h)-\nu(T)\big) \nabla_h\bs{\varphi}_h:\nabla_h\tilde e_{\bs{\varphi}}\right| \nonumber\\
&\leq \gamma_\nu \norm{T-T_h}_{L^6(\Omega)} \norm{\nabla_h\bs{\varphi}_h}_{L^3(\Omega)}
\left( \sum_{K\in\mathcal T_h} \norm{\nabla_h\tilde e_{\bs{\varphi}}}_{0,K}^2 \right)^{1/2}\nonumber\\
&\leq \gamma_\nu C_e^{nc} \norm{\nabla_h\bs{\varphi}_h}_{L^3(\Omega)} \norm{\bs{y}-\bs{y}_h}_{1,\mathcal T_h^{nc}} \norm{\tilde e_{\bs{\varphi}}}_{1,\mathcal T_h^{nc}}. \label{eq:Ta-bound}
\end{align}

For $T_b$, the continuity of $c^h$ and the same arguments used for
the convective terms in \cite[Theorem 6.1]{TKM_ncDDF} yield
\begin{align}
|T_b| &\leq |c^h(\bs{u}_h,\hat e_{\bs{\varphi}}, \tilde e_{\bs{\varphi}})| +
|c^h(\bs{u}_h-\bs{u},\bs{\varphi}, \tilde e_{\bs{\varphi}})| \nonumber\\
&\leq C_e^{nc} \left( M_{\bs{u}}^{nc} \norm{\hat e_{\bs{\varphi}}}_{1,\mathcal T_h^{nc}}
+ M_{\bs{\varphi}} \norm{\bs{u}-\bs{u}_h}_{1,\mathcal T_h^{nc}} \right)
\norm{\tilde e_{\bs{\varphi}}}_{1,\mathcal T_h^{nc}}. \label{eq:Tb-bound}
\end{align}

Similarly,
\begin{align}
|T_c| = |c^h(\tilde e_{\bs{\varphi}},\bs{u}_h, \tilde e_{\bs{\varphi}})|
\leq C_e^{nc}M_{\bs{u}}^{nc} \norm{\tilde e_{\bs{\varphi}}}_{1,\mathcal T_h^{nc}}^2 \label{eq:Tc-bound}
\end{align}

For $T_d$, we estimate the three terms separately. Using the
continuity of $c_{\bs{y}}^h$, the continuous adjoint bound
$\norm{\bs{\eta}}_{1,\Omega}\leq M_{\bs{\eta}}$, and the discrete
state bound, we obtain
\begin{align}
|c_{\bs{y}}^h(\tilde e_{\bs{\varphi}},
        \bs{y}-\bs{y}_h,\bs{\eta})|
&\leq
C_e^{nc}M_{\bs{\eta}}
\norm{\bs{y}-\bs{y}_h}_{1,\mathcal T_h^{nc}}
\norm{\tilde e_{\bs{\varphi}}}_{1,\mathcal T_h^{nc}},
\\
|c_{\bs{y}}^h(\tilde e_{\bs{\varphi}},
        \bs{y}_h,\hat e_{\bs{\eta}})|
&\leq
(C_e^{nc})^2 M_{\bs{y}}^{nc}
\norm{\hat e_{\bs{\eta}}}_{1,\mathcal T_h^{nc}}
\norm{\tilde e_{\bs{\varphi}}}_{1,\mathcal T_h^{nc}},
\\
|c_{\bs{y}}^h(\tilde e_{\bs{\varphi}},
        \bs{y}_h,\tilde e_{\bs{\eta}})|
&\leq
(C_e^{nc})^2 M_{\bs{y}}^{nc}
\norm{\tilde e_{\bs{\eta}}}_{1,\mathcal T_h^{nc}}
\norm{\tilde e_{\bs{\varphi}}}_{1,\mathcal T_h^{nc}}.
\end{align}
Consequently,
\begin{align}
|T_d|
\leq
C_e^{nc}
\Big[
M_{\bs{\eta}}
\norm{\bs{y}-\bs{y}_h}_{1,\mathcal T_h^{nc}}
+
C_e^{nc}M_{\bs{y}}^{nc}
\big(
\norm{\hat e_{\bs{\eta}}}_{1,\mathcal T_h^{nc}}
+
\norm{\tilde e_{\bs{\eta}}}_{1,\mathcal T_h^{nc}}
\big)
\Big]
\norm{\tilde e_{\bs{\varphi}}}_{1,\mathcal T_h^{nc}}.
\label{eq:Td-bound}
\end{align}

For the right-hand side term,
\begin{align}
|T_e|=|(\bs{u}_h-\bs{u},\tilde e_{\bs{\varphi}})|
\leq \bar C_d^{nc} \norm{\bs{u}-\bs{u}_h}_{0,\Omega} \norm{\tilde e_{\bs{\varphi}}}_{1,\mathcal T_h^{nc}}.
\label{eq:Te-bound}
\end{align}

The consistency term $T_f$ is handled exactly as $T_4$ in
\cite[Theorem 6.1]{TKM_ncDDF}. 
\begin{align}
|T_f|
&\leq
C\sqrt{\nu_2}\,h
\left(
\sum_{e\in\mathcal E_{\mathcal T_h}}
|e|^{-1}
\norm{
\nabla\bs{\varphi}\bs{n}_e
-\bs P_e(\nabla\bs{\varphi}\bs{n}_e)
}_{0,e}^2
\right)^{1/2}
\norm{\tilde e_{\bs{\varphi}}}_{1,\mathcal T_h^{nc}}.
\label{eq:Tf-bound}
\end{align}

We have $\nabla \bs{\varphi} \cdot \bs{n}_e \in H^{1/2+\delta}(K)$ for every element $K$ adjacent to $e$. Then \eqref{nc:approximation_properties} gives
\begin{align*}
	\norm{\nabla \bs{\varphi} \cdot \bs{n}_e - \bs{P}_e(\nabla \bs{\varphi} \cdot \bs{n}_e)}_{0,e} \leq C h_e^{\delta} |\nabla \bs{\varphi} \cdot \bs{n}_e|_{1/2 + \delta, K} \leq C h_e^{\delta} |\bs{\varphi}|_{3/2+\delta,K}. 
\end{align*}
For an interior edge $e = K^+ \cap K^-$, we can rewrite
\begin{align*}
 \norm{\nabla \bs{\varphi} \cdot \bs{n}_e - \bs{P}_e(\nabla \bs{\varphi} \cdot \bs{n}_e)}_{0,e}^2 \leq C h_e^{2\delta} \left(|\bs{\varphi}|^2_{3/2+\delta,K^+} + |\bs{\varphi}|^2_{3/2+\delta,K^-}\right). 
\end{align*}
Subsequently using $|e| \simeq h_e \simeq h$, we obtain 
\begin{align*}
	|T_f| \leq C \sqrt{\nu_2} h^{1/2+\delta} | \bs{\varphi} |_{3/2+\delta,\Omega} \norm{\tilde e_{\bs{\varphi}}}_{1,\mathcal T_h^{nc}}.
\end{align*}

For $I_a$, the continuity of
$c_{\bs{y}}^h$ gives
\begin{align}
|I_a|
&\leq
|c_{\bs{y}}^h(\bs{u},\hat e_{\bs{\eta}},
                    \tilde e_{\bs{\eta}})|
+
|c_{\bs{y}}^h(\bs{u}-\bs{u}_h,\bs{\eta}_h,
                    \tilde e_{\bs{\eta}})|
\nonumber\\
&\leq
\Big[
C_e^{nc}M_{\bs{u}}
\norm{\hat e_{\bs{\eta}}}_{1,\mathcal T_h^{nc}}
+
(C_e^{nc})^2M_{\bs{\eta}}^{nc}
\norm{\bs{u}-\bs{u}_h}_{1,\mathcal T_h^{nc}}
\Big]
\norm{\tilde e_{\bs{\eta}}}_{1,\mathcal T_h^{nc}}.
\label{eq:Ia-bound}
\end{align}

We split $I_b=I_{b,1}+I_{b,2}+I_{b,3}$. For the first term we get,
\begin{align}
|I_{b,1}|
&=
\left|
(\nu_T(T)\nabla_h(\bs{u}-\bs{u}_h):
\nabla\bs{\varphi},\tilde e_{\bs{\eta}})
\right|
\leq
C_{\nu_T}
\norm{\nabla_h(\bs{u}-\bs{u}_h)}_{0,\Omega}
\norm{\nabla\bs{\varphi}}_{L^3(\Omega)}
\norm{\tilde e_{\bs{\eta}}}_{L^6(\Omega)}
\nonumber\\
&\leq
C_{\nu_T}C_e^{nc}M_{\bs{\varphi}}
\norm{\bs{u}-\bs{u}_h}_{1,\mathcal T_h^{nc}}
\norm{\tilde e_{\bs{\eta}}}_{1,\mathcal T_h^{nc}}.
\label{eq:Ib1-bound}
\end{align}

For the second term, 
\begin{align}
|I_{b,2}| &= \left| ((\nu_T(T)-\nu_T(T_h)) \nabla_h\bs{u}_h:\nabla_h\bs{\varphi}_h ,\tilde e_{\bs{\eta}})\right| \nonumber\\
&\leq C_{\nu_{TT}} \norm{T-T_h}_{L^6(\Omega)} \norm{\nabla_h\bs{u}_h}_{L^3(\Omega)}
\norm{\nabla_h\bs{\varphi}_h}_{L^3(\Omega)} \norm{\tilde e_{\bs{\eta}}}_{L^6(\Omega)}
\nonumber\\
&\leq C_{\nu_{TT}}(C_e^{nc})^2 \widetilde M \norm{\nabla_h\bs{\varphi}_h}_{L^3(\Omega)}
\norm{\bs{y}-\bs{y}_h}_{1,\mathcal T_h^{nc}} \norm{\tilde e_{\bs{\eta}}}_{1,\mathcal T_h^{nc}}.
\label{eq:Ib2-bound}
\end{align}

Finally, since
$e_{\bs{\varphi}}=\hat e_{\bs{\varphi}} +\tilde e_{\bs{\varphi}}$,
\begin{align}
|I_{b,3}|=\left|(\nu_T(T)\nabla_h\bs{u}_h:\nabla_h e_{\bs{\varphi}},\tilde e_{\bs{\eta}})
\right| \leq C_{\nu_T}C_e^{nc}\widetilde M
\left(\norm{\hat e_{\bs{\varphi}}}_{1,\mathcal T_h^{nc}}
+ \norm{\tilde e_{\bs{\varphi}}}_{1,\mathcal T_h^{nc}}\right)
\norm{\tilde e_{\bs{\eta}}}_{1,\mathcal T_h^{nc}}.
\label{eq:Ib3-bound}
\end{align}

Similarly, split $I_c=I_{c,1}+I_{c,2}$. The first term is bounded
using the estimate from Section~\ref{Sec:DSP:A}:
\begin{align}
|I_{c,1}|=
\left|((F_{\bs{y}}(\bs{y}))^\top(\bs{\varphi}-\bs{\varphi}_h),\tilde e_{\bs{\eta}})
\right|
\leq C_{F_{\bs{y}}}C_p^{nc} \left(\norm{\hat e_{\bs{\varphi}}}_{1,\mathcal T_h^{nc}}
+ \norm{\tilde e_{\bs{\varphi}}}_{1,\mathcal T_h^{nc}}\right) \norm{\tilde e_{\bs{\eta}}}_{1,\mathcal T_h^{nc}}.
\label{eq:Ic1-bound}
\end{align}

For the second term, boundedness of $F_{\bs{y}\bs{y}}$ gives
\begin{align}
|I_{c,2}|
&=
\left|
\left(
\big(
F_{\bs{y}}(\bs{y})
-
F_{\bs{y}}(\bs{y}_h)
\big)^\top
\bs{\varphi}_h,
\tilde e_{\bs{\eta}}
\right)
\right|
\leq
C_{F_{\bs{y}\bs{y}}}
\norm{\bs{y}-\bs{y}_h}_{L^4(\Omega)}
\norm{\bs{\varphi}_h}_{L^4(\Omega)}
\norm{\tilde e_{\bs{\eta}}}_{0,\Omega}
\nonumber\\
&\leq
C_{F_{\bs{y}\bs{y}}}
(C_e^{nc})^2\widehat C_d^{nc}M_{\bs{\varphi}}^{nc}
\norm{\bs{y}-\bs{y}_h}_{1,\mathcal T_h^{nc}}
\norm{\tilde e_{\bs{\eta}}}_{1,\mathcal T_h^{nc}}.
\label{eq:Ic2-bound}
\end{align}

For $I_d$, we can obtain
\begin{align}
|I_d| = |(\bs{y}-\bs{y}_h,\tilde e_{\bs{\eta}})|
\leq \widehat C_d^{nc} \norm{\bs{y}-\bs{y}_h}_{0,\Omega}
\norm{\tilde e_{\bs{\eta}}}_{1,\mathcal T_h^{nc}}.
\label{eq:Id-bound}
\end{align}

Finally, proceeding as for $I_3$ in
\cite[Theorem 6.1]{TKM_ncDDF}, we obtain
\begin{align}
|I_e| &\leq C \sqrt{\bar\sigma}\;h \left( \sum_{e\in\mathcal E_{\mathcal T_h}}|e|^{-1}
\norm{\nabla\bs{\eta}\bs{n}_e-\bs P_e(\nabla\bs{\eta}\bs{n}_e)}_{0,e}^2\right)^{1/2}
\norm{\tilde e_{\bs{\eta}}}_{1,\mathcal T_h^{nc}} \nonumber \\
&\leq C \sqrt{\bar \sigma} h^{1/2+\delta} |\bs{\eta}|_{3/2+\delta,\Omega}\norm{\tilde e_{\bs{\eta}}}_{1,\mathcal T_h^{nc}}.  \label{eq:Ie-bound}
\end{align}

Choosing $\bs{\chi} = \bs{\Pi}_{nc} \bs{\varphi}, \bs{\theta} = \bs{\Pi}_{nc} \bs{\eta}$ and invoking the approximation properties given in \eqref{nc:approximation_properties} and Theorem \ref{NA_State} lead to the estimate \eqref{aPriori_varphieta}. To derive the adjoint pressure estimate we exploit the discrete inf-sup condition \cite[Section 5 - Example 4, and Section 6 - Hypothesis H.3.]{CrouzeixRaviart1973} similar to the proof of \cite[Theorem 6.1]{TKM_ncDDF} as follows:
\begin{align}\label{aPriori_eq9:Ad}
	\norm{\tilde{e}_{\xi}}_{0,\Omega} \leq \frac{1}{\tilde{\beta}} \sup_{\bs{v}_h \in \bs{V}_h \backslash \{\bs{0}\}} \frac{b^h(\bs{v}_h, e_{\xi})}{\norm{v_h}_{1,\mathcal{T}_h^{nc}}} + \frac{1}{\tilde{\beta}} \norm{\hat{e}_{\xi}}_{0,\Omega}.
\end{align}
For any $\bs{v}_h \in \bs{V}_h^{cr_0}$, we have
\begin{align}
	b^h(\bs{v}_h, \xi_h - \xi) &= a^{h}(\bs{y}_h; e_{\bs{\varphi}}, \bs{v}_h) + c^{h}(\bs{u}_h, e_{\bs{u}}, \bs{v}_h) + c^{h}(\bs{v}_h, \bs{u}_h, e_{\bs{\varphi}}) \nonumber\\
	&\quad - c_{\bs{y}}(\bs{v}_h, \bs{y}_h, e_{\bs{\eta}}) - \sum_{K \in {\mathcal{T}_h}} \int_{\partial K} \nu(T_h) (\nabla \bs{\varphi}  \bs{n}_K) \cdot \bs{v}_h \; ds. 
\end{align}
Utilizing the continuity properties of  and analogous steps used to bound $|T_4|$, we get
\begin{align}\label{aPriori_eq10:Ad}
	|b^h(\bs{v}_h, \xi_h-\xi)| &\leq \norm{\bs{v}_h}_{1,\mathcal{T}_h^{nc}} \Big[ C_a^{nc}  \norm{e_{\bs{\varphi}}}_{1,\mathcal{T}_h^{nc}}  + C_e^{nc} M_{\bs{u}}^{nc} \norm{e_{\bs{\varphi}}}_{1,\mathcal{T}_h^{nc}}  + C_e^{nc} M_{\bs{y}}^{nc} \norm{e_{\bs{\eta}}}_{1,\mathcal{T}_h^{nc}} \nonumber\\
	&\quad  + C \nu_2 h \Big(\sum_{e \in {\mathcal{E}_{\mathcal{T}_h}}}  |e|^{-1} \|{\nabla \bs{\varphi}  \bs{n}_{e} - \bs{P}_e \left(\nabla \bs{\varphi}  \bs{n}_{e}\right)\|}^2_{0,e}\Big)^{1/2} \Big].
\end{align}
The estimate \eqref{aPriori_zeta} follows by using \eqref{aPriori_eq10:Ad} in \eqref{aPriori_eq9:Ad}, choosing $\psi = \Pi_h \xi$ and using \eqref{aPriori_varphieta} and \eqref{nc:approximation_properties}.

\end{proof}

\section{Numerical analysis of the control problem}\label{NA:Control}
In this section we propose {\it explicit} and {\it implicit} discretization approaches for the the control variable and establish \textit{a priori} error estimates for the resulting control problem. The second approach also known as the {\it variational} approach was first introduced in \cite{Hinze}. It implicitly discretizes the control by projecting the optimal discrete adjoint state into $\bs{\mathcal{U}}_{ad}$.
\subsection{Explicit control discretization} Let $\bs{\mathcal{U}}_h \subset \bs{\mathcal{U}}$ denote the discrete control space. In the explicit case it is defined as,
$$\bs{\mathcal{U}}_h := \left\{\bs{U}_h \in \bs{\mathcal{U}}: \bs{U}_h|_{K} \in \mathbb{P}_0(K) \;\forall\; K \in \mathcal{T}_h\right\}.$$
The set of admissible controls is then defined as $\bs{\mathcal{U}}_{h,ad} := \bs{\mathcal{U}}_h \cap \bs{\mathcal{U}}_{ad}$. Recall that we use the following notations, $\bs{u}_h := \bs{u}_h(\bs{U}_h), p_h := p_h(\bs{U}_h), \bs{y}_h := \bs{y}_h(\bs{U}_h), \bs{\varphi}_h := \bs{\varphi}_h(\bs{U}_h), \xi_h := \xi_h(\bs{U}_h), \bs{\eta}_h := \bs{\eta}_h(\bs{U}_h).$ The finite dimensional approximate optimal control problem is defined by
\begin{eqnarray}\label{Ph:CF}
	\min_{\bs{U}_h \in \bs{\mathcal{U}}_{h,ad}} J(\bs{u}_h,\bs{y}_h,\bs{U}_h) = \frac{1}{2} \norm{\bs{u}_h-\bs{u}_d}^2_{0,\Omega} + \frac{1}{2}\norm{\bs{y}_h - \bs{y}_{d}}^2_{0,\Omega} + \frac{\lambda}{2} \norm{\bs{U}_h}^2_{0,\Omega}.
\end{eqnarray}
The discrete problem \eqref{Ph:CF} admits an optimal solution follows from arguments analogous to the proof of Theorem \ref{ExistenceOptimalControl}. Furthermore, the  \textit{optimize-then-discretize} approach results in the following discrete optimality system for the unknowns $(\bar{\bs{u}}_h, p_h, \bar{\bs{y}}_h, \bar{\bs{\varphi}}_h, \xi_h, \bar{\bs{\eta}}_h, \bar{\bs{U}}_h) \in \bs{V}_h^{cr_0} \times Q_h^0 \times \mathcal{M}_h^{cr} \times \bs{V}_h^{cr_0} \times Q_h^0 \times \mathcal{M}_h^{cr_0} \times \bs{\mathcal{U}}_{h,ad}$ satisfying $\bar{\bs{y}}_h|_{\Gamma} = \bs{y}_h^D$, which can be used to determine the discrete optimal control:
\begin{align}\label{DiscreteKKT_state}
	\left\{
	\begin{aligned}
		a^{h}(\bar{\bs{y}}_h; \bar{\bs{u}}_h, \bs{v}_h) + c^{h}(\bar{\bs{u}}_h, \bar{\bs{u}}_h, \bs{v}_h) + b^{h}(\bs{v}_h,p_h) &= d^h(\bar{\bs{y}}_h,\bs{v}_h) + (\bar{\bs{U}}_h,\bs{v}_h)_{\rparam,\rdual} \;\forall\; \bs{v}_h \in \bs{V}_h^{cr_0},\\
		b^h(\bar{\bs{u}}_h,q_h) &= 0 \;\forall\; q_h \in Q_h^0, \\
		a^h_{\bs{y}}(\bar{\bs{y}}_h,\bs{s}_h) + c^h_{\bs{y}}(\bar{\bs{u}}_h,\bar{\bs{y}}_h,\bs{s}_h) &= 0 \; \forall \; \bs{s}_h \in \mathcal{M}_h^{cr_0},
	\end{aligned}
	\right.	
\end{align}
	\begin{align}\label{DiscreteKKT_adjoint}			
		\left\{
		\begin{aligned}
			&a^h(\bar{\bs{y}}_h,\bar{\bs{\varphi}}_h,\bs{v}_h) - c^h(\bar{\bs{u}}_h, \bar{\bs{\varphi}}_h,\bs{v}_h) + c^h(\bs{v}_h,\bar{\bs{u}}_h,\bar{\bs{\varphi}}_h) + b(\bs{v}_h, \xi_h) \\ &\hspace{5cm} +c^h_{\bs{y}}(\bs{v}_h,\bar{\bs{y}}_h,\bar{\bs{\eta}}_h) = (\bar{\bs{u}}_h-\bs{u}_d,\bs{v}_h) \; \forall \; \bs{s}_h \in \bs{V}_h^{cr_0}, \\
			&\hspace{6cm}b^h(\bar{\bs{\varphi}}_h,q_h) = 0 \;\forall\; q_h \in Q_h^0,\\
			&a^h_{\bs{y}}(\bar{\bs{\eta}}_h,\bs{s}_h) - 	c^h_{\bs{y}}(\bar{\bs{u}}_h,\bar{\bs{\eta}}_h,\bs{s}_h) + (\nu_{T}(\bar{T}_h) \nabla_h \bar{\bs{u}}_h : \nabla_h \bar{\bs{\varphi}}_h, \bs{s}_h) \\
			&\hspace{4.24cm}- ((\bs{F}_{\bs{y}}(\bar{\bs{y}}_h))^{\top} \bar{\bs{\varphi}}_h, \bs{s}_h) = (\bar{\bs{y}}_h-\bs{y}_d,\bs{s}_h)\; \forall \; \bs{s}_h \in \mathcal{M}_h^{cr_0},
		\end{aligned}
	\right.
	\end{align}
and
\begin{align}\label{DiscreteKKT_VI}
	\begin{aligned}
		\left(\lambda \bar{\bs{U}}_h + \bar{\bs{\varphi}}_h, \bs{U}_h - \bar{\bs{U}}_h\right) &\geq 0 \; \forall \; \bs{U}_h \in \bs{\mathcal{U}}_{h,ad}.
	\end{aligned}
\end{align}

\subsubsection{A priori error analysis}	We assume in this section that the constraints $U_{a_j}, U_{b_j}$ are real constants. This along with the regularity of adjoint variable $\bar{\bs{\varphi}}$ and projection formula \eqref{projection_formula} imply that optimal control $\bar{\bs{U}} \in \bs{H}^{1}(\Omega)$. Let $\bs{\Pi}_0: \bs{L}^2(\Omega) \longrightarrow \bs{\mathcal{U}}_h$ be defined locally as  follows:
\begin{align}\label{Pi0_defn}
	\left(\bs{\Pi}_0 \bs{U}, \bs{v}\right)_K = \left(\bs{U}, \bs{v}\right)_K \; \forall \; \bs{v} \in \bs{\mathcal{U}}_h.
\end{align}
Notice that $\bs{\Pi}_0 \bs{\mathcal{U}}_{ad} \subset \bs{\mathcal{U}}_{h,ad}$ and has the approximation property (cf. \cite{Casas_Control_Projection}),
\begin{align}\label{Pi0_approximation}
	\| \bs{\Pi}_0 \bs{U} - \bs{U} \|_{0,K} \leq C_{\Pi} h_K \norm{\bs{U}}_{1,K}. 
\end{align}

	Now, we prove the existence of an appropriate solution $\bar{\bs{U}}_h$ in any neighborhood of $\bar{\bs{U}}.$  
	\begin{lemma}\label{weakConvg}
	Let $(\bar{\bs{U}}_h)_{h > 0}$ be any sequence of solutions to \eqref{Ph:CF}. Then there exists a weakly-converging subsequence (still indexed by h), and if the subsequence $(\bar{\bs{U}}_h)_{h > 0}$ is converging weakly to $(\bar{\bs{U}}),$ then $\bar{\bs{U}}$ is a solution of \eqref{P:CF} and
	 \begin{align}\label{Jh2J}
	 	\lim_{h \rightarrow 0} J(\bar{\bs{u}}_h, \bar{\bs{y}}_h, \bar{\bs{U}}_h) = J(\bar{\bs{u}}, \bar{\bs{y}}, \bar{\bs{U}}) =: \bar{J}.
	 \end{align}
	 Moreover, the weak convergence of $(\bar{\bs{U}}_h)_{h > 0}$ to $\bar{\bs{U}}$ and \eqref{Jh2J} implies 
	 \begin{align}\label{norm_convg_L2}
	 		\lim_{h \rightarrow 0} \|\bar{\bs{U}}_h - \bar{\bs{U}}\|_{0,\Omega} = 0.
	 \end{align}
\end{lemma}
\begin{proof}
	Note that $(\bar{\bs{U}}_h)_{h > 0}$ 	is bounded in $\bs{\mathcal{U}}_{h,ad},$ and the bound is independent of $h$, since by the construction it is contained in $\bs{\mathcal{U}}_{ad}.$ Then there exists a subsequence still indexed by $h,$ that converges weakly to $\bar{\bs{U}}$ in $\bs{L}^2(\Omega)$ due to the Banach-Alaoglu theorem. Moreover, $\bar{\bs{U}} \in \bs{\mathcal{U}}_{ad}$ follows from Mazur's theorem and the continuity and convexity of $J$ implies (see \cite[Lemma 4.2]{arada2002error} or \cite[Theorem 2.12]{troltzsch2010optimal})
	\begin{align}\label{weakConvg_Eq0}
	\bar{J} \leq J(\bar{\bs{u}}, \bar{\bs{y}}, \bar{\bs{U}}) \leq \liminf_{h \rightarrow 0} J(\bar{\bs{u}}_h, \bar{\bs{y}}_h, \bar{\bs{U}}_h).
	\end{align}
	Let $\bar{\bs{W}}$ be a solution of \eqref{P:CF}. Then we have $\|\bar{\bs{W}} - \Pi_0 \bar{\bs{W}}\|_{0,\Omega} \leq C h$ due to \eqref{Pi0_approximation}. Thus, (see  \cite[Lemma 4.2, part $(i)$]{arada2002error})
		$\lim_{h \rightarrow 0} J(\bar{\bs{u}}_h, \bar{\bs{y}}_h, \Pi_0 \bar{\bs{W}}) = J(\bar{\bs{u}}, \bar{\bs{y}}, \bar{\bs{W}}) = \bar{J}$ due to Theorem \ref{NA_State} and \eqref{Pi0_approximation}.
	Moreover, $\Pi_0 \bar{\bs{W}}$ is admissible for \eqref{Ph:CF}, and thus $J(\bar{\bs{u}}_h, \bar{\bs{y}}_h, \bar{\bs{U}}_h) \leq J(\bar{\bs{u}}_h, \bar{\bs{y}}_h, \Pi_0 \bar{\bs{W}}).$ Passing the limit in the last inequality, we obtain
	\begin{align}\label{weakConvg_Eq1}
		\liminf_{h \rightarrow 0} J(\bar{\bs{u}}_h, \bar{\bs{y}}_h, \bar{\bs{U}}_h) \leq \limsup_{h \rightarrow 0} J(\bar{\bs{u}}_h, \bar{\bs{y}}_h, \bar{\bs{U}}_h) \leq  \limsup_{h \rightarrow 0} J(\bar{\bs{u}}_h, \bar{\bs{y}}_h, \Pi_0 \bar{\bs{W}})  = J(\bar{\bs{u}}, \bar{\bs{y}}, \bar{\bs{W}}) = \bar{J}.
	\end{align}
	By \eqref{weakConvg_Eq0} and \eqref{weakConvg_Eq1}, we arrive at \eqref{Jh2J}.  In order to prove \eqref{norm_convg_L2} we first observe that
	$$\|\bar{\bs{U}}_h - \bar{\bs{U}}\|^2_{0,\Omega} = \|\bar{\bs{U}}_h\|^2_{0,\Omega} + \|\bar{\bs{U}}\|^2_{0,\Omega} - 2 \langle \bar{\bs{U}}_h - \bar{\bs{U}}, \bar{\bs{U}} \rangle - 2 \|\bar{\bs{U}}\|^2_{0,\Omega}.$$
	Now using \eqref{Jh2J} which implies that $\|\bar{\bs{U}}_h\|_{0,\Omega}^2 \rightarrow \|\bar{\bs{U}}\|^2_{0,\Omega}$ and weak convergence of $(\bar{\bs{U}}_h)_{h > 0}$ to $\bar{\bs{U}}$, we conclude \eqref{norm_convg_L2}.
\end{proof}

\begin{theorem}\label{apriori_ctrl_problem}
	Let $\bar{\bs{U}}$ be a locally optimal control in the sense of the Definition \ref{local_optimal_control} and satisfy the first order necessary and second order sufficient optimality conditions. Let $(\bar{\bs{u}}, p, \bar{\bs{y}}, \bar{\bs{U}})$ be the solution to the optimal control problem \eqref{P:CF}-\eqref{P:GE}, which together with $(\bar{\bs{\varphi}}, \xi, \bar{\bs{\eta}})$ satisfies the continuous optimality system. Further, let $(\bar{\bs{u}}_h, p_h, \bar{\bs{y}}_h, \bar{\bs{\varphi}}_h, \xi_h, \bar{\bs{\eta}}_h, \bar{\bs{U}}_h)$ be the solution of the corresponding discrete optimality system \eqref{DiscreteKKT_state}-\eqref{DiscreteKKT_VI}. Then under the hypothesis of Theorems \ref{NA_State} and \ref{NA_Adjoint}, there exists a positive constant $C$ independent of the mesh parameter $h$ such that the following estimates hold:
	\begin{align}
		\|\bar{\bs{U}} - \bar{\bs{U}}_h\|_{L^{\rparam}(\Omega)} &\leq C h^{1/2+\delta}\big(\|\bar{\bs{U}}\|_{1,\Omega} +  \norm{\bar{\bs{\varphi}}}_{3/2+\delta,\Omega} + \norm{\bar{\bs{\eta}}}_{3/2+\delta,\Omega}\big) \label{control_estimate}, \\
		\norm{\bar{\bs{u}} - \bar{\bs{u}}_h}_{1,\mathcal{T}_h^{nc}} + \norm{\bar{\bs{y}} - \bar{\bs{y}}_h}_{1,\mathcal{T}_h^{nc}} &\leq  C h^{1/2+\delta}\big(\|\bar{\bs{U}}\|_{1,\Omega} +  \norm{\bar{\bs{\varphi}}}_{3/2+\delta,\Omega} + \norm{\bar{\bs{\eta}}}_{3/2+\delta,\Omega} \nonumber \\ &\qquad\qquad+ \norm{\bar{\bs{u}}}_{3/2+\delta,\Omega} + \norm{\bar{\bs{y}}}_{3/2+\delta,\Omega}\big), \label{state_estimate}\\
		\norm{p - p_h}_{0,\Omega} &\leq C h^{1/2+\delta}\big(\|\bar{\bs{U}}\|_{1,\Omega} +  \norm{\bar{\bs{\varphi}}}_{3/2+\delta,\Omega} + \norm{\bar{\bs{\eta}}}_{3/2+\delta,\Omega} \nonumber \\ &\qquad\qquad+ \norm{\bar{\bs{u}}}_{3/2+\delta,\Omega} + \norm{\bar{\bs{y}}}_{3/2+\delta,\Omega} + \norm{p}_{1/2+\delta,\Omega}\big), \label{pressure_estimate}\\
		\norm{\bar{\bs{\varphi}} - \bar{\bs{\varphi}}_h}_{1,\mathcal{T}_h^{nc}} + \norm{\bar{\bs{\eta}} - \bar{\bs{\eta}}_h}_{1,\mathcal{T}_h^{nc}} &\leq C h^{1/2+\delta}\big(\|\bar{\bs{U}}\|_{1,\Omega} +  \norm{\bar{\bs{\varphi}}}_{3/2+\delta,\Omega} + \norm{\bar{\bs{\eta}}}_{3/2+\delta,\Omega}\big), \label{adjoint_estimate}\\
		\norm{\xi - \xi_h}_{0,\Omega} &\leq C h^{1/2+\delta}\big(\|\bar{\bs{U}}\|_{1,\Omega} +  \norm{\bar{\bs{\varphi}}}_{3/2+\delta,\Omega} + \norm{\bar{\bs{\eta}}}_{3/2+\delta,\Omega} + \norm{\xi}_{1/2+\delta,\Omega}\big). \label{Adjoint_pressure_estimate}
	\end{align}
\end{theorem}

	\begin{proof}
		 From \eqref{LUU} we know that
		\begin{align}\label{LUU_0}
			\bs{\mathcal{L}}_{\bs{U}\bs{U}} (\bar{\bs{u}}, \bar{\bs{y}}, \bar{\bs{U}}, \bar{\bs{\Lambda}}) \left[\bs{h}\right]^2 = \bs{\mathcal{L}}_{\bs{U} \bs{U}} (\bar{\bs{w}}, \bar{\bs{\Lambda}}) \left[\bs{h}\right]^2 = \lambda \norm{\bs{h}}^2_{0,\Omega}.
		\end{align}
		Now we perturb $\bs{\mathcal{L}}_{\bs{U}\bs{U}}(\bar{\bs{u}}, \bar{\bs{y}}, \bar{\bs{U}}, \bar{\bs{\Lambda}})$ only with respect to $\bar{\bs{U}},$ and since $(\bar{\bs{u}}, \bar{\bs{y}}, \bar{\bs{\Lambda}})$ depend on $\bar{\bs{U}},$ thus, by an abuse of notation we write, $\bs{\mathcal{L}}_{\bs{U}\bs{U}}(\bar{\bs{U}}).$   Choosing $\bs{h} = \bar{\bs{U}}_h - \bar{\bs{U}} \in \bs{L}^2(\Omega)$ in \eqref{LUU_0} and applying mean value Theorem for some $0 < \Theta < 1$ (see \cite[Theorem 7.9-1 (c)]{ciarlet2013linear}), yields
		\begin{align}
			(1 - \Theta)^2 \lambda \norm{\bs{h}}^2_{0,\Omega} = \bs{\mathcal{L}}_{\bs{U} \bs{U}}( \bar{\bs{U}}+\Theta \bs{h})\left[ \bs{h} \right]^2 &= \big(\bs{\mathcal{L}}_{{\bs{U}}} (\bar{\bs{U}}_h) - \bs{\mathcal{L}}_{\bs{U}} (\bar{\bs{U}})\big)\left[ \bs{h} \right], \nonumber \\
			&=  \left(\lambda \bar{\bs{U}}_h + {\bs{\varphi}}(\bar{\bs{U}}_h),  \bs{h}\right) - \left(\lambda \bar{\bs{U}} + \bar{\bs{\varphi}}, \bs{h}\right). \nonumber
		\end{align}
		Thus we obtain the following relation
		\begin{align}\label{Relation_1}
			(1 - \Theta)^2 \lambda \norm{ \bar{\bs{U}}_h - \bar{\bs{U}}}^2_{0,\Omega} = \left(\lambda \bar{\bs{U}} + \bar{\bs{\varphi}}, \bar{\bs{U}} - \bar{\bs{U}}_h\right) -  \left(\lambda \bar{\bs{U}}_h + {\bs{\varphi}}(\bar{\bs{U}}_h),  \bar{\bs{U}} - \bar{\bs{U}}_h\right).
		\end{align}
		The continuous variational inequality \eqref{vi} for $\bs{U} = \bar{\bs{U}}_h \in \bs{\mathcal{U}}_{h,ad} \subset \bs{\mathcal{U}}_{ad}$, and the discrete variational inequality \eqref{DiscreteKKT_VI} for $\bs{U}_h = \bs{\Pi}_0 \bar{\bs{U}}$ give the following relation:
		\begin{align}\label{Relation_2}
			\left(\lambda \bar{\bs{U}}_h + \bar{\bs{\varphi}}_h, \bs{\Pi}_0 \bar{\bs{U}} - \bar{\bs{U}}_h \right) \geq 0 \geq \left(\lambda \bar{\bs{U}} + \bar{\bs{\varphi}}, \bar{\bs{U}} - \bar{\bs{U}}_h\right).
		\end{align}
		Using \eqref{Relation_2} in \eqref{Relation_1} and simplifying results to
		\begin{align}
			(1 - \Theta)^2 \lambda \norm{\bar{\bs{U}}_h - \bar{\bs{U}}}^2_{0,\Omega} &\leq   \left(\lambda \bar{\bs{U}}_h + \bar{\bs{\varphi}}_h, \bs{\Pi}_0 \bar{\bs{U}} - \bar{\bs{U}}_h \right) - \left(\lambda \bar{\bs{U}}_h + {\bs{\varphi}}(\bar{\bs{U}}_h),  \bar{\bs{U}} - \bar{\bs{U}}_h\right)  \nonumber\\
			&= \left(\bar{\bs{\varphi}}_h - {\bs{\varphi}}(\bar{\bs{U}}_h), \bar{\bs{U}} - \bar{\bs{U}}_h\right) + \lambda \left(\bar{\bs{U}}_h, \bs{\Pi}_0 \bar{\bs{U}} -  \bar{\bs{U}}\right) + \left(\bar{\bs{\varphi}}_h, \bs{\Pi}_0 \bar{\bs{U}} - \bar{\bs{U}}\right). \label{Relation_3}
		\end{align}
		The second term on the right hand side of \eqref{Relation_3} vanishes due to \eqref{Pi0_defn} and the third term can be rewritten using $\bs{\Pi}_0$ as follows:
		\begin{align}
			\left(\bar{\bs{\varphi}}_h, \bs{\Pi}_0 \bar{\bs{U}} - \bar{\bs{U}}\right) = \left(\bar{\bs{\varphi}}_h - \bs{\Pi}_0 \bar{\bs{\varphi}}_h, \bs{\Pi}_0 \bar{\bs{U}} - \bar{\bs{U}}\right) \leq \norm{\bar{\bs{\varphi}}_h - \bs{\Pi}_0 \bar{\bs{\varphi}}_h}_{0,\Omega} \norm{\bs{\Pi}_0 \bar{\bs{U}} - \bar{\bs{U}}}_{0,\Omega}. \label{Relation_4}
		\end{align}
		The first term on the right hand side of \eqref{Relation_3} is bounded using discrete Sobolev embedding \cite[(4.3)]{TKM_ncDDF} and Young's inequality as follows:
		\begin{align}
			\left(\bar{\bs{\varphi}}_h - \bar{\bs{\varphi}}(\bar{\bs{U}}_h), \bar{\bs{U}} - \bar{\bs{U}}_h\right) &\leq \norm{\bar{\bs{\varphi}}_h - \bar{\bs{\varphi}}(\bar{\bs{U}}_h)}_{0,\Omega} \norm{\bar{\bs{U}} - \bar{\bs{U}}_h}_{0,\Omega} \nonumber \\
			&\leq   (1/2 \lambda (1-\Theta)^2) \norm{\bar{\bs{\varphi}}_h - {\bs{\varphi}}(\bar{\bs{U}}_h)}_{0,\Omega}^2 +  ((1-\Theta)^2 \lambda/2) \norm{\bar{\bs{U}} - \bar{\bs{U}}_h}_{0,\Omega}^2 . \label{Relation_5}
		\end{align}
		Plugging \eqref{Relation_4} and \eqref{Relation_5} in \eqref{Relation_3} and applying discrete Sobolev embedding \cite[(4.3)]{TKM_ncDDF}, \eqref{Pi0_approximation}, \eqref{aPriori_varphieta}  gives
		\begin{align}
			((1 - \Theta)^2 \lambda/2) \norm{\bar{\bs{U}}_h - \bar{\bs{U}}}^2_{0,\Omega} &\leq \norm{\bar{\bs{\varphi}}_h - \bs{\Pi}_0 \bar{\bs{\varphi}}_h}_{0,\Omega} \norm{\bs{\Pi}_0 \bar{\bs{U}} - \bar{\bs{U}}}_{0,\Omega}  + (1/2 \lambda (1-\Theta)^2) \norm{\bar{\bs{\varphi}}_h - \bar{\bs{\varphi}}(\bar{\bs{U}}_h)}_{0,\Omega}^2\nonumber\\
			&\leq C_{\Pi} h \norm{\bar{\bs{\varphi}}_h}_{1,\mathcal{T}_h^{nc}} \norm{\bs{\Pi}_0 \bar{\bs{U}} - \bar{\bs{U}}}_{0,\Omega} + ({C_{e}^{nc}}^2/2 \lambda (1-\Theta)^2) \norm{\bar{\bs{\varphi}}_h - \bar{\bs{\varphi}}(\bar{\bs{U}}_h)}_{1,\mathcal{T}_h^{nc}}^2 \nonumber\\
			&\leq \tilde{C} \left( h \norm{\bar{\bs{\varphi}}_h}_{1,\mathcal{T}_h^{nc}} \norm{\bs{\Pi}_0 \bar{\bs{U}} - \bar{\bs{U}}}_{0,\Omega} + h^{1+2\delta} \big(\norm{\bar{\bs{\varphi}}}_{3/2+\delta,\Omega}^2 + \norm{\bar{\bs{\eta}}}_{3/2+\delta,\Omega}^2\big) \right). \nonumber
		\end{align}
		where $\tilde{C} = \max{\left\{C_{\Pi}, ({C_{e}^{nc}}^2/2 \lambda (1-\Theta)^2)\right\}}$ is a positive constant. Inferring from \eqref{energyEst_varphi} the following
		$$\norm{\bar{\bs{\varphi}}_h}_{1,\mathcal{T}_h^{nc}} \leq C_{\bs{\varphi}} \big(\norm{\bar{\bs{u}}_h}_{0,\Omega} + \norm{\bs{u}_d}_{0,\Omega}\big) < C,$$ and employing \eqref{Pi0_approximation}, we reach the following estimate:
		\begin{align}\label{L2_estimate}
			\norm{\bar{\bs{U}}_h - \bar{\bs{U}}}_{0,\Omega} \leq C h^{1/2+\delta}\big(\|\bar{\bs{U}}\|_{1,\Omega} +  \norm{\bar{\bs{\varphi}}}_{3/2+\delta,\Omega} + \norm{\bar{\bs{\eta}}}_{3/2+\delta,\Omega}\big).
		\end{align}
		Now we show that using \eqref{SSC} and \eqref{L2_estimate}, we can achieve the same convergence rate for the control variable in the neighbourhood dictated by \eqref{SSC} which states, there exists $\varepsilon > 0$ and $\sigma > 0$ such that
		$$ \bs{\mathcal{L}}_{\bs{u} \bs{u}} (\bar{\bs{w}}, \bar{\bs{\Lambda}}) \left[\bs{\zeta}\right]^2 + \bs{\mathcal{L}}_{\bs{y} \bs{y}} (\bar{\bs{w}}, \bar{\bs{\Lambda}}) \left[\bs{\mu}\right]^2 + 
		\bs{\mathcal{L}}_{\bs{U} \bs{U}} (\bar{\bs{w}}, \bar{\bs{\Lambda}}) \left[\bs{h}\right]^2 = \bs{\mathcal{L}}_{\bs{w} \bs{w}} (\bar{\bs{w}}, \bar{\bs{\Lambda}}) \left[\left(\bs{\zeta}, \bs{\mu}, \bs{h}\right)\right]^2 \geq \sigma \norm{\bs{h}}^2_{L^{r}(\Omega)},$$
		with $\bs{h} = \bs{U} - \bar{\bs{U}} \;\;\forall\;\; \bs{U} \in \bs{\mathcal{U}}_{ad}, \bs{h}_i = 0 \;\; \mbox{on} \;\; \Omega_{\varepsilon,i} \; \mbox{for} \;\; i = 1,\ldots , n$. Observing from \eqref{LUU} that
		$$	\bs{\mathcal{L}}_{\bs{U}\bs{U}} (\bar{\bs{u}}, \bar{\bs{y}}, \bar{\bs{U}}, \bar{\bs{\Lambda}}) \left[\bs{h}\right]^2 = \bs{\mathcal{L}}_{\bs{U} \bs{U}} (\bar{\bs{w}}, \bar{\bs{\Lambda}}) \left[\bs{h}\right]^2 = \lambda \norm{\bs{h}}^2_{0,\Omega} \geq \lambda C_{e} \norm{\bs{h}}^2_{L^{\rparam}(\Omega)},$$
		we conclude that
		\begin{align}\label{SSC_U}
			\bs{\mathcal{L}}_{\bs{U}\bs{U}} (\bar{\bs{u}}, \bar{\bs{y}}, \bar{\bs{U}}, \bar{\bs{\Lambda}}) \left[\bs{h}\right]^2 \geq \tilde{\sigma} \norm{\bs{h}}^2_{L^{\rparam}(\Omega)}, \; \mbox{where} \; \tilde{\sigma} = \min{\left\{\sigma, \lambda C_{e}\right\}} \; \mbox{is a positive constant.}
		\end{align}
		Following \eqref{Relation_1} we can obtain
	 	\begin{align*}
			&\big(\bs{\mathcal{L}}_{\bs{U}} (\bar{\bs{U}}_h) - \bs{\mathcal{L}}_{\bs{U}} (\bar{\bs{U}})\big)\left[ \bs{h} \right] = \bs{\mathcal{L}}_{\bs{U} \bs{U}}( \bar{\bs{U}}+\Theta \bs{h})\left[ \bs{h} \right]^2 \nonumber \\
			&\geq \big(\bs{\mathcal{L}}_{\bs{U} \bs{U}}(\bar{\bs{U}} + \Theta \bs{h}) - \bs{\mathcal{L}}_{\bs{U} \bs{U}}(\bar{\bs{U}})\big) \left[ \bs{h} \right]^2 + \tilde{\sigma} \norm{\bs{h}}^2_{L^{\rparam}(\Omega)} = \norm{\bs{h}}^2_{L^2(\Omega)} \lambda (\Theta^2 - 2 \Theta) + \tilde{\sigma} \norm{\bs{h}}_{L^{\rparam}(\Omega)}^2, 
		\end{align*} 
		where in the last inequality we used \eqref{SSC_U} after adding and subtracting $\bs{\mathcal{L}}_{\bs{U} \bs{U}}(\bar{\bs{U}})\left[ \bs{h} \right]^2$. This leads to
		\begin{align}
			\tilde{\sigma} \norm{\bs{h}}_{L^{\rparam}(\Omega)}^2 \leq \big(\bs{\mathcal{L}}_{\bs{U}} (\bar{\bs{U}}_h) - \bs{\mathcal{L}}_{\bs{U}} (\bar{\bs{U}})\big)\left[ \bs{h} \right] + \lambda (2\Theta - \Theta^2) \norm{\bs{h}}^2_{0,\Omega}.
		\end{align}
		The first term on the right hand side is treated analogous to \eqref{Relation_1}  and then applying \eqref{L2_estimate} leads to control estimate \eqref{control_estimate}. 
		The estimates in \eqref{state_estimate} and \eqref{adjoint_estimate} can be readily derived using the stability results in Remark \ref{rem:state_uniqueness} and \eqref{adjoint_stability} and the a priori estimates \eqref{aPriori_uy}, \eqref{aPriori_varphieta} and \eqref{control_estimate} after decomposing them as follows:
		\begin{align*}
			\bar{\bs{u}} - \bar{\bs{u}}_h = \bar{\bs{u}} - {\bs{u}}_h(\bs{\bar{U}}) + {\bs{u}}_h(\bs{\bar{U}}) - \bar{\bs{u}}_h \; \mbox{and} \; \bar{\bs{y}} - \bar{\bs{y}}_h = \bar{\bs{y}} - {\bs{y}}_h(\bs{\bar{U}}) + {\bs{y}}_h(\bs{\bar{U}}) - \bar{\bs{y}}_h, \\
		 			\bar{\bs{\varphi}} - \bar{\bs{\varphi}}_h = \bar{\bs{\varphi}} - {\bs{\varphi}}_h(\bs{\bar{U}}) + {\bs{\varphi}}_h(\bs{\bar{U}}) - \bar{\bs{\varphi}}_h \; \mbox{and} \; \bar{\bs{\eta}} - \bar{\bs{\eta}}_h = \bar{\bs{\eta}} - {\bs{\eta}}_h(\bs{\bar{U}}) + {\bs{\eta}}_h(\bs{\bar{U}}) - \bar{\bs{\eta}}_h.
		\end{align*}
		To derive the pressure estimates we decompose them as follows:
		\begin{align}\label{decompose_pressure}
			\begin{aligned}
				&\norm{p - p_h}_{0,\Omega} \leq \norm{p(\bs{U}) - p_h(\bs{U})}_{0,\Omega} + \norm{p_h(\bs{U}) - p_h}_{0,\Omega}, \\
				&\norm{\xi - \xi_h}_{0,\Omega} \leq \norm{\xi(\bs{U}) - \xi_h(\bs{U})}_{0,\Omega} + \norm{\xi_h(\bs{U}) - \xi_h}_{0,\Omega}.
			\end{aligned}
		\end{align}
		Now the estimates \eqref{pressure_estimate} and \eqref{Adjoint_pressure_estimate} can be derived using \eqref{aPriori_p} and \eqref{aPriori_zeta}, respectively on the first term(s) on the right-hand side of \eqref{decompose_pressure}, and the remaining terms can be estimated similar to \eqref{aPriori_eq9:Ad}, respectively.
	\end{proof}
	
\subsection{Implicit control discretization} It is well known that the variational inequality associated with a distributed optimal control problem  allows us to write the optimal control $\bar{\bs{U}}$ in terms of a projection formula \eqref{projection_formula}. Then following the \textit{optimize-then-discretize} approach, this allows us to characterize the discrete optimal control $\bar{\bs{U}}_h$ as $\bar{\bs{U}}_h = \mathcal{P}_{\bs{\mathcal{U}}_{ad}} \left(- \frac{1}{\lambda} \bar{\bs{\varphi}}_{h}\right)$, where $\mathcal{P}_{\bs{\mathcal{U}}_{ad}}$ is the projection onto $\bs{\mathcal{U}}_{ad}$, and is understood in a component-wise sense. Notice that the right hand side in the above equality is a discrete object but lives in the space ${\bs{\mathcal{U}}_{ad}}$. This characterization is the motivation behind variational/implicit discretization of control and leads to the discrete optimality system  \eqref{DiscreteKKT_state}-\eqref{DiscreteKKT_VI} with the choice  $\bs{\mathcal{U}}_h = \bs{\mathcal{U}}$ and  $\bs{\mathcal{U}}_{h,ad} = \bs{\mathcal{U}}_{ad}$. In our current setting the implicitly discretized control $\bar{\bs{U}}_h$ is a projection of the non-conforming Crouzeix-Raviart finite element function onto the admissible set of controls, and is in fact computable. This is the key feature which distinguishes this approach from the explicit approach. The later hinges on first discretizing the control space and then the state and adjoint spaces. In particular, for the piecewise constant control case that we choose, the discrete optimal control control has the following form,
$ \bar{\bs{U}}_h = \mathcal{P}_{\bs{\mathcal{U}}_{h,ad}} \left(- \frac{1}{\lambda} \bar{\bs{\varphi}}_{h}\right) =  \mathcal{P}_{\bs{\mathcal{U}}_{ad}} \left(-\frac{1}{\lambda} \bs{\Pi}_0 \bar{\bs{\varphi}}_{h}\right).$

\subsubsection{A priori error analysis} Under implicit discretization of control, the existence of an appropriate solution $\bar{\bs{U}}_h$ in any neighborhood of $\bar{\bs{U}}$, follows similarly to the explicit case. Furthermore, following steps analogous to the proof of Theorem \ref{apriori_ctrl_problem}, we obtain
\begin{align}\label{Relation_*}
	(1 - \Theta)^2 \lambda \norm{ \bar{\bs{U}}_h - \bar{\bs{U}}}^2_{0,\Omega} = \left(\lambda \bar{\bs{U}} + \bar{\bs{\varphi}}, \bar{\bs{U}} - \bar{\bs{U}}_h\right) -  \left(\lambda \bar{\bs{U}}_h + {\bs{\varphi}}(\bar{\bs{U}}_h),  \bar{\bs{U}} - \bar{\bs{U}}_h\right).
\end{align}
Now, using the fact that under implicit discretization of control the continuous variational inequality \eqref{vi} for $\bs{U} = \bar{\bs{U}}_h \in  \bs{\mathcal{U}}_{ad}$, and the discrete variational inequality \eqref{DiscreteKKT_VI} for $\bs{U}_h = \bar{\bs{U}}$ yield the following relation:
\begin{align}\label{Relation_imp}
	\left(\lambda \bar{\bs{U}}_h + \bar{\bs{\varphi}}_h, \bar{\bs{U}} - \bar{\bs{U}}_h \right) \geq 0 \geq \left(\lambda \bar{\bs{U}} + \bar{\bs{\varphi}}, \bar{\bs{U}} - \bar{\bs{U}}_h\right).
\end{align}
Substituting \eqref{Relation_imp} into \eqref{Relation_*} results to,
	\begin{align}\label{variational_rem}
		(1 - \Theta)^2 \lambda \norm{\bar{\bs{U}}_h - \bar{\bs{U}}}^2_{0,\Omega} &\leq   \left(\lambda \bar{\bs{U}}_h + \bar{\bs{\varphi}}_h, \bar{\bs{U}} - \bar{\bs{U}}_h \right) - \left(\lambda \bar{\bs{U}}_h + {\bs{\varphi}}(\bar{\bs{U}}_h),  \bar{\bs{U}} - \bar{\bs{U}}_h\right) \nonumber\\
		&= \left(\bar{\bs{\varphi}}_h - {\bs{\varphi}}(\bar{\bs{U}}_h), \bar{\bs{U}} - \bar{\bs{U}}_h\right) \leq \norm{ \bar{\bs{\varphi}}_h - {\bs{\varphi}}(\bar{\bs{U}}_h) }_{0,\Omega} \norm{ \bar{\bs{U}} - \bar{\bs{U}}_h }_{0,\Omega}.
	\end{align} 		
Now an application of discrete Sobolev embedding followed by \eqref{aPriori_varphieta} leads to the control estimate \eqref{control_estimate} in the implicit case as well. The rest of the estimates follow in the same way as for the explicit case. We also note that, it is well known in the literature that the implicit discretization approach leads to {\it a priori} error estimates of $\mathcal{O}(h^2)$ for the control variable.  We observe this behavior in our numerical experiments (see Table \ref{table2}) and in theory are only limited by the fact that we do not have the best approximation for $\norm{ \bar{\bs{\varphi}}_h - {\bs{\varphi}}(\bar{\bs{U}}_h) }_{0,\Omega} \sim \mathcal O(h^{3/2 + \delta})$ in \eqref{variational_rem}.

\section{Numerical Experiments}\label{NE} 
In this section we provide a set of examples which verify the convergence rates proved in Theorem \ref{control_estimate}, and present an application oriented example. A primal dual active set strategy as a \textit{semi-smooth Newton (SSN)} method is employed to determine the optimal control from the discrete optimality system  \eqref{DiscreteKKT_state}-\eqref{DiscreteKKT_VI} under both explicit and implicit discretizations of control (see for more details \cite{Hintermuller_SSN, JCD_Book, Hinze, Hinze_SSN}). The resulting linearized system is solved using the linear solver \textit{MUMPS} \cite{AmestoyDuffKosterLExcellent2001}  in open source finite element library FEniCS \cite{alnaes2015fenics}. The essential boundary conditions are implemented using the Dirichlet boundary class feature of FEniCS and the zero mean value for the pressure approximation is tackled using a Lagrange multiplier approach.


\subsection{Accuracy test} Let the computational domain be $\Omega := [0,1] \times [0,1]$ and consider a sequence of uniformly refined meshes on $\Omega$ with mesh parameter $h$.  We take the buoyancy term  of the form $\boldsymbol{F}(\boldsymbol{y}) = (T + N_{r} S) \boldsymbol{g},$ where $N_r$ is the solutal to the thermal buoyancy ratio, and the viscosity is chosen to be of exponential form, that is, $\nu(T) = \nu_2 \exp(-T)$. The other parameters are defined as follows: $\boldsymbol{g} = (0,1)^{\top}, \boldsymbol{K}^{-1} = \sigma \boldsymbol{I}, \boldsymbol{D} = 1000 \boldsymbol{I}.$ We use the following smooth functions as the closed form solutions of \eqref{P:CF}-\eqref{P:GE} in our accuracy test:
\begin{align*}
	&\boldsymbol{u}(x,y) = \left(\sin(\pi x) \cos(\pi y), - \cos(\pi x) \sin(\pi y)\right)^{\top}, \; p(x,y) = \cos(\pi x) \exp(y), \; T(x,y) = 0.5 + 0.5 \cos(xy), \\
	&S(x,y) = 0.1+0.3\exp(xy), \; \boldsymbol{\varphi}(x,y) = \left(\sin(\pi x)^2 \sin(\pi y) \cos(\pi y), - \sin(\pi y)^2 \sin(\pi x) \cos(\pi x)\right), \\
	&\zeta(x,y) =  \cos(\pi y) \exp(x), \; \eta^T(x,y) = 0.5 \cos(x y) x (x-1) y (y-1), \;   \eta^S(x,y) = 0.5 \exp(x y) x (x-1) y (y-1), \\
	& U_{j}(x,y) = \min\left(0.25, \max\left(-0.1, -\frac{1}{\lambda} \varphi_{j}(x,y)\right)\right), \; \lambda = 1.
\end{align*}
The boundary conditions, additional external forces and source terms are prescribed according to the above prescribed manufactured solutions. The errors for the state variables and control in their natural norms and the corresponding convergence rates are denoted as follows:
\begin{align*}
	&\boldsymbol{e}_{\boldsymbol{u}} = \norm{\boldsymbol{u} - \boldsymbol{u}_h}_{1,\mathcal{T}_h^{nc}},\; \boldsymbol{e}_{p} = \norm{p - p_h}_{0,\Omega},\; \boldsymbol{e}_{T} = \norm{T-T_h}_{1,\mathcal{T}_h^{nc}}, \boldsymbol{e}_{S} = \norm{S - S_h}_{1,\mathcal{T}_h^{nc}}, \\
	& \boldsymbol{e}_{\boldsymbol{U}_{1}} = \|U_{1} - U_{h_1}\|_{L^{\rparam}(\Omega)},\; \boldsymbol{e}_{\boldsymbol{U}_{2}} = \|U_{2} - U_{h_{2}}\|_{L^{\rparam}(\Omega)}, \;\boldsymbol{rate} = \log(\boldsymbol{e}_{(\cdot)} / \tilde{\boldsymbol{e}}_{(\cdot)}) (\log(h / \tilde{h}))^{-1},
\end{align*}
where $\boldsymbol{e}, \tilde{\boldsymbol{e}}$ denote the errors generated on two consecutive mesh refinements of mesh size $h$ and $\tilde{h},$ respectively. The error in control is measured in the $x$ and $y$ components separately and the errors in adjoint variables are denoted by $\boldsymbol{e}_{\boldsymbol{\varphi}}, \boldsymbol{e}_{\zeta}, \boldsymbol{e}_{\eta^T}, \boldsymbol{e}_{\eta^S},$ which are defined analogously to the state variable errors. The errors and convergence rates  focusing on cases where the viscosity and permeability coefficients scale differently, changing from Flow to Darcy regimes  under explicit and implicit control discretizatrions are reported in Table \ref{table1} and Table \ref{table2}; where $\boldsymbol{It}$ denotes the number of iterations the \textit{semi-smooth Newton} algorithm takes to reach the absolute tolerance of $10^{-6}$ and $\boldsymbol{\chi}_{(\cdot)}$ denotes the degrees of freedom associated with the space $(\cdot)$. The tabulated errors confirm the theoretical bounds of Theorem \ref{control_estimate} and that the discrete velocities are divergence free. 
{\footnotesize
	\begin{longtable}{@{}lllllllllllll@{}} \hline
		\multicolumn{13}{c}{\textbf{Flow regime} ($\nu_2 = 1, \; \sigma = 1$)}  \\ \hline
		\multicolumn{13}{l}{$\boldsymbol{\chi}_{\boldsymbol{V}_h^{cr}}$ \hspace{0.6cm} $\boldsymbol{e}_{\boldsymbol{u}}$  \hspace{0.6cm} $\boldsymbol{rate}$ \hspace{0.2cm} $\norm{\bdiv \boldsymbol{u}_h}_{\infty}$ \hspace{0.35cm} $\boldsymbol{e}_{\boldsymbol{\varphi}}$ \hspace{0.5cm} $\boldsymbol{rate}$ \hspace{0.2cm} $\norm{\bdiv \boldsymbol{\varphi}_h}_{\infty}$ \hspace{0.15cm} $\boldsymbol{\chi}_{Q_h^0}$ \hspace{0.65cm} $\boldsymbol{e}_p$ \hspace{0.4cm} $\boldsymbol{rate}$ \hspace{0.6cm} $\boldsymbol{e}_{\xi}$ \hspace{0.5cm} $\boldsymbol{rate}$ \hspace{0.2cm} $\boldsymbol{It}$} \\ \hline 
		\hspace{0.3cm}66 		& 2.7363  & \hspace{0.2cm}\;\;- & 2.6E-15 & \hspace{0.2cm}2.2562   & \;\;- &\hspace{-0cm}2.4E-15             &   \hspace{0.25cm}18 &  0.3532  & \;\;- & \hspace{-0cm}0.4113 & \hspace{-0cm}\;\;-& \hspace{-0.2cm}6\\
		\hspace{0.3cm}240      & 1.7111 & \hspace{0.15cm}0.68   & 4.4E-15 & \hspace{0.2cm}1.0684 & 1.08    & \hspace{-0cm}1.3E-15        &    \hspace{0.25cm}72 & 0.1845 & 0.94 & \hspace{-0cm}0.2175 & \hspace{-0cm}0.92 & \hspace{-0.2cm}6	\\
		\hspace{0.3cm}912       & 0.9578 & \hspace{0.15cm}0.84   & 3.5E-15 & \hspace{0.2cm}0.5089 & 1.07    & \hspace{-0cm}1.5E-15      &   \hspace{0.25cm}288 & 0.0892 & 1.05 & \hspace{-0cm}0.0984 & \hspace{-0cm}1.14	& \hspace{-0.2cm}6	\\
		\hspace{0.3cm}3552    & 0.5048 & \hspace{0.15cm}0.92   & 5.9E-15 & \hspace{0.2cm}0.2452 & 1.05    & \hspace{-0cm}2.9E-15      &   \hspace{0.25cm}1152 & 0.0431 & 1.05 & \hspace{-0cm}0.0462 & \hspace{-0cm}1.09	& \hspace{-0.2cm}5	\\
		\hspace{0.3cm}14016   & 0.2584 & \hspace{0.15cm}0.97     & 1.3E-14 & \hspace{0.2cm}0.1200 & 1.03       & \hspace{-0cm}7.9E-15  &   \hspace{0.25cm}4608 & 0.0212 & 1.02 & \hspace{-0cm}0.0228 & \hspace{-0cm}1.02	 & \hspace{-0.2cm}5	\\ 
		\hspace{0.3cm}55680  & 0.1307 & \hspace{0.15cm}0.98     & 2.9E-14 & \hspace{0.2cm}0.0593 & 1.02       & \hspace{-0cm}1.8E-14  &   \hspace{0.25cm}18432 & 0.0105 & 1.01 & \hspace{-0cm}0.0117&	\hspace{-0cm}1.0	& \hspace{-0.2cm}5\\ 
		\hline
		\multicolumn{13}{l}{\hspace{0.25cm}$\boldsymbol{e}_T$  \hspace{0.6cm} $\boldsymbol{rate}$ \hspace{0.5cm} $\boldsymbol{e}_S$ \hspace{0.6cm} $\boldsymbol{rate}$ \hspace{0.9cm} $\boldsymbol{e}_{\eta^T}$ \hspace{0.4cm} $\boldsymbol{rate}$ \hspace{0.5cm} $\boldsymbol{e}_{\eta^S}$ \hspace{0.6cm} $\boldsymbol{rate}$ \hspace{0.4cm} $\boldsymbol{\chi}_{\boldsymbol{\mathcal{U}}_{h,ad}}$ \hspace{0.3cm} $\boldsymbol{e}_{\boldsymbol{U}_{\boldsymbol{1}}}$ \hspace{0.35cm} $\boldsymbol{rate}$ \hspace{0.45cm} $\boldsymbol{e}_{\boldsymbol{U}_{\boldsymbol{2}}}$ \hspace{0.3cm} $\boldsymbol{rate}$}  \\ \hline 
		\hspace{0.3cm}1.6757 & \hspace{0.1cm}\;\;-   & 2.8265 &\hspace{0.2cm}\;\;-   & \hspace{0.2cm}0.9276 & \hspace{0.1cm}\;\;-   &\hspace{0.1cm}1.3509 &\hspace{0.2cm}\;\;-	&18 &  0.0785  &\hspace{0.2cm}\;\;- &   0.0826 & \;\;- \\
		\hspace{0.3cm}0.8509 & \hspace{0.1cm}0.98 & 1.4152 & \hspace{0.2cm}1.0 & \hspace{0.2cm}0.4754 &  \hspace{0.1cm}0.96 & \hspace{0.1cm}0.7101 & \hspace{0.2cm}0.93 &72 &   0.0321 & \hspace{0.2cm}1.29 &  0.308 &	1.42   \\
		\hspace{0.3cm}0.4281 & \hspace{0.1cm}0.99 & 0.7080 &\hspace{0.2cm}1.0 & \hspace{0.2cm}0.2392 &  \hspace{0.1cm}0.99  &	\hspace{0.1cm}0.3597 &\hspace{0.2cm}0.98&288 & 0.0151 & \hspace{0.2cm}1.08 & 0.0148 &	1.06\\
		\hspace{0.3cm}0.2144 & \hspace{0.1cm}1.0 & 0.2144 & \hspace{0.2cm}1.0 &\hspace{0.2cm}0.1197 &  \hspace{0.1cm}1.0 	&  \hspace{0.1cm}0.1804 & \hspace{0.2cm}1.0&1152 & 0.0075 &\hspace{0.2cm}1.01 &  0.0074 &0.99\\
	    \hspace{0.3cm}0.1073 & \hspace{0.1cm}1.0 & 0.1770 & \hspace{0.2cm}1.0 & \hspace{0.2cm}0.0599 & \hspace{0.1cm}1.0  &\hspace{0.1cm}0.0903   & \hspace{0.2cm}1.0&4608 & 0.0037 &\hspace{0.2cm}1.02 & 0.0037 &	1.01\\ 
		\hspace{0.3cm}0.0536 & \hspace{0.1cm}1.0 & 0.0885 & \hspace{0.2cm}1.0 & \hspace{0.2cm}0.0299 & \hspace{0.1cm}1.0  &	\hspace{0.1cm}0.0452 &\hspace{0.2cm}1.0&18432 & 0.0018 &\hspace{0.2cm}1.0 &0.0018 & 1.0\\ 
		\hline
		\multicolumn{13}{c}{\textbf{Stokes regime} ($\nu_2 = 10, \; \sigma = 0$)}  \\ \hline
		\multicolumn{13}{l}{$\boldsymbol{\chi}_{\boldsymbol{V}_h^{cr}}$ \hspace{0.6cm} $\boldsymbol{e}_{\boldsymbol{u}}$  \hspace{0.6cm} $\boldsymbol{rate}$ \hspace{0.2cm} $\norm{\bdiv \boldsymbol{u}_h}_{\infty}$ \hspace{0.35cm} $\boldsymbol{e}_{\boldsymbol{\varphi}}$ \hspace{0.5cm} $\boldsymbol{rate}$ \hspace{0.2cm} $\norm{\bdiv \boldsymbol{\varphi}_h}_{\infty}$ \hspace{0.15cm} $\boldsymbol{\chi}_{Q_h^0}$ \hspace{0.65cm} $\boldsymbol{e}_p$ \hspace{0.4cm} $\boldsymbol{rate}$ \hspace{0.6cm} $\boldsymbol{e}_{\xi}$ \hspace{0.5cm} $\boldsymbol{rate}$ \hspace{0.2cm} $\boldsymbol{It}$} \\ \hline 
		\hspace{0.3cm}66 		& 3.6250  & \hspace{0.2cm}\;\;- & 1.33E-15 & \hspace{0.2cm}4.5523   & \;\;- &\hspace{-0cm}4.7E-15             &   \hspace{0.25cm}18 &  0.8079  & \;\;- & \hspace{-0cm}1.2520   & \hspace{-0cm}\;\;-& \hspace{-0.2cm}5\\
		\hspace{0.3cm}240      & 1.8689 & \hspace{0.15cm}0.96   & 4.4E-15 & \hspace{0.2cm}2.4415 & 0.9    & \hspace{-0cm}1.3E-15        &    \hspace{0.25cm}72 & 0.4301 & 0.91 & \hspace{-0cm}0.7032     & \hspace{-0cm}0.83 & \hspace{-0.2cm}5	\\
		\hspace{0.3cm}912       & 0.9427 & \hspace{0.15cm}0.99   & 2.9E-15 & \hspace{0.2cm}1.2424 & 0.97    & \hspace{-0cm}1.7E-15      &   \hspace{0.25cm}288 & 0.2159 & 0.99 & \hspace{-0cm}0.3409 & \hspace{-0cm}1.04	& \hspace{-0.2cm}5	\\
		\hspace{0.3cm}3552    & 0.4726 & \hspace{0.15cm}1.0   & 5.7E-15 & \hspace{0.2cm}0.6234 & 0.99    & \hspace{-0cm}3.3E-15      &   \hspace{0.25cm}1152 & 0.1077 & 1.0 & \hspace{-0cm}0.1670    & \hspace{-0cm}1.03	& \hspace{-0.2cm}5	\\
		\hspace{0.3cm}14016   & 0.2365 & \hspace{0.15cm}1.0     & 1.3E-14 & \hspace{0.2cm}0.3118 & 1.0       & \hspace{-0cm}7.9E-15  &   \hspace{0.25cm}4608  & 0.0538 & 1.0 & \hspace{-0cm}0.0829   & \hspace{-0cm}1.01	 & \hspace{-0.2cm}5	\\ 
		\hspace{0.3cm}55680  & 0.1182 & \hspace{0.15cm}1.0     & 2.7E-14 & \hspace{0.2cm}0.1559 & 1.0       & \hspace{-0cm}1.8E-14  &   \hspace{0.25cm}18432 & 0.0269 & 1.0 & \hspace{-0cm}0.0414    &	\hspace{-0cm}1.0	& \hspace{-0.2cm}5\\ 
		\hline
		\multicolumn{13}{l}{\hspace{0.25cm}$\boldsymbol{e}_T$  \hspace{0.6cm} $\boldsymbol{rate}$ \hspace{0.5cm} $\boldsymbol{e}_S$ \hspace{0.6cm} $\boldsymbol{rate}$ \hspace{0.9cm} $\boldsymbol{e}_{\eta^T}$ \hspace{0.4cm} $\boldsymbol{rate}$ \hspace{0.5cm} $\boldsymbol{e}_{\eta^S}$ \hspace{0.6cm} $\boldsymbol{rate}$ \hspace{0.4cm} $\boldsymbol{\chi}_{\boldsymbol{\mathcal{U}}_{h,ad}}$ \hspace{0.3cm} $\boldsymbol{e}_{\boldsymbol{U}_{\boldsymbol{1}}}$ \hspace{0.35cm} $\boldsymbol{rate}$ \hspace{0.45cm} $\boldsymbol{e}_{\boldsymbol{U}_{\boldsymbol{2}}}$ \hspace{0.3cm} $\boldsymbol{rate}$}  \\ \hline 
		\hspace{0.3cm}1.6757 & \hspace{0.1cm}\;\;-   & 2.8265 &\hspace{0.2cm}\;\;-   & \hspace{0.2cm}0.9276 & \hspace{0.1cm}\;\;-   &\hspace{0.1cm}1.3509 &\hspace{0.2cm}\;\;-	&18 &  0.0597   &\hspace{0.2cm}\;\;-     &   0.0595 & \;\;- \\
		\hspace{0.3cm}0.8509 & \hspace{0.1cm}0.98 & 1.4152 & \hspace{0.2cm}1.0 & \hspace{0.2cm}0.4754 &  \hspace{0.1cm}0.96 & \hspace{0.1cm}0.7101 & \hspace{0.2cm}0.93 &72 &   0.0279 & \hspace{0.2cm}1.09 &  0.0279  & 1.09    \\
		\hspace{0.3cm}0.4281 & \hspace{0.1cm}0.99 & 0.7080 &\hspace{0.2cm}1.0 & \hspace{0.2cm}0.2392 &  \hspace{0.1cm}0.99  &	\hspace{0.1cm}0.3597 &\hspace{0.2cm}0.98&288 & 0.0145 & \hspace{0.2cm}0.95 & 0.0145   &	0.95\\
		\hspace{0.3cm}0.2144 & \hspace{0.1cm}1.0 & 0.3541 & \hspace{0.2cm}1.0 &\hspace{0.2cm}0.1197 &  \hspace{0.1cm}1.0 	&  \hspace{0.1cm}0.1804 & \hspace{0.2cm}1.0&1152 & 0.0074 &\hspace{0.2cm}0.97          &  0.0074  & 0.97\\
		\hspace{0.3cm}0.1073 & \hspace{0.1cm}1.0 & 0.1770 & \hspace{0.2cm}1.0 & \hspace{0.2cm}0.0599 & \hspace{0.1cm}1.0  &\hspace{0.1cm}0.0903   & \hspace{0.2cm}1.0&4608 & 0.0037 &\hspace{0.2cm}1.0          & 0.0037   & 1.0\\ 
		\hspace{0.3cm}0.0536 & \hspace{0.1cm}1.0 & 0.0885 & \hspace{0.2cm}1.0 & \hspace{0.2cm}0.0299 & \hspace{0.1cm}1.0  &	\hspace{0.1cm}0.0452 &\hspace{0.2cm}1.0&18432 & 0.0018 &\hspace{0.2cm}1.0         &0.0018    & 1.0\\ 
		\hline
		\multicolumn{13}{c}{\textbf{Darcy regime} ($\nu_2 = 1, \; \sigma = 10000$)}  \\ \hline
		\multicolumn{13}{l}{$\boldsymbol{\chi}_{\boldsymbol{V}_h^{cr}}$ \hspace{0.6cm} $\boldsymbol{e}_{\boldsymbol{u}}$  \hspace{0.6cm} $\boldsymbol{rate}$ \hspace{0.2cm} $\norm{\bdiv \boldsymbol{u}_h}_{\infty}$ \hspace{0.35cm} $\boldsymbol{e}_{\boldsymbol{\varphi}}$ \hspace{0.5cm} $\boldsymbol{rate}$ \hspace{0.2cm} $\norm{\bdiv \boldsymbol{\varphi}_h}_{\infty}$ \hspace{0.15cm} $\boldsymbol{\chi}_{Q_h^0}$ \hspace{0.65cm} $\boldsymbol{e}_p$ \hspace{0.4cm} $\boldsymbol{rate}$ \hspace{0.6cm} $\boldsymbol{e}_{\xi}$ \hspace{0.5cm} $\boldsymbol{rate}$ \hspace{0.2cm} $\boldsymbol{It}$} \\ \hline 
		\hspace{0.3cm}66 		& 16.401  & \hspace{0.2cm}\;\;- 	& 8.8E-16 & \hspace{0.2cm}15.457   & \;\;- &\hspace{-0cm}5.5E-16             &   \hspace{0.25cm}18 &  87.600  & \;\;- & \hspace{-0cm}42.188 & \hspace{-0cm}\;\;-& \hspace{-0.2cm}4\\
		\hspace{0.3cm}240      & 7.3961   & \hspace{0.15cm}1.15   & 1.8E-14 & \hspace{0.2cm}7.6961 & 1.01    & \hspace{-0cm}1.3E-15        &    \hspace{0.25cm}72 & 56.757 & 0.63 & \hspace{-0cm}44.241 & \hspace{-0cm}-0.07 & \hspace{-0.2cm}4	\\
		\hspace{0.3cm}912       & 3.4361 & \hspace{0.15cm}1.11    & 3.5E-15 & \hspace{0.2cm}3.7037 & 1.06    & \hspace{-0cm}1.5E-15      &   \hspace{0.25cm}288 & 31.281 & 0.86 & \hspace{-0cm}29.079 & \hspace{-0cm}0.61	& \hspace{-0.2cm}4	\\
		\hspace{0.3cm}3552    & 1.6473  & \hspace{0.15cm}1.06   & 6.6E-15 & \hspace{0.2cm}1.7601 & 1.07    & \hspace{-0cm}3.4E-15      &   \hspace{0.25cm}1152 & 16.265 & 0.94 & \hspace{-0cm}15.941 & \hspace{-0cm}0.87	& \hspace{-0.2cm}4	\\
		\hspace{0.3cm}14016   & 0.8052 & \hspace{0.15cm}1.03   & 1.3E-14 & \hspace{0.2cm}0.8466 & 1.06       & \hspace{-0cm}8.8E-15  &   \hspace{0.25cm}4608 & 8.2737 & 0.98 & \hspace{-0cm}8.2298 & \hspace{-0cm}0.95	 & \hspace{-0.2cm}4	\\ 
		\hspace{0.3cm}55680  & 0.3978  & \hspace{0.15cm}1.02   & 2.7E-14 & \hspace{0.2cm}0.4133 & 1.03       & \hspace{-0cm}1.9E-14  &   \hspace{0.25cm}18432 & 4.1700 & 0.99 & \hspace{-0cm}4.1642 &	\hspace{-0cm}0.98	& \hspace{-0.2cm}4\\ 
		\hline
		\multicolumn{13}{l}{\hspace{0.25cm}$\boldsymbol{e}_T$  \hspace{0.6cm} $\boldsymbol{rate}$ \hspace{0.5cm} $\boldsymbol{e}_S$ \hspace{0.6cm} $\boldsymbol{rate}$ \hspace{0.9cm} $\boldsymbol{e}_{\eta^T}$ \hspace{0.4cm} $\boldsymbol{rate}$ \hspace{0.5cm} $\boldsymbol{e}_{\eta^S}$ \hspace{0.6cm} $\boldsymbol{rate}$ \hspace{0.4cm} $\boldsymbol{\chi}_{\boldsymbol{\mathcal{U}}_{h,ad}}$ \hspace{0.3cm} $\boldsymbol{e}_{\boldsymbol{U}_{\boldsymbol{1}}}$ \hspace{0.35cm} $\boldsymbol{rate}$ \hspace{0.45cm} $\boldsymbol{e}_{\boldsymbol{U}_{\boldsymbol{2}}}$ \hspace{0.3cm} $\boldsymbol{rate}$}  \\ \hline 
		\hspace{0.3cm}1.6757 & \hspace{0.1cm}\;\;-   & 2.8265 &\hspace{0.2cm}\;\;-   & \hspace{0.2cm}0.9276 & \hspace{0.1cm}\;\;-   &\hspace{0.1cm}1.3509 &\hspace{0.2cm}\;\;-	&18 &  0.0617   &\hspace{0.2cm}\;\;- &   0.0617 & \;\;- \\
		\hspace{0.3cm}0.8509 & \hspace{0.1cm}0.98 & 1.4152 & \hspace{0.2cm}1.0 & \hspace{0.2cm}0.4754 &  \hspace{0.1cm}0.96 & \hspace{0.1cm}0.7101 & \hspace{0.2cm}0.93 &72 &   0.0285 & \hspace{0.2cm}1.11 &  0.0285 &	1.11    \\
		\hspace{0.3cm}0.4281 & \hspace{0.1cm}0.99 & 0.7080 &\hspace{0.2cm}1.0 & \hspace{0.2cm}0.2392 &  \hspace{0.1cm}0.99  &	\hspace{0.1cm}0.3597 &\hspace{0.2cm}0.98&288 & 0.0146 & \hspace{0.2cm}0.96 & 0.0146 &	0.96\\
		\hspace{0.3cm}0.2144 & \hspace{0.1cm}1.0 & 0.2144 & \hspace{0.2cm}1.0 &\hspace{0.2cm}0.1197 &  \hspace{0.1cm}1.0 	&  \hspace{0.1cm}0.1804 & \hspace{0.2cm}1.0&1152 & 0.0074 &\hspace{0.2cm}0.98 &  0.0074 &0.98\\
		\hspace{0.3cm}0.1073 & \hspace{0.1cm}1.0 & 0.1770 & \hspace{0.2cm}1.0 & \hspace{0.2cm}0.0599 & \hspace{0.1cm}1.0  &\hspace{0.1cm}0.0903   & \hspace{0.2cm}1.0&4608 & 0.0037 &\hspace{0.2cm}1.01 & 0.0037 &	1.01\\ 
		\hspace{0.3cm}0.0536 & \hspace{0.1cm}1.0 & 0.0885 & \hspace{0.2cm}1.0 & \hspace{0.2cm}0.0299 & \hspace{0.1cm}1.0  &	\hspace{0.1cm}0.0452 &\hspace{0.2cm}1.0&18432 & 0.0019 &\hspace{0.2cm}1.0 &0.0019 & 1.0\\ 
		\hline
		\caption{Experimental errors and order of convergence under explicit (piecewise constant) discretization of control in state, adjoint and control variables in their natural norms under Flow, Stokes and Darcy regimes, respectively.}\label{table1}
	\end{longtable}
}


\begin{remark}[Jump stabilization in the Darcy regime]
In the strongly reaction-dominated Darcy regime, we consider a jump-stabilized variant of the discrete scheme
\[
a_h^{\mbox{stab}}(\bs{y}_h;\bs{u}_h,\bs{v}_h) := a_h(\bs{y}_h;\bs{u}_h,\bs{v}_h) + s_h(\bs{u}_h,\bs{v}_h); \; s_h(\bs{u}_h,\bs{v}_h) := \sum_{e \in \mathcal E(\mathcal T_h)} \int_e \frac{a_0\nu_2}{h_e} \jump{\bs{u}_h}:\jump{\bs{v}_h},ds, a_0=10\sqrt{\sigma}.
\]
The corresponding error is measured in the augmented natural norm
\[
\|\bs{v}\|_{1,\mathcal T_h^{nc},\mbox{st}}^2 := \sum_{K \in \mathcal T_h}
\left( \sigma \|\bs{v}\|_{0,K}^2 + \nu_2 \|\nabla \bs{v}\|_{0,K}^2 \right) +
\sum_{e \in \mathcal E(\mathcal T_h)} \frac{1}{h_e} \| \jump{\bs{v}} \|_{0,e}^2  \quad \forall \bs{v} \in \bs{H}^1(\mathcal T_h).
\]
The additional stabilization is coercive since
$
s_h(\bs{v}_h,\bs{v}_h) = a_0\nu_2
\sum_{e\in\mathcal E(\mathcal T_h)}
\frac{1}{h_e}\| \jump{\bs{v}_h} \|_{0,e}^2
\ge 0
$
Hence, using the coercivity of $a_h$,
\[
a_h^{\mbox{stab}}(\bs{y}_h;\bs{v}_h,\bs{v}_h) \geq \min\{\alpha_a^{nc},a_0\nu_2\}
\|\bs{v}_h\|_{1,\mathcal T_h^{nc},\mbox{st}}^2.
\]
Moreover, the stabilization is consistent. Indeed, for the exact velocity $\bs{u} \in H_0^1(\Omega)$, jump of the solution across interior edges is zero. Therefore
$
s_h(\bs{u}_h,\bs{v}_h)=0 \quad \forall,\bs{v}_h\in V_h^{cr_0}.
$
The flow and Stokes regimes directly validate the scheme analyzed in the preceding sections. In the strongly reaction-dominated Darcy regime, we additionally consider the above jump-stabilized variant and compare it with the original scheme in order to assess the effect of the stabilization. The stabilization is introduced here for the reaction-dominated regime and is not included in the theoretical analysis developed above. A detailed analysis of the corresponding stabilized scheme will be carried out in future work.	
\end{remark}

	{\footnotesize
	\begin{longtable}{@{}lllllllllllll@{}} \hline
		\multicolumn{13}{c}{\textbf{Flow regime} ($\nu_2 = 1, \; \sigma = 1$)}  \\ \hline
		\multicolumn{13}{l}{$\boldsymbol{\chi}_{\boldsymbol{V}_h^{cr}}$ \hspace{0.6cm} $\boldsymbol{e}_{\boldsymbol{u}}$  \hspace{0.6cm} $\boldsymbol{rate}$ \hspace{0.2cm} $\norm{\bdiv \boldsymbol{u}_h}_{\infty}$ \hspace{0.35cm} $\boldsymbol{e}_{\boldsymbol{\varphi}}$ \hspace{0.5cm} $\boldsymbol{rate}$ \hspace{0.2cm} $\norm{\bdiv \boldsymbol{\varphi}_h}_{\infty}$ \hspace{0.15cm} $\boldsymbol{\chi}_{Q_h^0}$ \hspace{0.65cm} $\boldsymbol{e}_p$ \hspace{0.4cm} $\boldsymbol{rate}$ \hspace{0.8cm} $\boldsymbol{e}_{\xi}$ \hspace{0.5cm} $\boldsymbol{rate}$ \hspace{0.4cm} $\boldsymbol{It}$} \\ \hline 
		\hspace{0.3cm}66 		& 2.0692  & \hspace{0.2cm}\;\;- & 8.8E-16 & \hspace{0.2cm}1.9019   & \;\;- &\hspace{-0cm}3.8E-16             &   \hspace{0.25cm}18 &  0.8092  & \;\;- & \hspace{-0cm}0.4169 & \hspace{-0cm}\;\;-& \hspace{-0.2cm}6\\
		\hspace{0.3cm}240      & 1.1332 & \hspace{0.15cm}0.87   & 1.2E-14 & \hspace{0.2cm}1.201 & 0.66    & \hspace{-0cm}6.6E-15        &    \hspace{0.25cm}72 & 0.5211 & 0.63 & \hspace{-0cm}0.3500 & \hspace{-0cm}0.25 & \hspace{-0.2cm}6	\\
		\hspace{0.3cm}912       & 0.5799 & \hspace{0.15cm}0.97   & 2.9E-15 & \hspace{0.2cm}0.6543 & 0.88    & \hspace{-0cm}1.5E-15      &   \hspace{0.25cm}288 & 0.2877 & 0.86 & \hspace{-0cm}0.2486 & \hspace{-0cm}0.49	& \hspace{-0.2cm}6	\\
		\hspace{0.3cm}3552    & 0.2916 & \hspace{0.15cm}0.99   & 1.1E-14 & \hspace{0.2cm}0.3323 & 0.98    & \hspace{-0cm}1.13E-14      &   \hspace{0.25cm}1152 & 0.1496 & 0.94 & \hspace{-0cm}0.1429 & \hspace{-0cm}0.8	& \hspace{-0.2cm}6	\\
		\hspace{0.3cm}14016   & 0.1461 & \hspace{0.15cm}1.0     & 1.4E-14 & \hspace{0.2cm}0.1662 & 1.0       & \hspace{-0cm}8.1E-15  &   \hspace{0.25cm}4608 & 0.0760 & 0.98 & \hspace{-0cm}0.0750 & \hspace{-0cm}0.93	 & \hspace{-0.2cm}6	\\ 
		\hspace{0.3cm}55680  & 0.0731 & \hspace{0.15cm}1.0     & 2.7E-14 & \hspace{0.2cm}0.0839 & 1.0       & \hspace{-0cm}2.0E-14  &   \hspace{0.25cm}18432 & 0.0382 & 0.99 & \hspace{-0cm}0.0382&	\hspace{-0cm}0.97	& \hspace{-0.2cm}6\\ 
		\hline
		\multicolumn{13}{l}{\hspace{0.25cm}$\boldsymbol{e}_T$  \hspace{0.6cm} $\boldsymbol{rate}$ \hspace{0.5cm} $\boldsymbol{e}_S$ \hspace{0.6cm} $\boldsymbol{rate}$ \hspace{0.9cm} $\boldsymbol{e}_{\eta^T}$ \hspace{0.4cm} $\boldsymbol{rate}$ \hspace{0.5cm} $\boldsymbol{e}_{\eta^S}$ \hspace{0.6cm} $\boldsymbol{rate}$ \hspace{0.3cm} $\boldsymbol{\chi}_{\boldsymbol{\mathcal{U}}_{h,ad}}$ \hspace{0.3cm} $\boldsymbol{e}_{\boldsymbol{U}_{\boldsymbol{1}}}$ \hspace{0.75cm} $\boldsymbol{rate}$ \hspace{0.45cm} $\boldsymbol{e}_{\boldsymbol{U}_{\boldsymbol{2}}}$ \hspace{0.4cm} $\boldsymbol{rate}$}  \\ \hline 
		\hspace{0.3cm}1.6757 & \hspace{0.1cm}\;\;-   & 2.8265 &\hspace{0.2cm}\;\;-   & \hspace{0.2cm}0.9258 & \hspace{0.1cm}\;\;-   &\hspace{0.1cm}1.3509 &\hspace{0.2cm}\;\;-& -	  &  0.04121  &\hspace{0.2cm}\;\;- &   0.04207 & \;\;- \\
		\hspace{0.3cm}0.8509 & \hspace{0.1cm}0.98 & 1.4152 & \hspace{0.2cm}1.0 & \hspace{0.2cm}0.4750 &  \hspace{0.1cm}0.96 & \hspace{0.1cm}0.7101 & \hspace{0.2cm}0.93 & -  &   0.0181 & \hspace{0.2cm}1.18 &  0.01845 &	1.19   \\
		\hspace{0.3cm}0.4281 & \hspace{0.1cm}0.99 & 0.7080 &\hspace{0.2cm}1.0 & \hspace{0.2cm}0.2391 &  \hspace{0.1cm}0.99  &	\hspace{0.1cm}0.3597 &\hspace{0.2cm}0.98 & -& 0.0060 & \hspace{0.2cm}1.58 & 0.0061 &	1.58\\
		\hspace{0.3cm}0.2144 & \hspace{0.1cm}1.0 & 0.2144 & \hspace{0.2cm}1.0 &\hspace{0.2cm}0.1197 &  \hspace{0.1cm}1.0 	&  \hspace{0.1cm}0.1804 & \hspace{0.2cm}1.0 & -   & 0.0017 &\hspace{0.2cm}1.83 &  0.0017 &1.83\\
		\hspace{0.3cm}0.1073 & \hspace{0.1cm}1.0 & 0.1770 & \hspace{0.2cm}1.0 & \hspace{0.2cm}0.0599 & \hspace{0.1cm}1.0  &\hspace{0.1cm}0.0903   & \hspace{0.2cm}1.0  & - & 4.94e-04 &\hspace{0.2cm}1.79 & 5.01e-04 &	1.79\\ 
		\hspace{0.3cm}0.0536 & \hspace{0.1cm}1.0 & 0.0885 & \hspace{0.2cm}1.0 & \hspace{0.2cm}0.0299 & \hspace{0.1cm}1.0  &	\hspace{0.1cm}0.0452 &\hspace{0.2cm}1.0   & - & 1.34e-04 &\hspace{0.2cm}1.87 &1.35e-04 & 1.89\\ 
		\hline
		\multicolumn{13}{c}{\textbf{Stokes regime} ($\nu_2 = 10, \; \sigma = 0$)}  \\ \hline
		\multicolumn{13}{l}{$\boldsymbol{\chi}_{\boldsymbol{V}_h^{cr}}$ \hspace{0.6cm} $\boldsymbol{e}_{\boldsymbol{u}}$  \hspace{0.6cm} $\boldsymbol{rate}$ \hspace{0.2cm} $\norm{\bdiv \boldsymbol{u}_h}_{\infty}$ \hspace{0.35cm} $\boldsymbol{e}_{\boldsymbol{\varphi}}$ \hspace{0.5cm} $\boldsymbol{rate}$ \hspace{0.2cm} $\norm{\bdiv \boldsymbol{\varphi}_h}_{\infty}$ \hspace{0.15cm} $\boldsymbol{\chi}_{Q_h^0}$ \hspace{0.65cm} $\boldsymbol{e}_p$ \hspace{0.5cm} $\boldsymbol{rate}$ \hspace{0.7cm} $\boldsymbol{e}_{\xi}$ \hspace{0.5cm} $\boldsymbol{rate}$ \hspace{0.5cm} $\boldsymbol{It}$} \\ \hline 
		\hspace{0.3cm}66 		& 3.624  & \hspace{0.2cm}\;\;- & 1.7E-15 & \hspace{0.2cm}4.5523   & \;\;- &\hspace{-0cm}6.6E-16             &   \hspace{0.25cm}18 &  0.8082  & \;\;- & \hspace{-0cm}1.2520   & \hspace{-0cm}\;\;-& \hspace{-0.2cm}5\\
		\hspace{0.3cm}240      & 1.8689 & \hspace{0.15cm}0.96   & 1.3E-15 & \hspace{0.2cm}2.4415 & 0.9    & \hspace{-0cm}1.8E-15        &    \hspace{0.25cm}72 & 0.4301 & 0.91 & \hspace{-0cm}0.7032     & \hspace{-0cm}0.83 & \hspace{-0.2cm}5	\\
		\hspace{0.3cm}912       & 0.9427 & \hspace{0.15cm}0.99   & 3.0E-15 & \hspace{0.2cm}1.2424 & 0.97    & \hspace{-0cm}2.2E-15      &   \hspace{0.25cm}288 & 0.2159 & 0.99 & \hspace{-0cm}0.3409 & \hspace{-0cm}1.04	& \hspace{-0.2cm}5	\\
		\hspace{0.3cm}3552    & 0.4726 & \hspace{0.15cm}1.0   & 6.6E-15 & \hspace{0.2cm}0.6234 & 0.99    & \hspace{-0cm}2.2E-15      &   \hspace{0.25cm}1152 & 0.1077 & 1.0 & \hspace{-0cm}0.1670    & \hspace{-0cm}1.03	& \hspace{-0.2cm}5	\\
		\hspace{0.3cm}14016   & 0.2365 & \hspace{0.15cm}1.0     & 1.3E-14 & \hspace{0.2cm}0.3118 & 1.0       & \hspace{-0cm}9.1E-15  &   \hspace{0.25cm}4608  & 0.0538 & 1.0 & \hspace{-0cm}0.0829   & \hspace{-0cm}1.01	 & \hspace{-0.2cm}5	\\ 
		\hspace{0.3cm}55680  & 0.1182 & \hspace{0.15cm}1.0     & 2.9E-14 & \hspace{0.2cm}0.1559 & 1.0       & \hspace{-0cm}1.8E-14  &   \hspace{0.25cm}18432 & 0.0269 & 1.0 & \hspace{-0cm}0.0414    &	\hspace{-0cm}1.0	& \hspace{-0.2cm}5\\ 
		\hline
		\multicolumn{13}{l}{\hspace{0.25cm}$\boldsymbol{e}_T$  \hspace{0.6cm} $\boldsymbol{rate}$ \hspace{0.5cm} $\boldsymbol{e}_S$ \hspace{0.6cm} $\boldsymbol{rate}$ \hspace{0.9cm} $\boldsymbol{e}_{\eta^T}$ \hspace{0.4cm} $\boldsymbol{rate}$ \hspace{0.5cm} $\boldsymbol{e}_{\eta^S}$ \hspace{0.6cm} $\boldsymbol{rate}$ \hspace{0.4cm} $\boldsymbol{\chi}_{\boldsymbol{\mathcal{U}}_{h,ad}}$ \hspace{0.3cm} $\boldsymbol{e}_{\boldsymbol{U}_{\boldsymbol{1}}}$ \hspace{0.75cm} $\boldsymbol{rate}$ \hspace{0.45cm} $\boldsymbol{e}_{\boldsymbol{U}_{\boldsymbol{2}}}$ \hspace{0.3cm} $\boldsymbol{rate}$}  \\ \hline 
		\hspace{0.3cm}1.6757 & \hspace{0.1cm}\;\;-   & 2.8265 &\hspace{0.2cm}\;\;-   & \hspace{0.2cm}0.9276 & \hspace{0.1cm}\;\;-   &\hspace{0.1cm}1.3509 &\hspace{0.2cm}\;\;-	&- &  0.03929   &\hspace{0.2cm}\;\;-     &   0.03874 & \;\;- \\
		\hspace{0.3cm}0.8509 & \hspace{0.1cm}0.98 & 1.4152 & \hspace{0.2cm}1.0 & \hspace{0.2cm}0.4754 &  \hspace{0.1cm}0.96 & \hspace{0.1cm}0.7101 & \hspace{0.2cm}0.93 &- &   0.01558 & \hspace{0.2cm}1.33 &  0.01549  & 1.32    \\
		\hspace{0.3cm}0.4281 & \hspace{0.1cm}0.99 & 0.7080 &\hspace{0.2cm}1.0 & \hspace{0.2cm}0.2392 &  \hspace{0.1cm}0.99  &	\hspace{0.1cm}0.3597 &\hspace{0.2cm}0.98&- & 0.00414 & \hspace{0.2cm}1.91 & 0.00409   &	1.92\\
		\hspace{0.3cm}0.2144 & \hspace{0.1cm}1.0 & 0.3541 & \hspace{0.2cm}1.0 &\hspace{0.2cm}0.1197 &  \hspace{0.1cm}1.0 	&  \hspace{0.1cm}0.1804 & \hspace{0.2cm}1.0&- & 0.00108 &\hspace{0.2cm}1.94          &  0.00106  & 1.94\\
		\hspace{0.3cm}0.1073 & \hspace{0.1cm}1.0 & 0.1770 & \hspace{0.2cm}1.0 & \hspace{0.2cm}0.0599 & \hspace{0.1cm}1.0  &\hspace{0.1cm}0.0903   & \hspace{0.2cm}1.0&- & 0.00033 &\hspace{0.2cm}1.69          & 0.00033   & 1.69\\ 
		\hspace{0.3cm}0.0536 & \hspace{0.1cm}1.0 & 0.0885 & \hspace{0.2cm}1.0 & \hspace{0.2cm}0.0299 & \hspace{0.1cm}1.0  &	\hspace{0.1cm}0.0452 &\hspace{0.2cm}1.0&- & 9.25e-05 &\hspace{0.2cm}1.86         &9.16e-05    & 1.86\\ 
		\hline
		\multicolumn{13}{c}{\textbf{Darcy regime} ($\nu_2 = 1, \; \sigma = 10000$)}  \\ \hline
		\multicolumn{13}{l}{$\boldsymbol{\chi}_{\boldsymbol{V}_h^{cr}}$ \hspace{0.6cm} $\boldsymbol{e}_{\boldsymbol{u}}$  \hspace{0.6cm} $\boldsymbol{rate}$ \hspace{0.2cm} $\norm{\bdiv \boldsymbol{u}_h}_{\infty}$ \hspace{0.35cm} $\boldsymbol{e}_{\boldsymbol{\varphi}}$ \hspace{0.5cm} $\boldsymbol{rate}$ \hspace{0.2cm} $\norm{\bdiv \boldsymbol{\varphi}_h}_{\infty}$ \hspace{0.15cm} $\boldsymbol{\chi}_{Q_h^0}$ \hspace{0.65cm} $\boldsymbol{e}_p$ \hspace{0.4cm} $\boldsymbol{rate}$ \hspace{1.0cm} $\boldsymbol{e}_{\xi}$ \hspace{0.5cm} $\boldsymbol{rate}$ \hspace{0.2cm} $\boldsymbol{It}$} \\ \hline 
		\hspace{0.3cm}66 		& 16.401  & \hspace{0.2cm}\;\;- 	& 7.6E-16 & \hspace{0.2cm}15.457   & \;\;- &\hspace{-0cm}6.3E-16             &   \hspace{0.25cm}18 &  87.600  & \;\;- & \hspace{-0cm}42.188 & \hspace{-0cm}\;\;-& \hspace{-0.2cm}4\\
		\hspace{0.3cm}240      & 7.3961   & \hspace{0.15cm}1.15   & 1.3E-14 & \hspace{0.2cm}7.6961 & 1.01    & \hspace{-0cm}1.0E-15        &    \hspace{0.25cm}72 & 56.757 & 0.63 & \hspace{-0cm}44.241 & \hspace{-0cm}-0.07 & \hspace{-0.2cm}4	\\
		\hspace{0.3cm}912       & 3.4361 & \hspace{0.15cm}1.11    & 6.6E-15 & \hspace{0.2cm}3.7037 & 1.06    & \hspace{-0cm}3.5E-15      &   \hspace{0.25cm}288 & 31.281 & 0.86 & \hspace{-0cm}29.079 & \hspace{-0cm}0.61	& \hspace{-0.2cm}4	\\
		\hspace{0.3cm}3552    & 1.6473  & \hspace{0.15cm}1.06   & 6.6E-15 & \hspace{0.2cm}1.7601 & 1.07    & \hspace{-0cm}4.2E-15      &   \hspace{0.25cm}1152 & 16.265 & 0.94 & \hspace{-0cm}15.941 & \hspace{-0cm}0.87	& \hspace{-0.2cm}4	\\
		\hspace{0.3cm}14016   & 0.8052 & \hspace{0.15cm}1.03   & 1.5E-14 & \hspace{0.2cm}0.8466 & 1.06       & \hspace{-0cm}7.9E-15  &   \hspace{0.25cm}4608 & 8.2737 & 0.98 & \hspace{-0cm}8.2298 & \hspace{-0cm}0.95	 & \hspace{-0.2cm}4	\\ 
		\hspace{0.3cm}55680  & 0.3978  & \hspace{0.15cm}1.02   & 2.8E-14 & \hspace{0.2cm}0.4133 & 1.03       & \hspace{-0cm}1.7E-14  &   \hspace{0.25cm}18432 & 4.1700 & 0.99 & \hspace{-0cm}4.1642 &	\hspace{-0cm}0.98	& \hspace{-0.2cm}4\\ 
		\hline
		\multicolumn{13}{l}{\hspace{0.25cm}$\boldsymbol{e}_T$  \hspace{0.6cm} $\boldsymbol{rate}$ \hspace{0.5cm} $\boldsymbol{e}_S$ \hspace{0.6cm} $\boldsymbol{rate}$ \hspace{0.9cm} $\boldsymbol{e}_{\eta^T}$ \hspace{0.4cm} $\boldsymbol{rate}$ \hspace{0.5cm} $\boldsymbol{e}_{\eta^S}$ \hspace{0.6cm} $\boldsymbol{rate}$ \hspace{0.4cm} $\boldsymbol{\chi}_{\boldsymbol{\mathcal{U}}_{h,ad}}$ \hspace{0.3cm} $\boldsymbol{e}_{\boldsymbol{U}_{\boldsymbol{1}}}$ \hspace{0.75cm} $\boldsymbol{rate}$ \hspace{0.45cm} $\boldsymbol{e}_{\boldsymbol{U}_{\boldsymbol{2}}}$ \hspace{0.3cm} $\boldsymbol{rate}$}  \\ \hline 
		\hspace{0.3cm}1.6757 & \hspace{0.1cm}\;\;-   & 2.8265 &\hspace{0.2cm}\;\;-   & \hspace{0.2cm}0.9276 & \hspace{0.1cm}\;\;-   &\hspace{0.1cm}1.3509 &\hspace{0.2cm}\;\;-	&- &  0.04075  &\hspace{0.2cm}\;\;- &   0.04074& \;\;- \\
		\hspace{0.3cm}0.8509 & \hspace{0.1cm}0.98 & 1.4152 & \hspace{0.2cm}1.0 & \hspace{0.2cm}0.4754 &  \hspace{0.1cm}0.96 & \hspace{0.1cm}0.7101 & \hspace{0.2cm}0.93 &- &   0.01583 & \hspace{0.2cm}1.36 &  0.01582 &	1.36    \\
		\hspace{0.3cm}0.4281 & \hspace{0.1cm}0.99 & 0.7080 &\hspace{0.2cm}1.0 & \hspace{0.2cm}0.2392 &  \hspace{0.1cm}0.99  &	\hspace{0.1cm}0.3597 &\hspace{0.2cm}0.98&- & 0.00451 & \hspace{0.2cm}1.81 & 0.00451 &	1.81\\
		\hspace{0.3cm}0.2144 & \hspace{0.1cm}1.0 & 0.2144 & \hspace{0.2cm}1.0 &\hspace{0.2cm}0.1197 &  \hspace{0.1cm}1.0 	&  \hspace{0.1cm}0.1804 & \hspace{0.2cm}1.0&- & 0.00121 &\hspace{0.2cm}1.89 &  0.00121 &1.89\\
		\hspace{0.3cm}0.1073 & \hspace{0.1cm}1.0 & 0.1770 & \hspace{0.2cm}1.0 & \hspace{0.2cm}0.0599 & \hspace{0.1cm}1.0  &\hspace{0.1cm}0.0903   & \hspace{0.2cm}1.0&- & 0.00037 &\hspace{0.2cm}1.7 & 0.00037 &	1.7\\ 
		\hspace{0.3cm}0.0536 & \hspace{0.1cm}1.0 & 0.0885 & \hspace{0.2cm}1.0 & \hspace{0.2cm}0.0299 & \hspace{0.1cm}1.0  &	\hspace{0.1cm}0.0452 &\hspace{0.2cm}1.0&- & 0.00010 &\hspace{0.2cm}1.87 &0.00010 & 1.87\\ 
		\hline
		\caption{Experimental errors and order of convergence under implicit (variational) discretization of control in state, adjoint and control variables in their natural norms under Flow, Stokes and Darcy regimes, respectively.}\label{table2}
	\end{longtable}
}

\begin{remark}\label{variatonal_part2}
	The minimizers obtained from the explicit and implicit control discretization approaches are not the same,  infact it has been shown in \cite{Hinze} for a convex optimal control problem that the implicit approach is more efficient in resolving the active sets due to the decoupling of discrete active sets and the finite element grid. We also see this phenomena in our numerical experiments (see Table \ref{table2}), by observing the errors of the exact control with the discretized control, which are lower for the implicit case. We also refer to Figure \ref{fig:active_sets_48}, which visually show the upper and lower active sets $(\mathcal{A}_a \; \mbox{and} \; \mathcal{A}_b)$, and inactive sets $(\mathcal{I})$ for the discrete controls obtained from the fully discrete scheme using piecewise constant controls and the variational discretization of the control on the mesh with $N=48$.  
\end{remark}

\begin{figure}[htb!]
    \centering
    \includegraphics[width=0.24\linewidth]{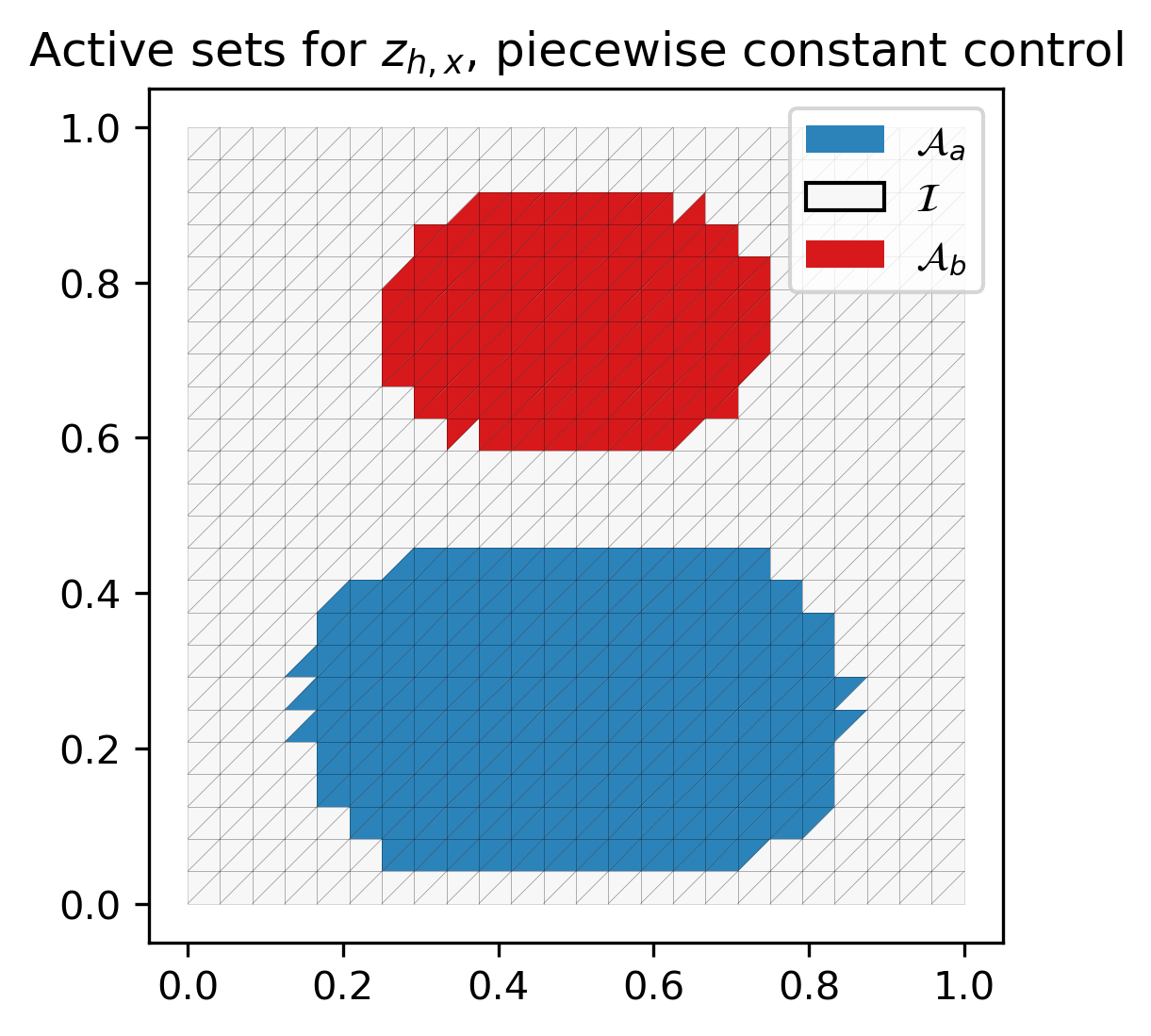}\hfill%
    \includegraphics[width=0.24\linewidth]{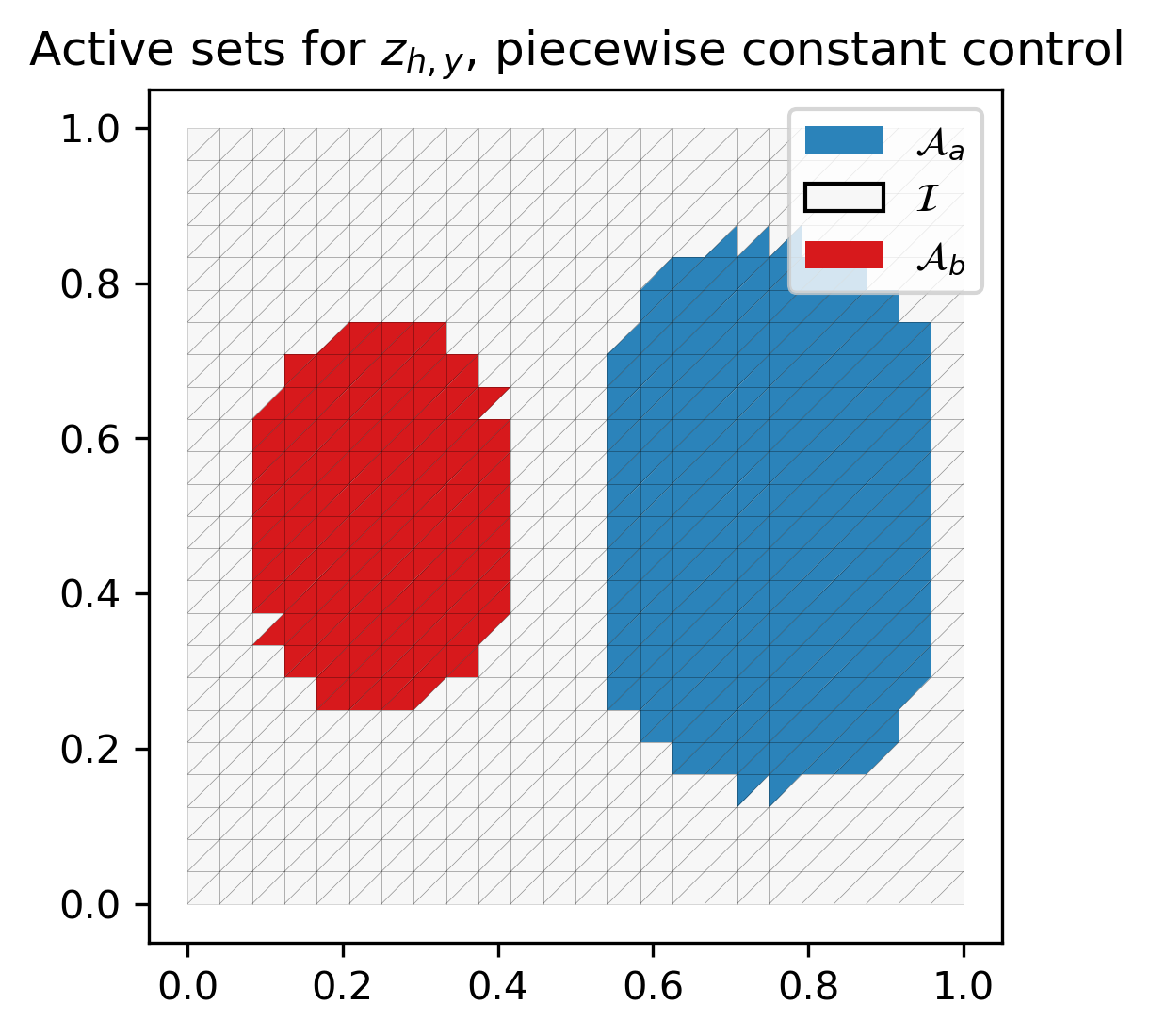}\hfill%
    \includegraphics[width=0.24\linewidth]{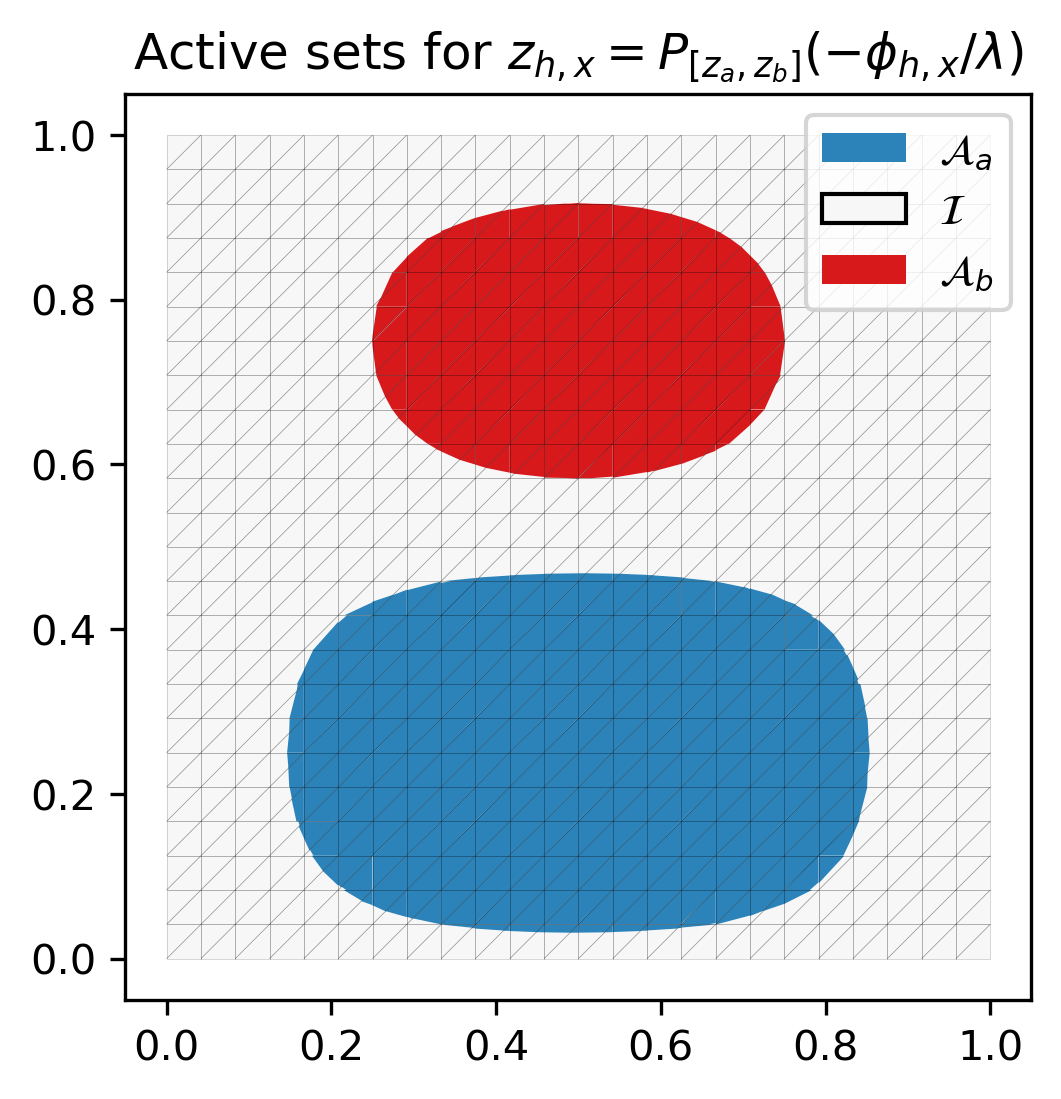}\hfill%
    \includegraphics[width=0.24\linewidth]{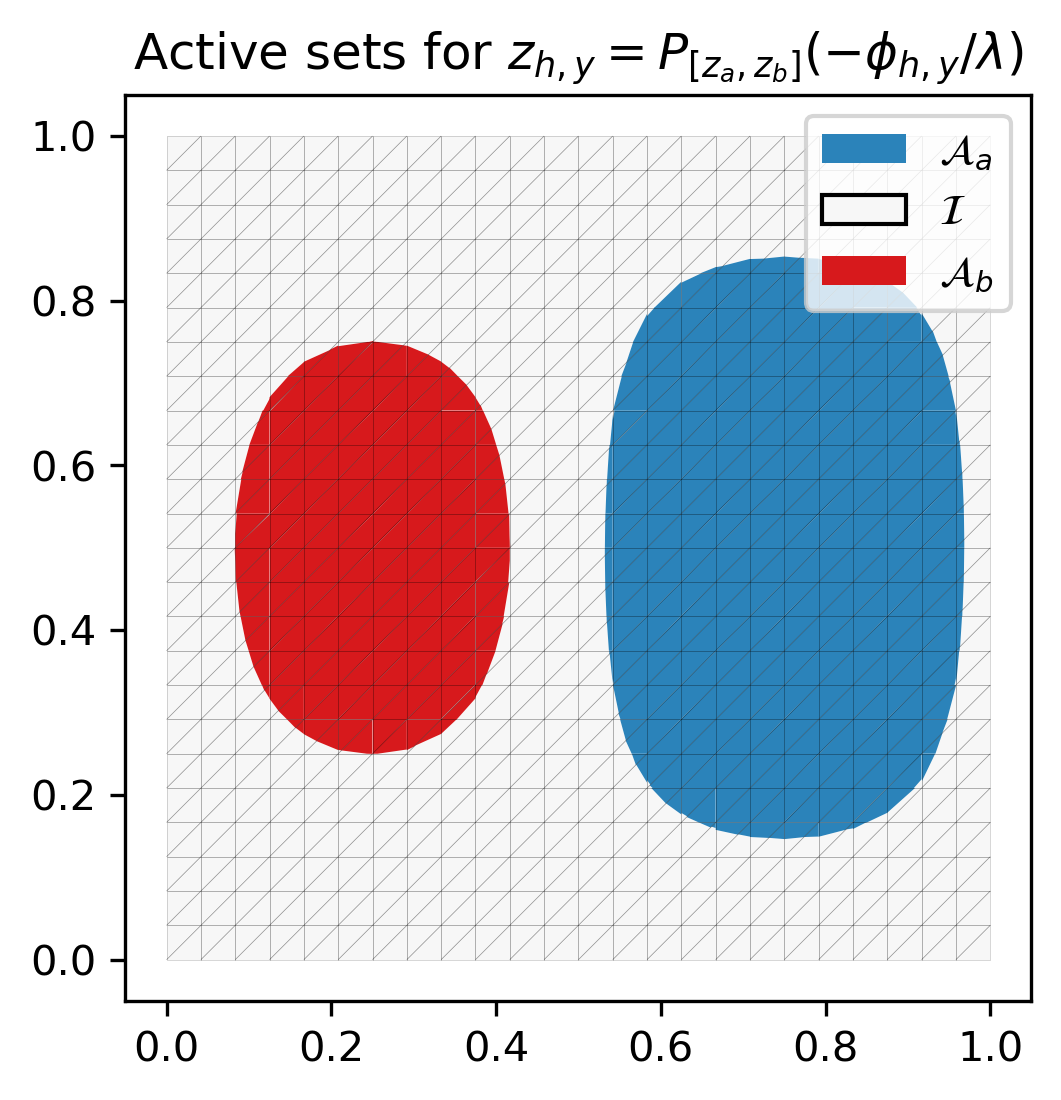}
    \caption{Discrete active and inactive sets on the mesh with $N=48$.
    From left to right: the components $z_{h,x}$ and $z_{h,y}$ obtained
    with piecewise constant controls, followed by the components $z_{h,x}$ and $z_{h,y}$ obtained with
    variational discretization.}
    \label{fig:active_sets_48}
\end{figure}

\subsection{Optimal control of Soret and Dufour effects in a porous cavity} 

Using the dimensionless variables described in \cite[Section 5.2]{burger2019h}, we can write the equations describing transport phenomena in a square porous cavity with thermal $(T)$ and concentration $(C)$ diffusion in the form \eqref{P:GE} and it's optimal control will be described by \eqref{P:CF}-\eqref{P:GE}. Set $\boldsymbol{K} = Da \boldsymbol{I}, \nu(T) = 1$ and $\boldsymbol{F}(\boldsymbol{y}) = (G_{r_{T}} T + G_{r_{C}} C) \boldsymbol{g},$ where $\boldsymbol{g} = (0,-1)^{\top}$ points in the direction of gravity with $\boldsymbol{y} = (T, C)^{\top}$. The diffusion coefficient $\boldsymbol{D}$ is given by,
$\boldsymbol{D} = \left[R_k / Pr \;\; Du; \;\; Sr \;\; 1/Sc\right],$ where $R_k$ is the thermal conductivity ratio, $G_{r_{T}}, G_{r_{C}}$ are the thermal and solutal Grashof numbers respectively, $Da = \kappa / H^2$ is the Darcy number, $Pr = \nu / \alpha$ the Prandtl number, $Sc = \nu / D_C$ the Schmidt number, and the ratio $Le = Sc / Pr$ is the Lewis number.  We perform a computational test to show that \eqref{P:CF}-\eqref{P:GE} steers the velocity $(\boldsymbol{u})$, Temperature $(T)$ and Concentration $(C)$ to the desired states $\boldsymbol{u}_d, T_d$ and $C_d$, respectively. The computational domain, $\Omega$ is the unit square with no-slip velocity conditions on $\Gamma$. Temperature and concentration  are kept at $T_0, C_0$ and $T_1, C_1$ at the right and left walls, respectively, $T_0 < T_1$ and $C_0 < C_1$. The horizontal walls are adiabatic and impermeable. The corresponding adjoint equation also has no-slip velocity conditions on $\Gamma$ and the adjoint  temperature and concentration have no-slip boundary conditions on the right and left walls with adiabatic and impermeable horizontal walls. The mesh used to discretize the computational domain has $100 \times 100$ elements and the relative tolerance used for the termination of the \textit{semi-smooth Newton} algorithm is $10^{-6}.$ For the values $\lambda = 10^{-4}, R_k = 1.0, Le = 2.0, Pr = 5.0, Da = 1.0, N = 0, T_0 = -1, C_0 = -1, T_1 = 1, C_1 = 1$, we set the Rayleigh value $Ra = G_{r_T} Pr Da$ as $250$ and the Soret and Dufour parameters as $Sr = 0.1$ and $Du = 0.1$, respectivley. The optimal solution profiles steering towards the desired states $\boldsymbol{u}_d = \boldsymbol{0}, T_d = 1 - 2 x, C_d = 1 - 2 x$ with $U_{a_j} = -5, U_{b_j} = 5$ under piecewise constant (12 Newton iterations) and variational discretization (10 Newton iterations) of control are reported along with their uncontrolled profiles are reported in Figures \ref{Fig:SoDu_velocity} and \ref{Fig:SoDu_temp_conc}. The corresponding active and inactive sets are reported in Figure \ref{Fig:ActInact_SoDu}.


\begin{figure}[htb!]
    \centering
    \begin{minipage}{0.8\linewidth}
        \centering
        \includegraphics[width=0.32\linewidth]{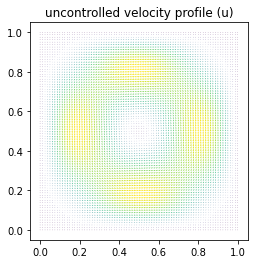}\hfill
        \includegraphics[width=0.32\linewidth]{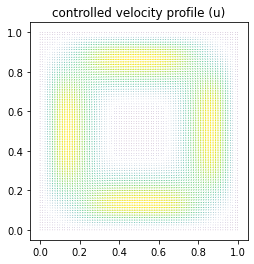}\hfill
        \includegraphics[width=0.32\linewidth]{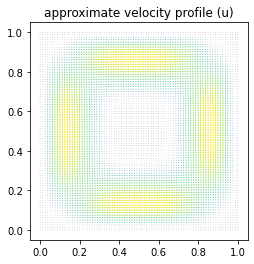}
    \end{minipage}
    \caption{Velocity profiles for the uncontrolled system, the controlled system using piecewise-constant control discretization, and the controlled system using variational control discretization, respectively. Compared to the uncontrolled system, the controlled systems show a large low-velocity region in the interior of the porous cavity.}
    \label{Fig:SoDu_velocity}
\end{figure}



\begin{figure}[htb!]
    \centering
    \begin{minipage}{0.8\linewidth}
        \centering
        \includegraphics[width=0.32\linewidth]{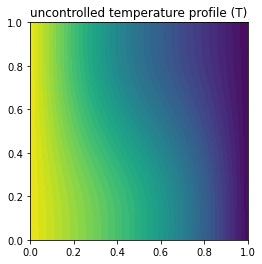}\hfill%
        \includegraphics[width=0.32\linewidth]{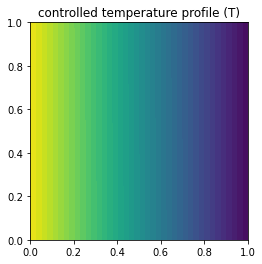}\hfill%
        \includegraphics[width=0.32\linewidth]{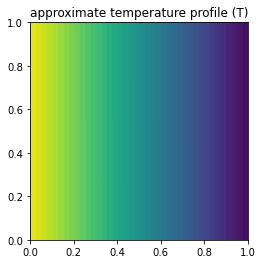}

        \par\medskip

        \includegraphics[width=0.32\linewidth]{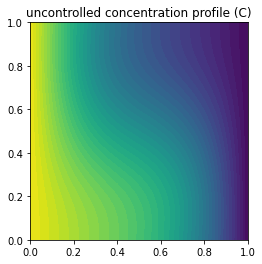}\hfill%
        \includegraphics[width=0.32\linewidth]{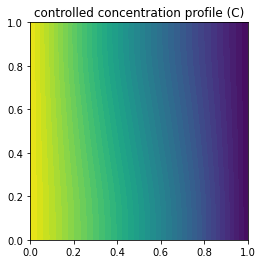}\hfill%
        \includegraphics[width=0.32\linewidth]{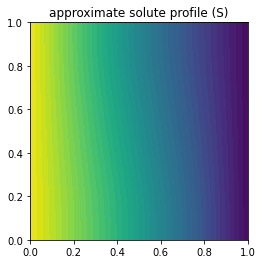}
    \end{minipage}
    \caption{Temperature profiles (top row) and solute concentration
    profiles (bottom row). The columns correspond, from left to right,
    to the uncontrolled system, the controlled system using
    piecewise-constant control discretization, and the controlled
    system using variational control discretization. The uncontrolled
    temperature and concentration profiles are distorted by the fluid
    flow, while the controlled profiles are steered towards the desired
    temperature and concentration, respectively.}
    \label{Fig:SoDu_temp_conc}
\end{figure}

\begin{figure}[htb!]
    \centering
    \includegraphics[width=0.24\linewidth]{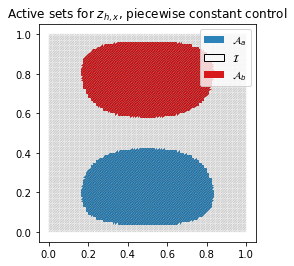}\hfill%
    \includegraphics[width=0.24\linewidth]{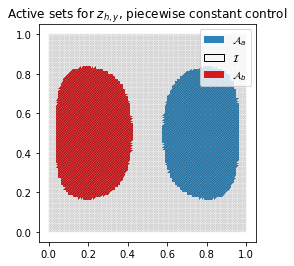}\hfill%
    \includegraphics[width=0.24\linewidth]{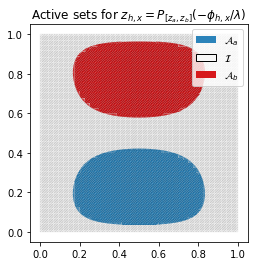}\hfill%
    \includegraphics[width=0.24\linewidth]{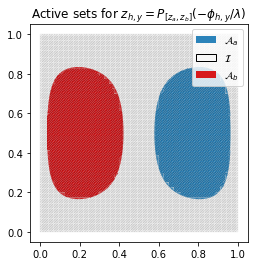}
    \caption{Discrete active and inactive sets for the controlled system.
    From left to right: the components $z_{h,x}$ and $z_{h,y}$ obtained
    with  using piecewise constant controls, followed by the components $z_{h,x}$ and $z_{h,y}$ obtained with variational discretization.}
    \label{Fig:ActInact_SoDu}
\end{figure}



\renewcommand*{\arraystretch}{1.2}

\begin{table}[htb!]
\centering
\caption{Table of some constants appearing in the analysis.}
\label{tab:constants}

\begingroup
\small
\renewcommand{\arraystretch}{1.15}
\setlength{\tabcolsep}{4pt}

\begin{tabularx}{\textwidth}{|c|X|c|X|}
\hline
\textbf{Symbol} & \textbf{Description}
& \textbf{Symbol} & \textbf{Description} \\
\hline

$C$
& Generic positive constant, independent of the mesh size.
& $M_{\bs{y}}$
& Energy estimate bound on $\|\bs{y}\|_{1,\Omega}$. \\
\hline

$\gamma_\nu$
& Lipschitz constant of the viscosity $\nu(\cdot)$.
& $M$
& Regularity bound for
$\|\bs{u}\|_{H^{3/2+\delta}(\Omega)}
+\|\bs{y}\|_{[H^{3/2+\delta}(\Omega)]^2}$
(see \cite[Theorem 3.9]{TKM_ncDDF}). \\
\hline

$\nu_1,\nu_2$
& Positive viscosity bounds: $\nu_1\le\nu(T)\le\nu_2$.
& $\widehat\alpha_a$
& Constant in the uniqueness proof for the state equation
(see \cite[3.48]{TKM_ncDDF}). \\
\hline

$\gamma_F$
& Lipschitz constant of the buoyancy map $\bs{F}(\cdot)$.
& $\sigma$
& $\sigma:=\|K^{-1}\|_{\infty,\Omega}$, used in the
discrete norm for $\bs{u}_h$. \\
\hline

$C_F$
& Growth constant:
$|\bs{F}(\bs{y})|\le C_F|\bs{y}|$.
& $\bar\sigma$
& $\bar\sigma:=\|D\|_{\infty,\Omega}$, used in the
discrete norm for $\bs{y}_h$. \\
\hline

$\alpha_1$
& Uniform positive-definiteness constant of
$\bs{K}^{-1}$.
& $C_T$
& Inverse trace inequality constant:
$\|v\|_{0,\partial K}\le C_T h_K^{-1/2}\|v\|_{0,K}$. \\
\hline

$\alpha_2$
& Positive-definiteness constant of $\bs{D}$.
& $C_a^{\mathrm{nc}}$
& Continuity constant for $a_h(\cdot;\cdot,\cdot)$. \\
\hline

$C_a$
& Continuity constant for $a(\bs{y};\bs{u},\bs{v})$.
& $\widehat C_a^{\mathrm{nc}}$
& Continuity constant for $a_h^y(\cdot,\cdot)$. \\
\hline

$\widehat C_a$
& Continuity constant for $a_{\bs{y}}(\bs{y},\bs{s})$.
& $C_e^{\mathrm{nc}}$
& Mesh-independent constant in discrete Sobolev, trace,
inverse, and jump estimates. Its parameter dependence
is specified in Lemma \ref{ncStateCty}. \\
\hline

$C_{r'}^d$
& Sobolev embedding constant:
$\|w\|_{L^{r'}(\Omega)}\le C_{r'}^d\|w\|_{1,\Omega}$.
& $\alpha_a^{\mathrm{nc}}$
& Coercivity constant for $a_h(\cdot;\cdot,\cdot)$. \\
\hline

$C_{\mathrm{gn}}$
& Constant in the fractional Gagliardo--Nirenberg inequality.
& $\alpha_2^{\mathrm{nc}}$
& Coercivity constant for $a_h^y(\cdot,\cdot)$. \\
\hline

$\alpha_a$
& Coercivity constant of $a(\cdot;\bs{v},\bs{v})$.
& $\widetilde\beta$
& Discrete inf-sup constant for
$(\bs{V}_{h,0}^{\mathrm{cr}},Q_h^0)$. \\
\hline

$C_P$
& Poincar\'e constant.
& $M_u^{\mathrm{nc}}$
& Bound for $\|\bs{u}_h\|_{1,\mathcal T_h^{\mathrm{nc}}}$. \\
\hline

$C_u,\;C_y$
& Constants in energy estimates for $\bs{u}$ and $\bs{y}$
(see \cite[Lemma 3.4]{TKM_ncDDF}).
& $M_y^{\mathrm{nc}}$
& Bound for $\|\bs{y}_h\|_{1,\mathcal T_h^{\mathrm{nc}}}$. \\
\hline

$M_{\bs{u}}$
& Energy estimate bound on $\|\bs{u}\|_{1,\Omega}$.
& $\widetilde M$
& Assumed bound on $\|\nabla_h u_h\|_{L^3(\Omega)}$
in the uniqueness proof for the discrete state equation. \\
\hline

\end{tabularx}
\endgroup
\end{table}


\section*{Acknowledgements}
The authors greatly acknowledge the funding from SERB-CRG India (Grant Number: CRG/2021/002569). The first author also gratefully thanks Indian Institute of Technology Roorkee, where majority of this works was done.

%
%
%

\addcontentsline{toc}{section}{References}
\bibliographystyle{ieeetr} 
\bibliography{References}
.\end{document}